\pdfoutput=1
\documentclass[reqno]{amsart}
\usepackage{amsmath}
\usepackage{amsthm}
\usepackage{amssymb}

\usepackage[left=1.25in, right=1.25in]{geometry}
\usepackage{tikz}[2010/10/13]
\usepackage[all]{xy}

\usepackage{indentfirst}
\usepackage{bm}
\usepackage{mathrsfs}

\usepackage{latexsym, subfigure, float, here, url}

\usepackage{hyperref}
\usepackage{bookmark}
\usepackage{microtype}
\usepackage{graphicx}

\newcommand{\gra}[1]{\raisebox{-.4cm}{\includegraphics[height=1cm]{UP#1.pdf}}}
\newcommand{\graa}[1]{\raisebox{-.6cm}{\includegraphics[height=1.5cm]{UP#1.pdf}}}
\newcommand{\grb}[1]{\raisebox{-.8cm}{\includegraphics[height=2cm]{UP#1.pdf}}}

\providecommand{\etalchar}[1]{$^{#1}$}
\newcommand{\Pl}{\mathscr{P}}

\begin{document}
	
\title{Classification of Thurston-relation subfactor planar algebras}
\author{Corey Jones, Zhengwei Liu, Yunxiang Ren}
\maketitle

\begin{abstract}
	
	 Bisch and Jones suggested the skein theoretic classification of planar algebras and investigated the ones generated by 2-boxes with the second author. In this paper,  we consider 3-box generators and classify subfactor planar algebras generated by a non-trivial 3-box satisfying a relation proposed by Thurston. The subfactor planar algebras in the classification are either $E^6$ or the ones from representations of quantum $SU(N)$. We introduce a new method to determine positivity of planar algebras and new techniques to reduce the complexity of computations.
\end{abstract}

\newtheorem{Lemma}{Lemma}
\theoremstyle{plain}
\newtheorem{theorem}{Theorem~}[section]
\newtheorem{main}{Main Theorem~}
\newtheorem{lemma}[theorem]{Lemma~}
\newtheorem{assumption}[theorem]{Assumption~}
\newtheorem{proposition}[theorem]{Proposition~}
\newtheorem{corollary}[theorem]{Corollary~}
\newtheorem{definition}[theorem]{Definition~}
\newtheorem{defi}[theorem]{Definition~}
\newtheorem{notation}[theorem]{Notation~}
\newtheorem{example}[theorem]{Example~}
\newtheorem*{remark}{Remark~}
\newtheorem*{cor}{Corollary~}
\newtheorem*{question}{Question}
\newtheorem*{claim}{Claim}
\newtheorem*{conjecture}{Conjecture~}
\newtheorem*{fact}{Fact~}
\renewcommand{\proofname}{\bf Proof}

\section{introduction}

Jones intiated the modern theory of subfactors to study quantum symmetry \cite{Jon83}. A deep theorem of Popa says the standard invariants completely classify strongly amenable subfactors \cite{Pop94}. In particular, the $A$, $D$, $E$ classification of subfactors up to index 4 is a quantum analogy of Mackey correspondence.
Moreover, Popa introduced standard $\lambda$-lattices as an axiomatization of the standard invariant \cite{Pop95}, which completes Ocneanu's axiomatization for finite depth subfactors \cite{Ocn88}.

Jones introduced (subfactor) planar algebras as an axiomatization of the standard invariant of subfactors, which capture its ulterior topological properties.
He suggested studying planar algebras by skein theory, a presentation theory which allows one to completely describe the entire planar algebra in terms of generators and relations, both algebraic and topological.

An $m$-box generator for a planar algebra is usually represented by a $2m$-valent labelled vertex. A planar algebra is a representation of labelled planar diagrams, namely {\it fully labelled planar tangles}, on vector spaces in the flavor of TQFT \cite{Ati88}:
the representation is well defined up to isotopy;
the target vector space only depends on the boundary condition of diagrams;
in particular, diagrams without boundary are mapped to the ground field, called the partition function.
A planar algebra is called a subfactor planar algebra if its partition function is positive definite. In this case, the vector spaces become Hilbert spaces.

A skein theory is a presentation of a planar algebra by generators and relations, such that the partition function can be evaluated by the relations.
The planar algebra is determined by these generators and relations once an evaluation algorithm is obtained. Therefore, one can ask for a skein theoretic classification of planar algebras, as suggested by Bisch and Jones.

One can find interesting skein theories for planar diagrams in graph theory, such as the exchange relation \cite{Lan02} motivated by biprojections from intermediate subfactors \cite{Bis94}, and the relations coming from the discharge method \cite{MPS15}.  Other types of skein theory were introduced, including Thurston's relation \cite{Thu04}, the Jellyfish relation \cite{BMPS} and the Yang-Baxter relation \cite{LiuYB}, motivated by knot theory and suggested by Bisch and Jones \cite{BisJon00,BisJon03}.

The simplest planar algebra is the Temperley-Lieb-Jones planar algebra which has neither generators nor relations.
Planar algebras generated by 1-boxes were completely analyzed by Jones \cite{JonPA}.
Bisch and Jones initiated the classification of planar algebras generated by a single 2-box \cite{BisJon00}. Based on the subsequent work of Bisch, Jones and the second author \cite{BisJon03,BJL}, a classification of singly generated Yang-Baxter relation planar algebra was achieved in \cite{LiuYB}, where a new family of planar algebras was constructed. Planar algebras with multiple 2-box generators were discussed in \cite{Liuex}.

In this paper, we study planar algebras generated by a single 3-box.
There is a known two parameter family of examples $\mathscr{P}_\bullet^{H}(q,r)$ related to quantum $SU(N)$, see Example 2.5 in \cite{JonPA}.
We classify all $q,r$ for which $\mathscr{P}_\bullet^{H}(q,r)$ has a positive partition function,  (See Theorem \ref{Thm:positivity}).  Its semisimple quotient is a subfactor planar algebra. The corresponding subfactors are known as Jones-Wenzl subfactors \cite{Jon87,Wen88}. The subfactor planar algebras were constructed by Xu \cite{Xu98S}.
The skein theory of $\mathscr{P}_\bullet^{H}(q,r)$ is inherited from the HOMFLY-PT skein relations \cite{Homfly,PT88}.
Thurston provided an intrinsic skein theory which was designed for 6-valent planar graphs \cite{Thu04}, which we call here Thurston's relation.

In this paper, we classify subfactor planar algebras with Thurston's relation generated by a non-trivial 3-box, namely singly generated Thurston-relation planar algebras:
\begin{theorem}[Main Theorem]\label{Thm:main}
	Any singly generated Thurston-relation planar algebra is either $E_6$ or (the semisimple quotient of) $\mathscr{P}_\bullet^{H}(q,r)$. Moreover, $r=q^N$ for some $N\in\mathbb{N}$, $N\geq3$, and $q=e^{\frac{i\pi}{N+l}}$ for some $l\in\mathbb{N}$, $l\geq3$, or $q\geq1$.
\end{theorem}
We remark that it would be interesting to find a skein theory for a single 3-box which gives a family of subfactor planar algebras involving $E_6$.

The paper is organized as follows.
In \S \ref{Sec:Pre} , we recall planar algebras, HOMFLY-PT skein theory, and Thurston's skein theory.
The $m$-box space of a singly generated Thurston-relation planar algebra $\mathscr{P}_{\bullet}$ generically has dimension $m!$.

In \S \ref{Sec:generic} , we classify the generic case, namely $\dim \mathscr{P}_{4,\pm}=24$.
We set up five formal variables for Thurston's relation using 3-box relations. We prove that only two variables survive after considering 4-box relations (Theorems \ref{thm1}, \ref{thm2}). Then we identify the two-parameter family with $\mathscr{P}_\bullet^{H}(q,r)$ (Theorem \ref{Thm:main1}).
Technically we simplify the computation by working on the reduced planar algebra with respect to the second Jones-Wenzl idempotent $f_2$. The reduced planar algebra has smaller $m$-box spaces and their dimensions are $1,0,1,2,9, \cdots$.
\footnote{One may consider the planar algebra to be generated by two trivalent vertices, but the relations become much more complicated.
	The H-I relation of the trivalent vertices is known as the 6-j symbol. However, the 6-j symbol is only known for the first couple of objects in terms of the formal variables. Our planar algebra may contain infinitely many simple objects. Thus the H-I relation in the reduced planar algebra does not give an evaluation algorithm. Thereby the H-I relation for the other shading (given by Thurston's relation) is necessary and important.}

In \S \ref{Sec:reduced} , we classify the reduced case, namely $\dim \mathscr{P}_{4,\pm}\leq 23$.

In \S \ref{Section:Pos}, we classify all $q,r$ for which $\mathscr{P}_{\bullet}^H(q,r)$  has a positive Markov trace. In this case, $\mathscr{P}_\bullet^{H}(q,r)$ has a unique involution $*$, so that its semi-simple quotient is a subfactor planar algebra.
Then we complete our classification, Theorem \ref{Thm:main}.

\section{Preliminaries}\label{Sec:Pre}

We refer the readers to \cite{JonPA} for definition, properties, examples and skein theory of planar algebras.

\subsection{HOMFLY-PT planar algebras}

The HOMFLY-PT polynomial is a link invariant given by a braid $\gra{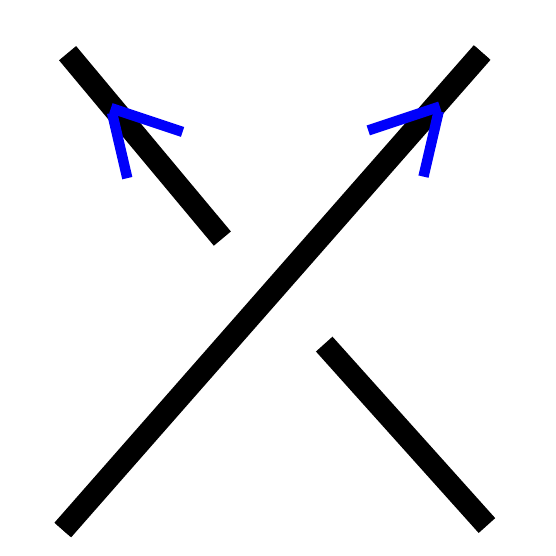}$ satisfying Reidemeister moves I, II, III and the Hecke relation.

\begin{align}
	&\text{Hecke relation:}    & \gra{b+}-\gra{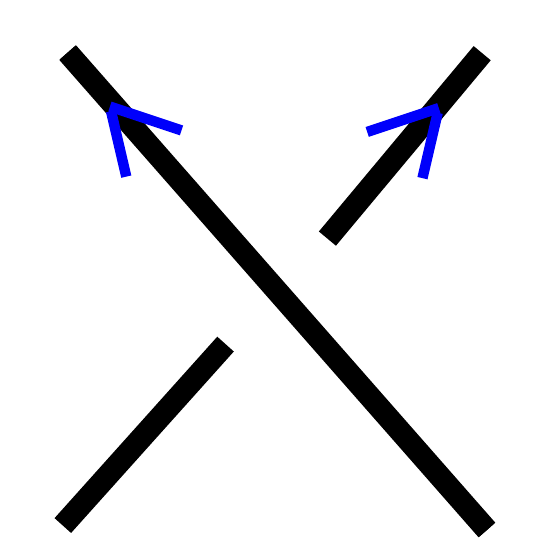}&=(q-q^{-1})\gra{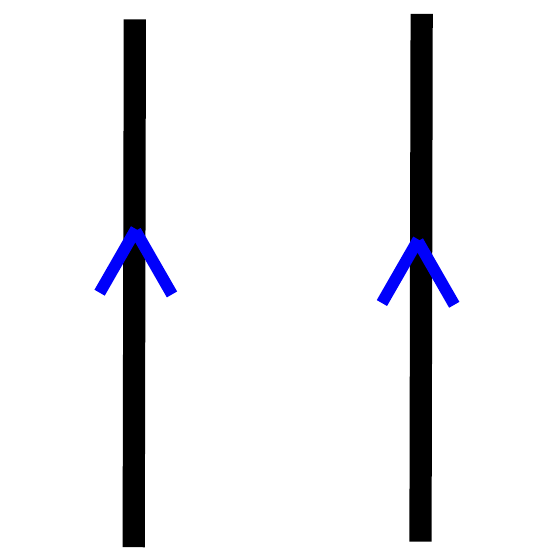}, && \\
	&\text{Reidemeister moves I:}   & \gra{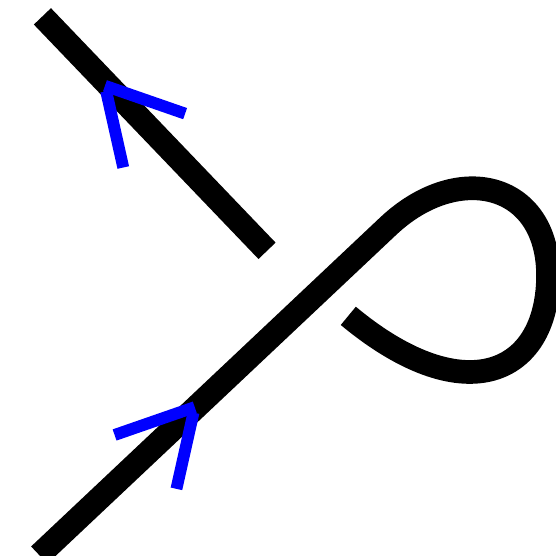}&=r\gra{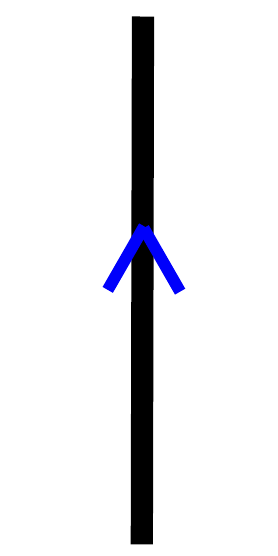}; & \gra{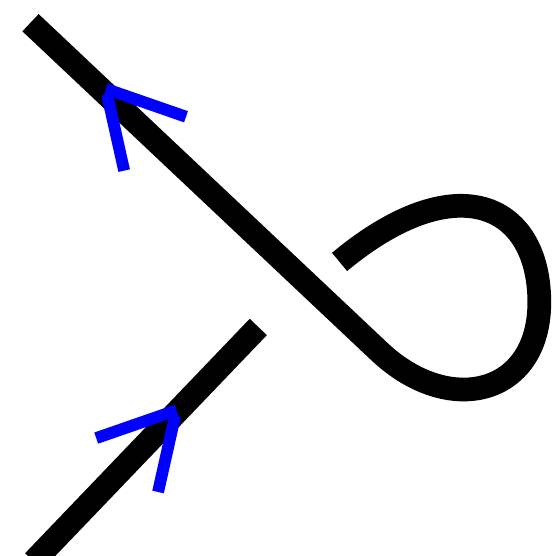}&=r^{-1}\gra{b1id}; \\
	&& \gra{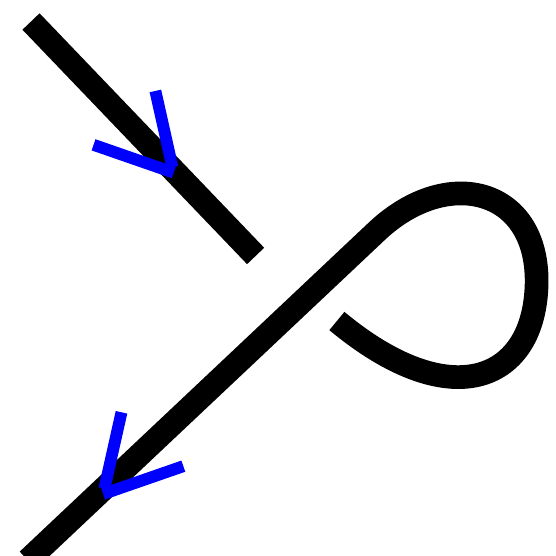}&=r \gra{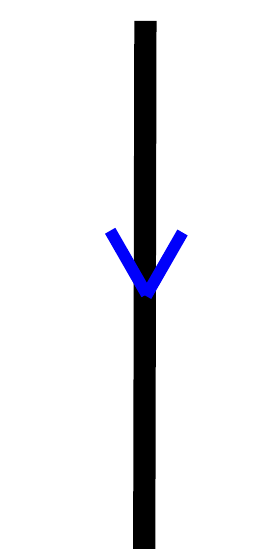}; & \gra{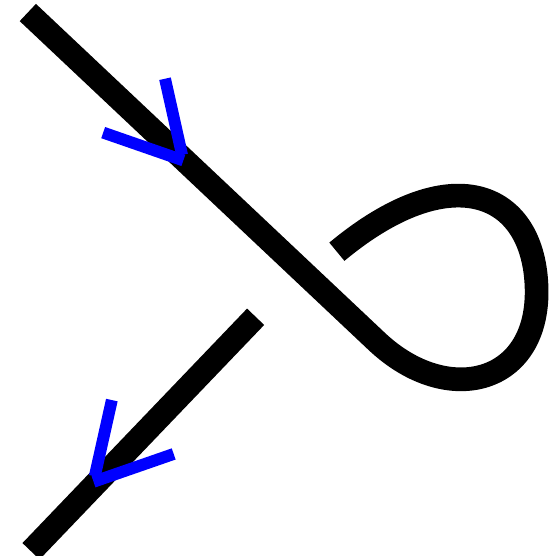}&=r^{-1}\gra{b1id-},\\
	&\text{Reidemeister moves II:}  & \gra{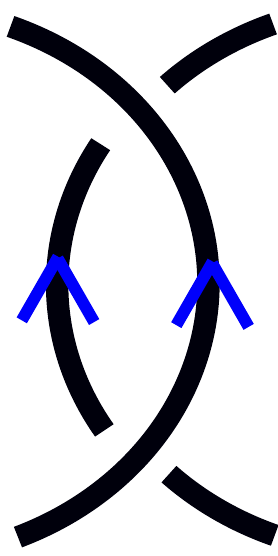}&=\gra{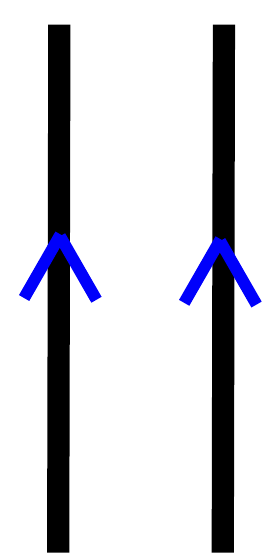}; &  \gra{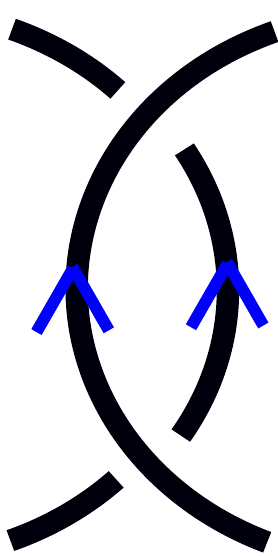}&=\gra{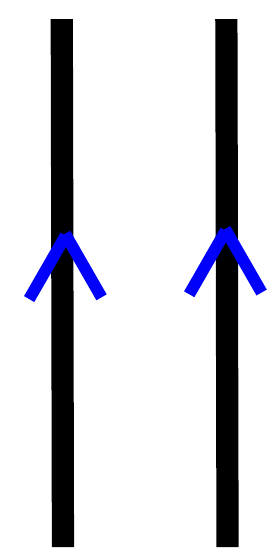}; \\
	&& \gra{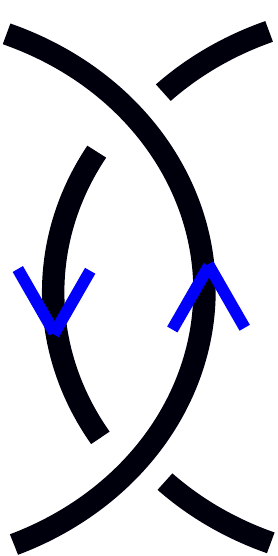}&=\gra{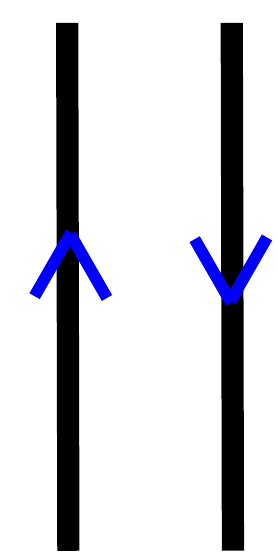}; &  \gra{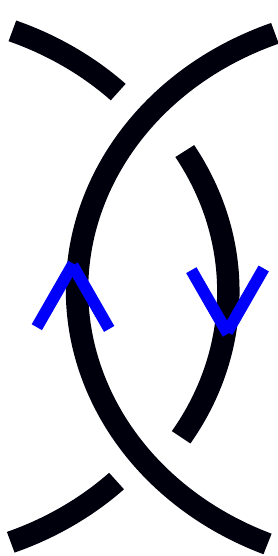}&=\gra{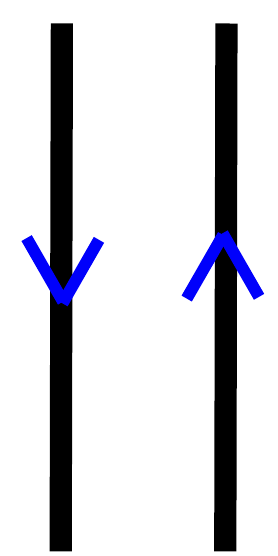},\\
	&\text{Reidemeister moves III:}  & \gra{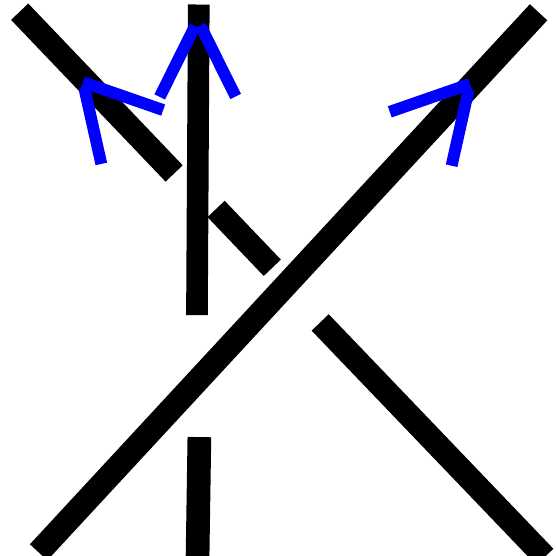}&=\gra{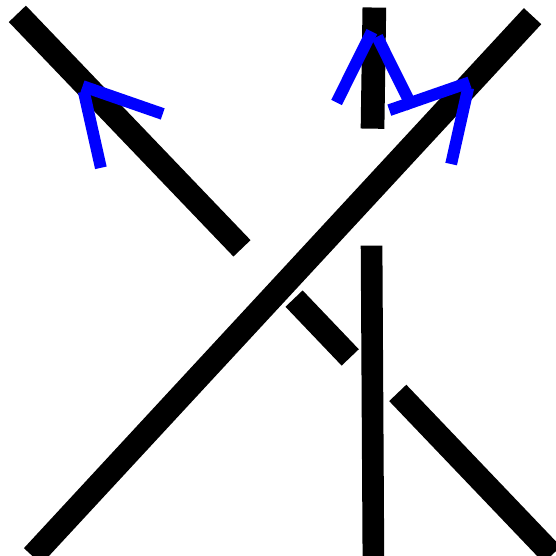}; &  \gra{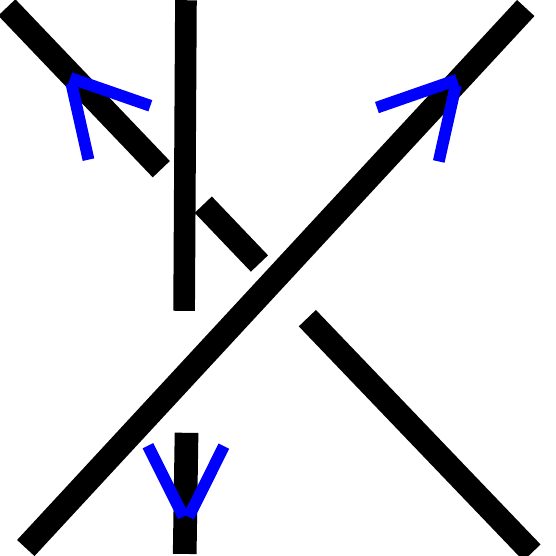}&=\gra{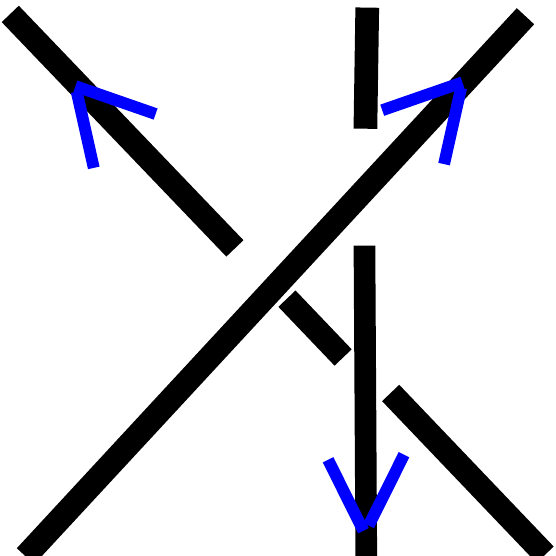},\\
	&\text{circle parameter:} & \gra{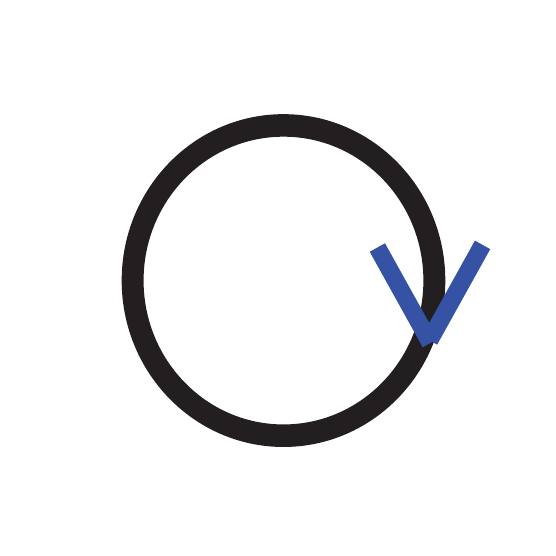}&=\gra{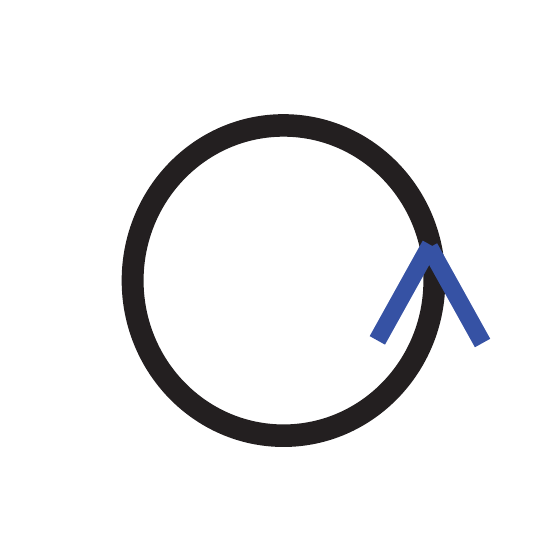}=\delta,\label{circpara} \\
	&&r-r^{-1}&=\delta(q-q^{-1})
\end{align}
\begin{remark}
	When $q=\pm1$, we have $r=\pm1$. Then the skein relation is determined by the circle parameter $\delta$. When $q\neq\pm1$, $\delta=\frac{r-r^{-1}}{q-q^{-1}}$.
\end{remark}

Let $\sigma_i$, $i\geq1$, be the diagram obtained by adding $i-1$ oriented (from bottom to top) through-strings on the left of $\gra{b+}$.
The Hecke algebra of type $A$ is a (unital) filtered algebra $H_{\bullet}$. The algebra $H_n$ is generated by $\sigma_i$, $1\leq i \leq n-1$ and $H_n$ is identified as a subalgebra of $H_{n+1}$ by adding an oriented through string on the right.
Over the field $\mathbb{C}(q,r)$, the equivalence classes of minimal idempotents of $H_n$ are indexed by Young-diagrams with $n$ cells. The trace formula is given by

\begin{theorem}[\cite{Res87,AisMor98}] \label{traceformula}
	Let $\lambda$ be a Young diagram and $m_\lambda$ the minimal idempotent corresponding to $\lambda$, then
	\begin{equation}\label{equ:traceformula}
	Tr(m_\lambda)=\prod_{i,j}\displaystyle\frac{rq^{c(i,j)}-r^{-1}q^{-c(i,j)}}{q^{h_{(i,j)}}-q^{-h_{(i,j)}}}
	\end{equation}	
	where $c(i,j)=j-i$ is the content of the cell $(i,j)$ in $\lambda$ and $h_{(i,j)}$ is its hook length.
\end{theorem}

Jones studied the planar algebras $\mathscr{P}_\bullet^H(q,r)$ associated with HOMFLY-PT skein relation \cite{JonPA}. Its $n$-box space consists of HOMFLY-PT diagrams which have $2n$ boundary points and alternating orientation on the boundary as follows:
$$\grb{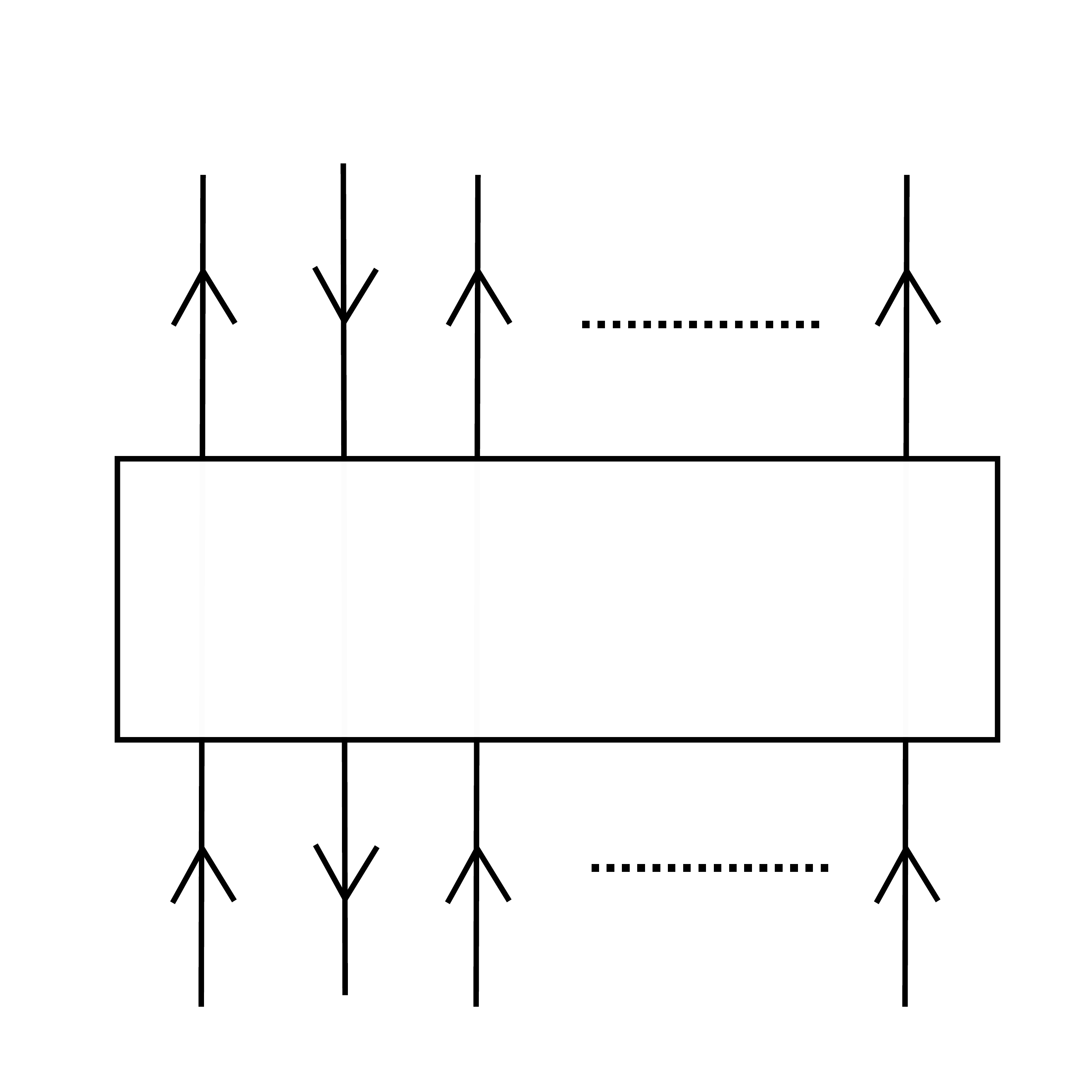}.$$
Moreover, he proved that this planar algebra is generated by a 3-box:
\begin{theorem}[Jones]\label{Thm:Jones3-box}
	The planar algebra $\mathscr{P}_\bullet^H(q,r)$ is generated by $\graa{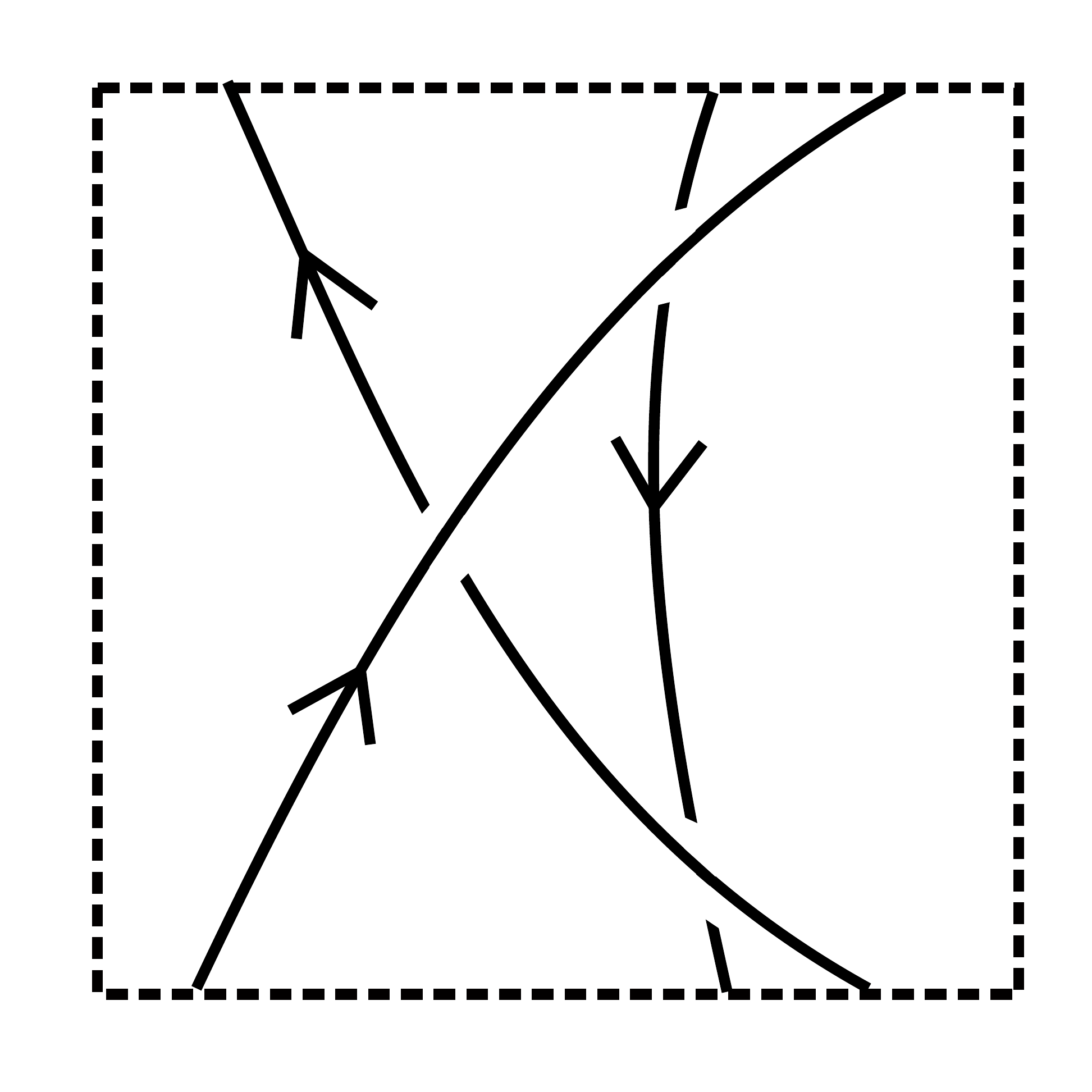}$.
\end{theorem}
When $r=q^N$ for some $N\in\mathbb{N}$ and $q=e^{\frac{i\pi}{N+l}}$ for some $l\in\mathbb{N}$ or $q\geq1$, $\mathscr{P}_\bullet^H(q,r)$  admits an involution * such that the Markov trace is positive semidefinite. Therefore, the semisimple quotient of $\mathscr{P}_\bullet^H(q,r)$ is a subfactor planar algebra, which can be constructed from the representation theory of quantum SU(N) \cite{Xu98S}. (When $q=1$, we have $r=1$ and $\delta=N$.)

We prove that these are the only possibilities such that $\mathscr{P}_\bullet^H(q,r)$ has positivity in Theorem \ref{Thm:positivity}. First, we identify the isomorphism classes of $\mathscr{P}_\bullet^H(q,r)$:
\begin{proposition}\label{lemsym}
	The four planar algebras $\mathscr{P}_\bullet^H(q,r)$,$\mathscr{P}_\bullet^H(-q^{-1},r)$,$\mathscr{P}_\bullet^H(-q,-r)$,$\mathscr{P}_\bullet^H(q^{-1},r^{-1})$ are isomorphic.
\end{proposition}
\begin{proof}
	Note that $(q,r)$ and $(-q^{-1},r)$ define the same planar algebra.
	
	The isomorphism between $\mathscr{P}_\bullet^H(q,r)$ and $\mathscr{P}_\bullet^H(-q,-r)$ is induced by sending $\gra{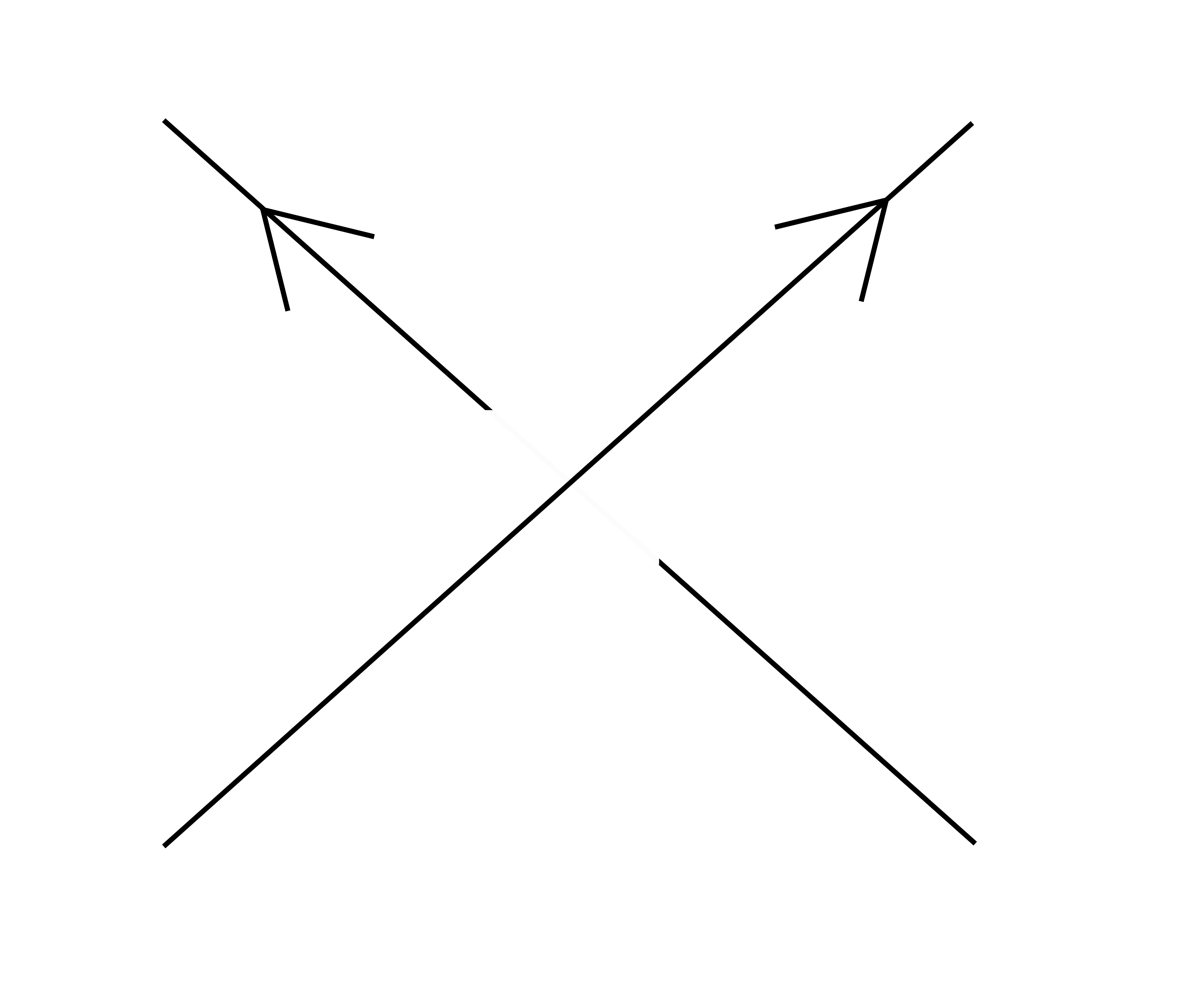}$ to $-\gra{HOMFLYG}$.
	
	The isomorphism between $\mathscr{P}_\bullet^H(q,r)$ and  $\mathscr{P}_\bullet^H(q^{-1},r^{-1})$ is induced by sending $\gra{HOMFLYG}$ to $\gra{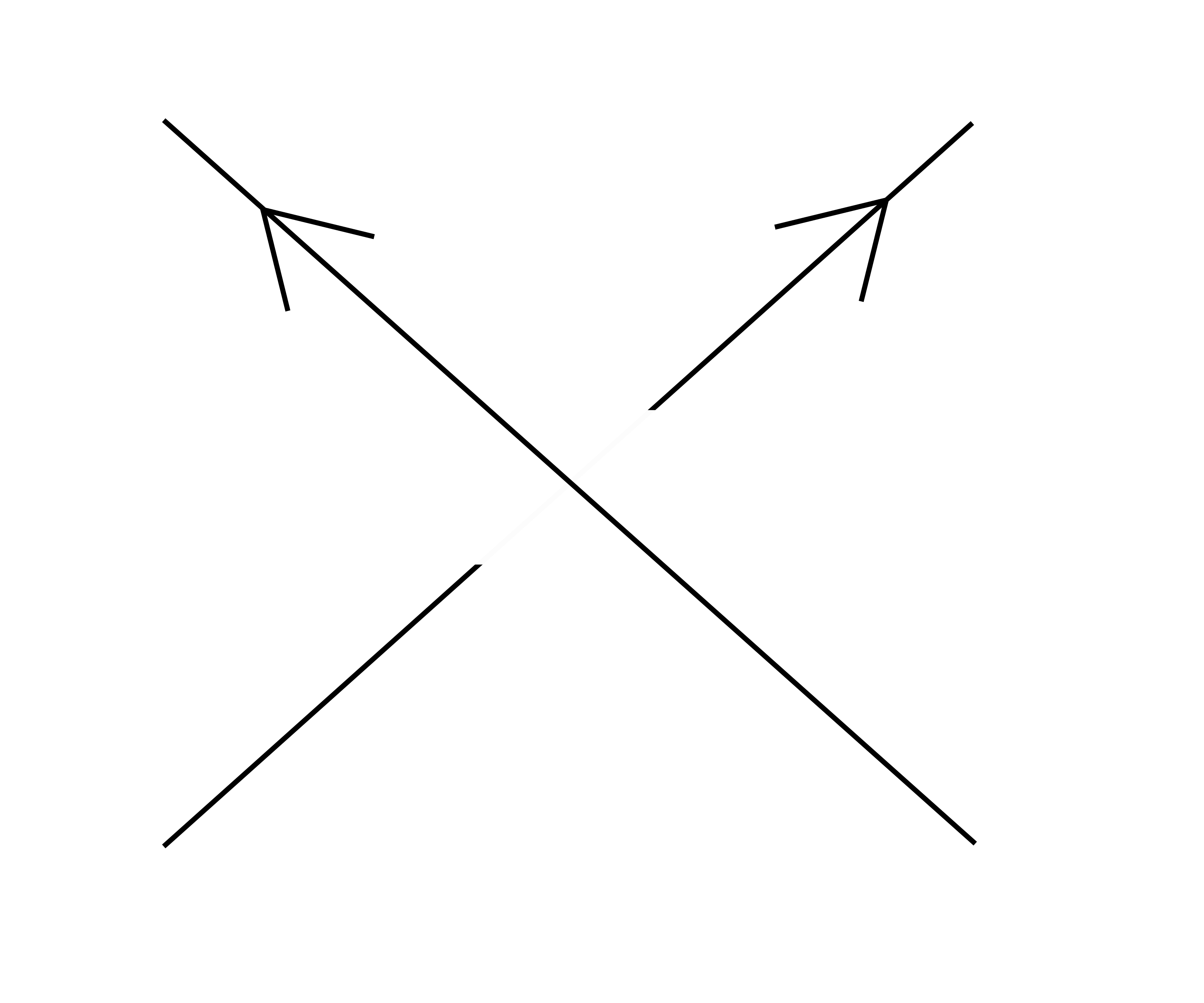}$.
	
\end{proof}

\subsection{Thurston's skein relations}
As emphasized in the introduction, skein theory provides an important perspective from which to understand a planar algebra for many reasons. Skein theories are important starting points for the construction and classification of planar algebras.
In this paper we will study subfactor planar algebras generated by a 3-box.

Recall that $\mathscr{P}_\bullet^H(q,r)$ is generated by a 3-box. Although it has a skein theory derived from the HOMFLY-PT skein relation,
Thurston provides a skein theory intrinsic to the 3-box generator of $\mathscr{P}_\bullet^H(q,r)$ \cite{Thu04}:
\begin{definition}[Thurston's relation \cite{Thu04}]
We say a $3$-box $S$ satisfies the Thurston's relation if
\begin{align}
&1\rightarrow0 \text{ Move:} &&\graa{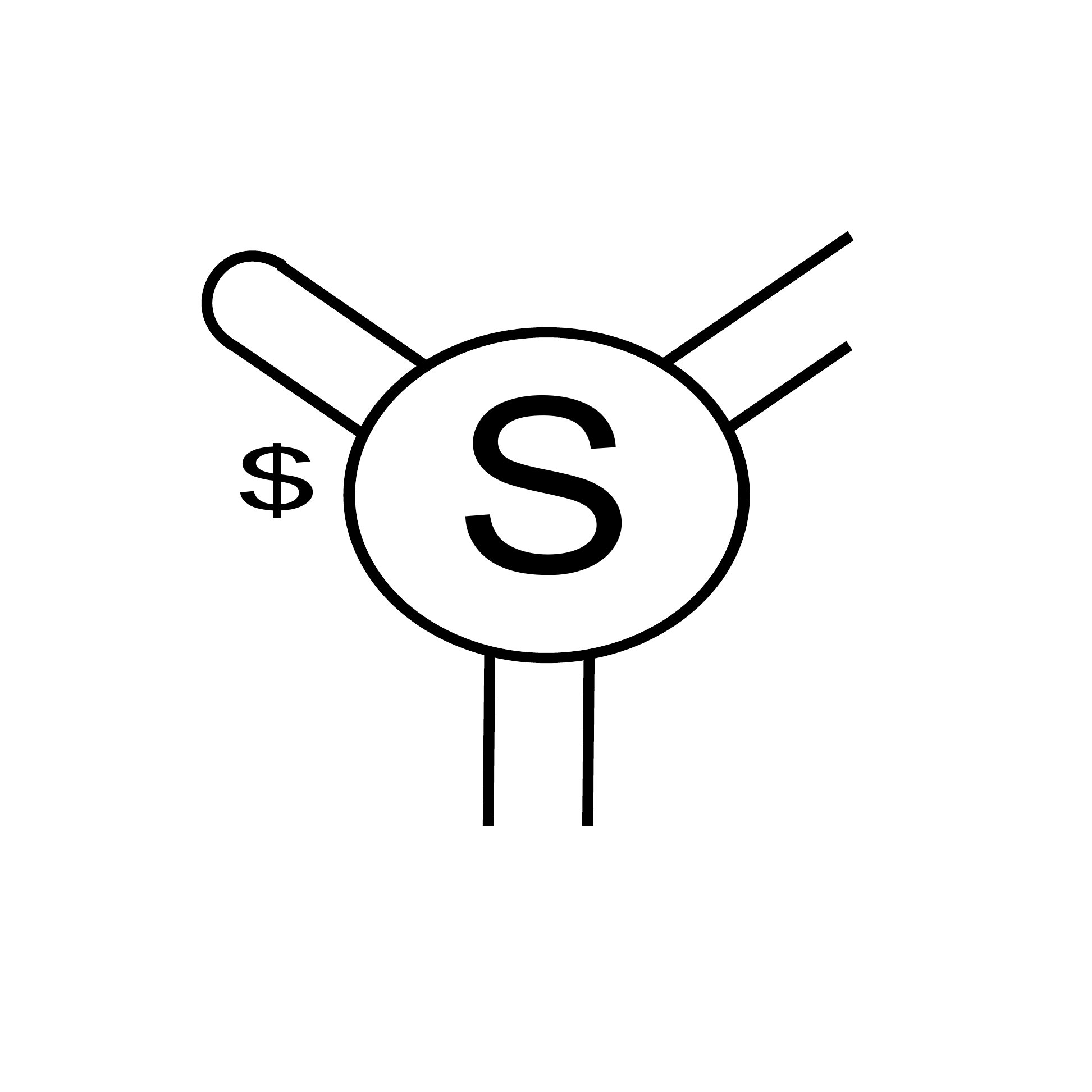},\graa{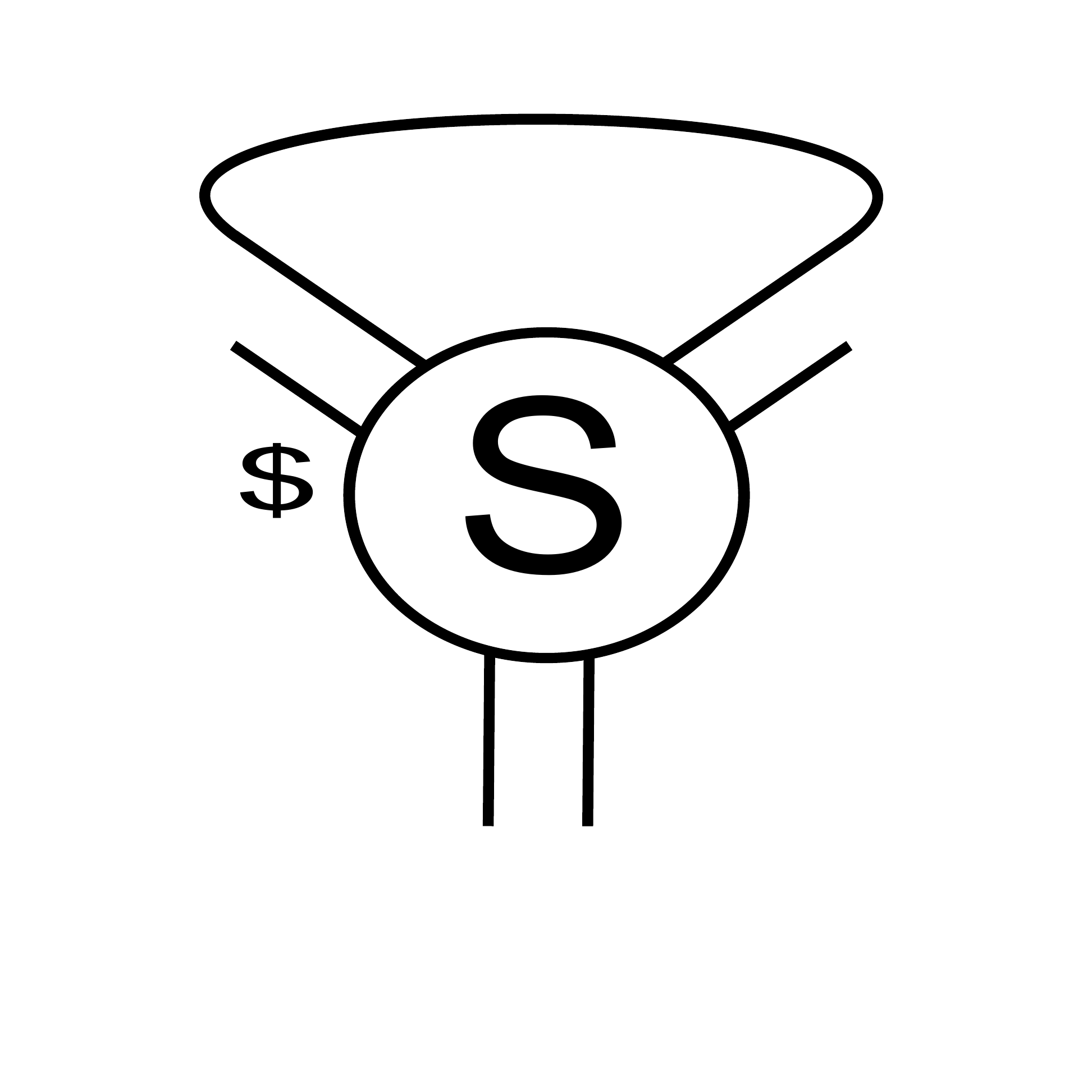},\graa{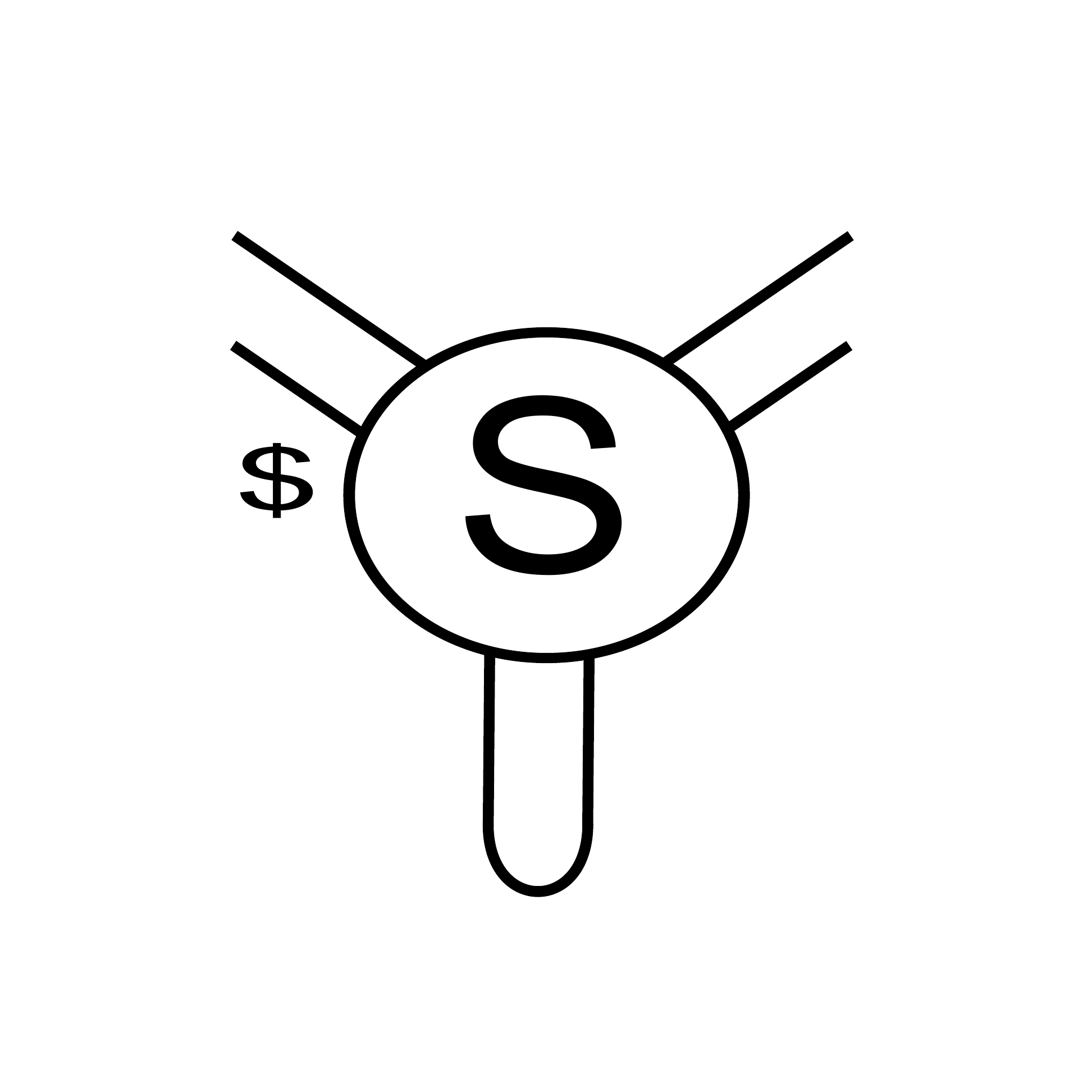},\ \dots\ ,\graa{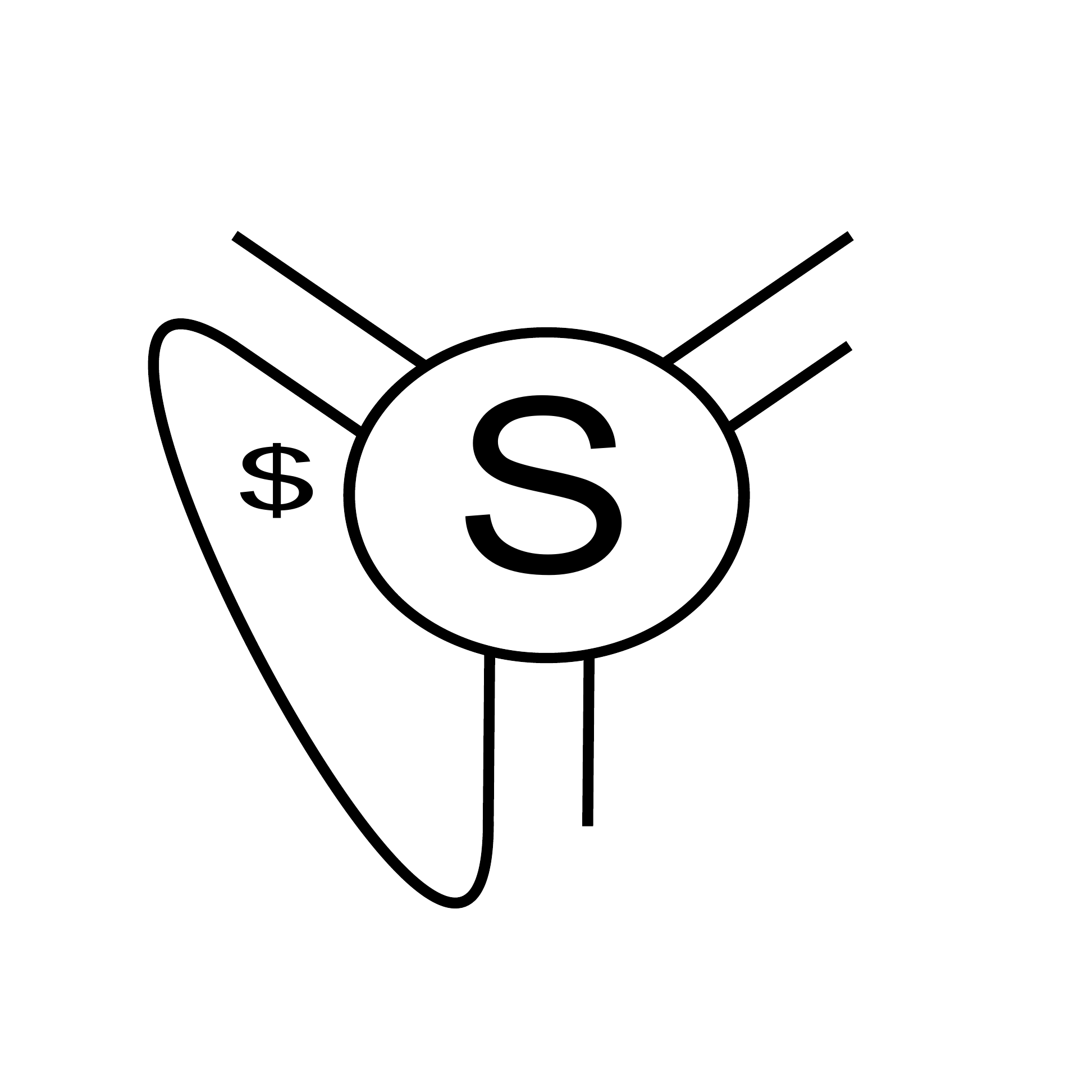}=(lower\ terms); \label{Equ:H-I0}\\
\text{Unshaded } &2\leftrightarrow2 \text{ move:} &&\graa{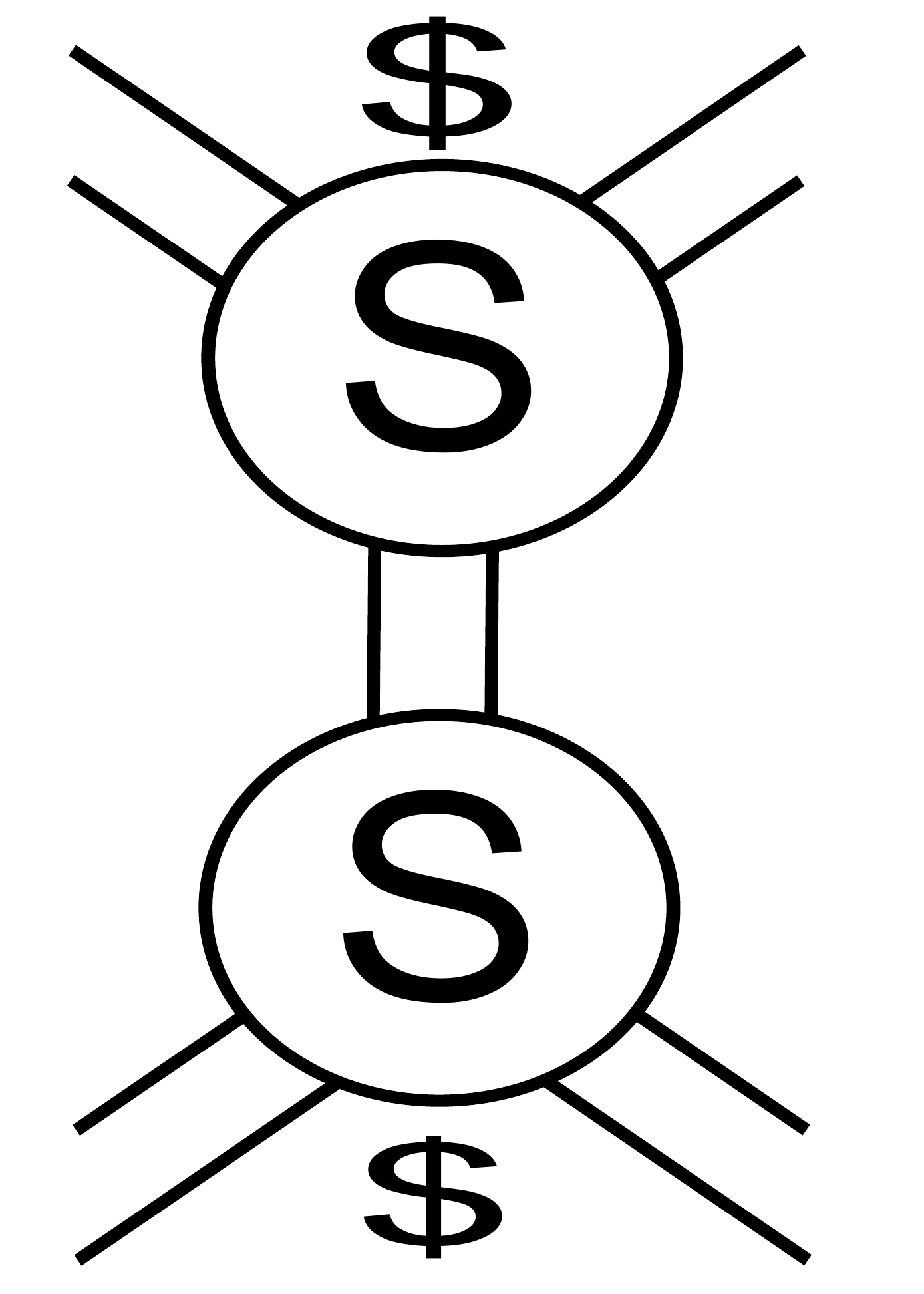}=a\ \graa{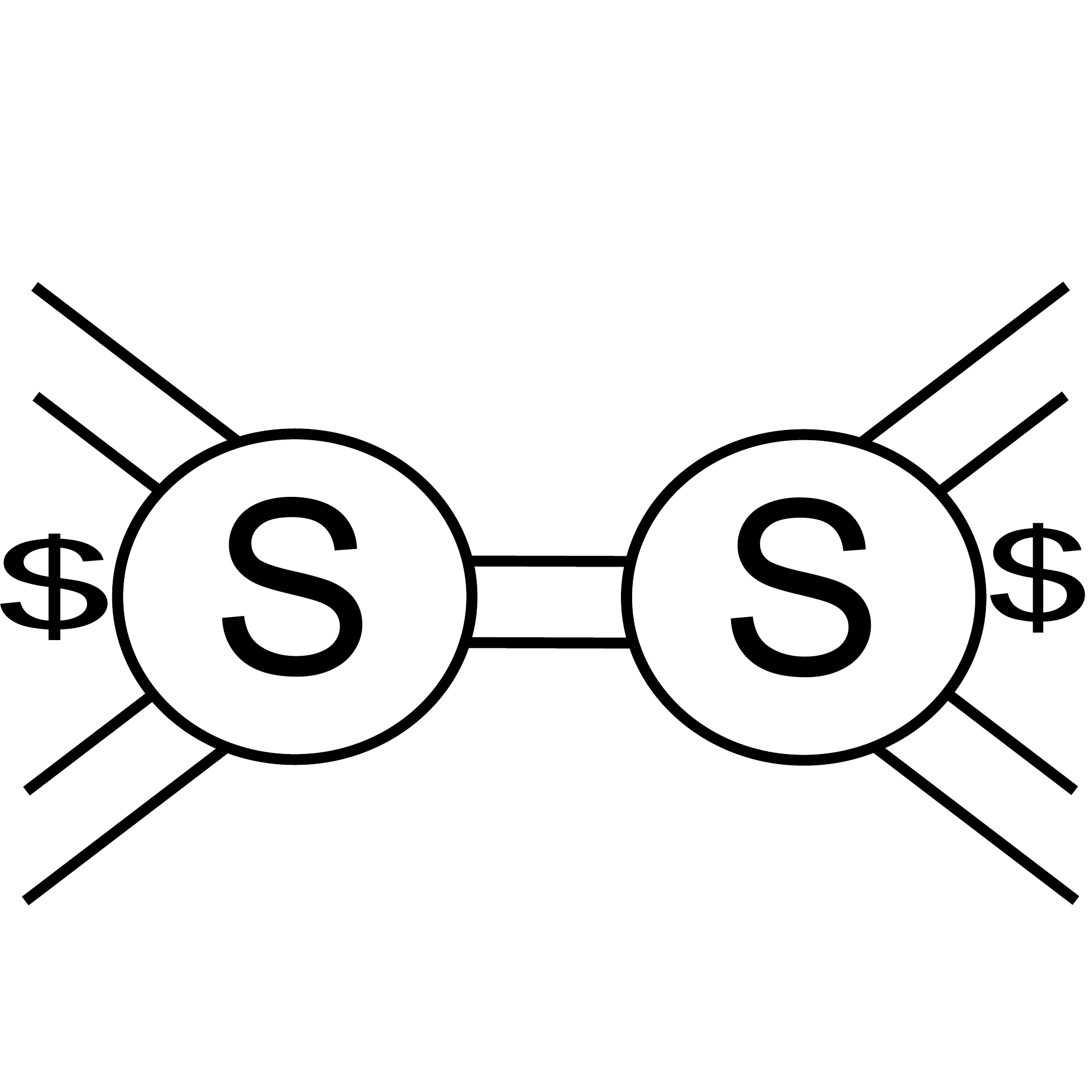}+(lower~terms), ~a\neq0 \label{Equ:H-I1} \;;\\
\text{Shaded } &2\leftrightarrow2 \text{ move:} &&\graa{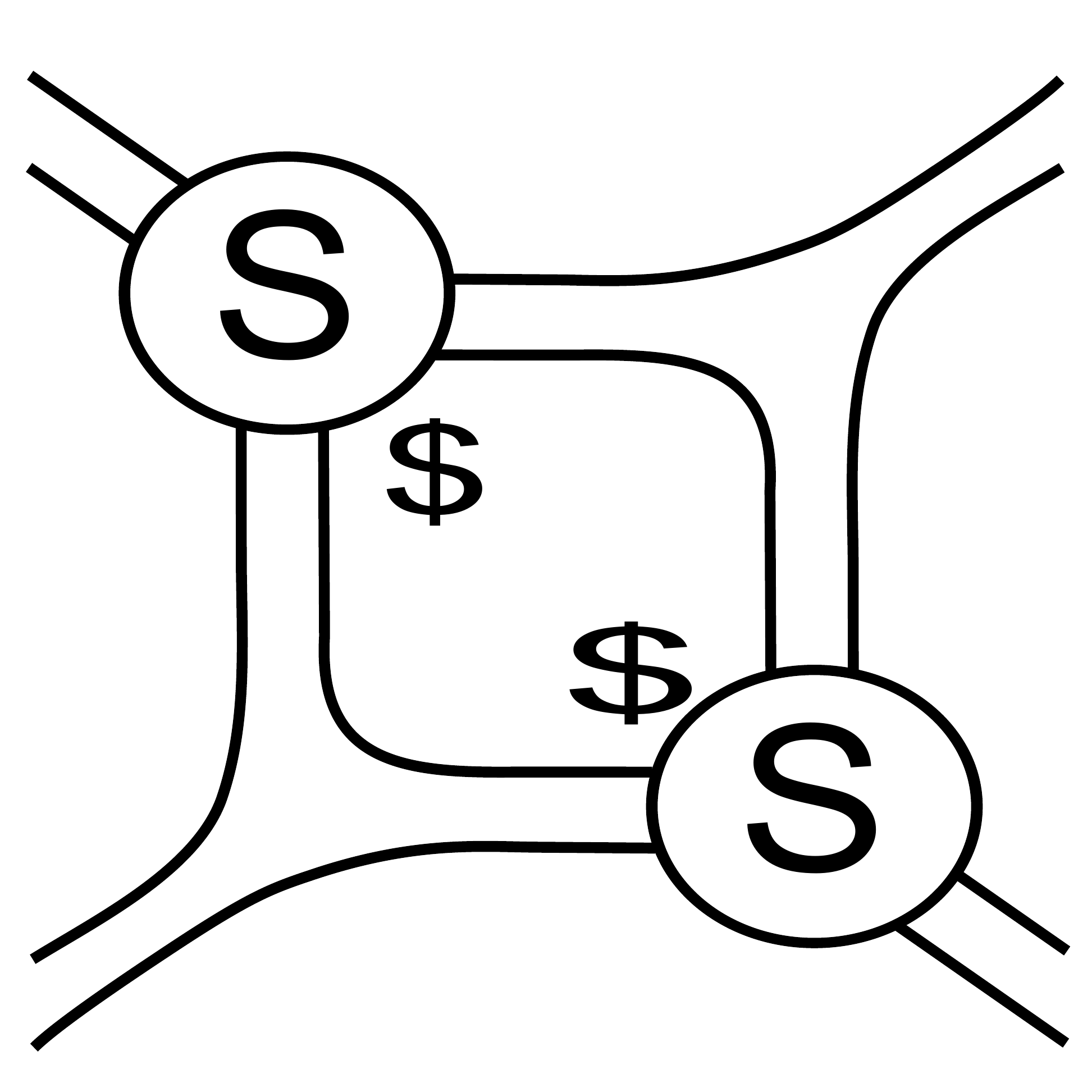}=b\graa{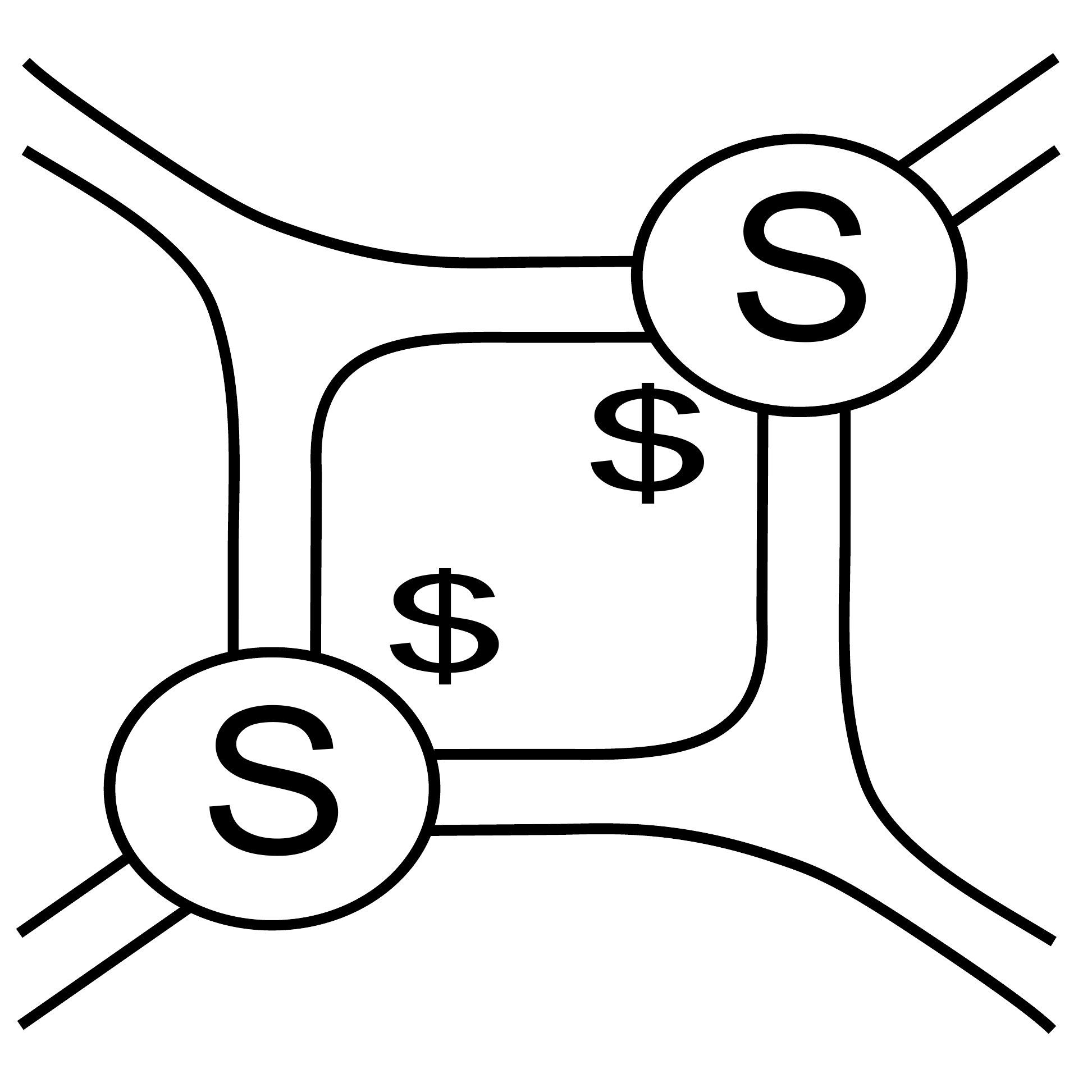}+(lower~terms), ~b\neq0 \label{Equ:H-I2} \;,
\end{align}
where lower terms are a linear combination of diagrams with less generators.
\end{definition}

\begin{remark}
  One can generalize Thurston's relation for multiple 3-box generators.
\end{remark}

He proved that the above relations are evaluable, in the sense that there exists an algorithm for evaluating every closed planar diagram to a scalar. Moreover, the standard forms in the sense of Thurston \cite{Thu04} give a basis of the $n$-box space when their dimension achieves the maximum of $n!$.

\begin{theorem}[Thurston, \cite{Thu04}]
	Suppose $\mathscr{P}_\bullet$ admits Thurston's relation. Then $\mathscr{P}_{n,\pm}$ is spanned by the standard forms and $\dim{\mathscr{P}_{n,\pm}}\leq n!$.
\end{theorem}

\begin{corollary}[Thurston, \cite{Thu04}]\label{Cor:24-basis}
	In the generic case, namely $\mathscr{P}_{4,\pm}=24$, we have the basis of the 4-box space given by the standard form as follows:
\itemize{
\item 14 Temperley-Lieb diagrams;

\item 8 diagrams in the annular consequences, which we denote by ${AC}$;

\item 2 diagrams with two generators: \graa{I} and \graa{SI}.
}
Moreover, one can replace them by the other diagrams with two generators using Thurston's relation \eqref{Equ:H-I1}, \eqref{Equ:H-I2}.
\end{corollary}

Thurston's relation look similar to the 6j symbol in a monoidal category, but it provides a much better evaluation algorithm. The 6j symbols in a monoidal category are the coefficients of the change of basis matrix in the hom spaces, and a monoidal category is determined by the 6j symbols up to monoidal equivalence, however it seems hopeless to determine the 6j symbols in general.
Actually for the example $\mathscr{P}_\bullet^H(q,r)$ appearing in our classification, it is difficult to compute the 6j symbols. Thurston's relation only requires partial 6j symbols for two objects, together with the data for the dual planar algebra.
This combination, rather than considering only one side, appears to be powerful, and determines the planar algebra completely.

The main purpose of this paper is to classify all subfactor planar algebras generated by a 3-box $S$ satisfying Thurston's relation.
We give the classification for the generic case in \S \ref{Sec:generic} and for the reduced case in \S \ref{Sec:reduced}.
Since the subfactor with index at most $4$ is classified, we only need to consider the case $\delta^2> 4$.
Actually, $E_6$ and extended $E_6$ subfactor planar algebras are generated by a 3-box with Thurston's relation. The dimension of their 4-box spaces are 21 and 22 respectively.

\section{Generic case}\label{Sec:generic}
In this section, we classify subfactor planar algebras $\mathscr{P}_{\bullet}$ generated by a 3-box $S$ with Thurston's relation, such that $\mathscr{P}_{4,\pm}=24$.
In this case, we have $\delta>2$ and $\mathscr{P}_{3,\pm}=6$.
Most results in this section also work for case $\mathscr{P}_{4,\pm}=23$.
\subsection{Generators}

\begin{notation}
Suppose $\mathscr{P}_{\bullet}$ is a spherical planar algebra.
We use the following notations: $e_n$ is the $n^{\rm th}$ Jones projection; $f_n$ is the $n^{\rm th}$ Jones-Wenzl idempotent; and $\mathscr{I}_{n,\pm}$ is the basic construction ideal in $\mathscr{P}_{n,\pm}$.
\end{notation}

Since $\dim(\mathscr{P}_{3,+})=6$, there exists two minimal idempotents $P$ and $Q$
in $\mathscr{P}_{3,+}/\mathscr{I}_{3,+}$ and $P+Q=f_3$. Since $tr(f_3)\neq 0$, we assume that $tr(Q)\neq 0$.
We take $S=\gamma P-Q$, where $\gamma=\frac{tr(P)}{tr(Q)}$, as the generator for $\mathscr{P}_\bullet$. Then $S$ has the following relations:

\begin{enumerate}
	\item
	$S$ is totally uncappable, i.e,
\begin{align} \label{Equ:S1}
0&=\grb{Uncappable1}=\grb{Uncappable2}=\cdots=\grb{Uncappable3}=\grb{Uncappable4}.
\end{align}
	\item
	$S$ is an eigenvector of the (2-click) rotation $\rho$. i.e
\begin{align} \label{Equ:S2}
 \mathcal{\rho}(S)&=\omega S, \text{~and~} \omega^3=1.
\end{align}
	\item
	$S$ satisfies a quadratic relation:
\begin{align} \label{Equ:S3}
S^2&=(\gamma-1)S+\gamma f_3.
\end{align}
\end{enumerate}

Note that Property (1) is $(1\rightarrow0)$ move in Thurston's relation.

Now we focus on the $f_2$-cutdown of $\Pl_{n,+}$ instead of the entire subfactor planar algebra.
Technically, this reduces the dimension of the $n$-box space and simplify the computation.
Elements in the $n$-box cutdown space will be elements $x\in \Pl_{n,+}$ of the form
$$\grb{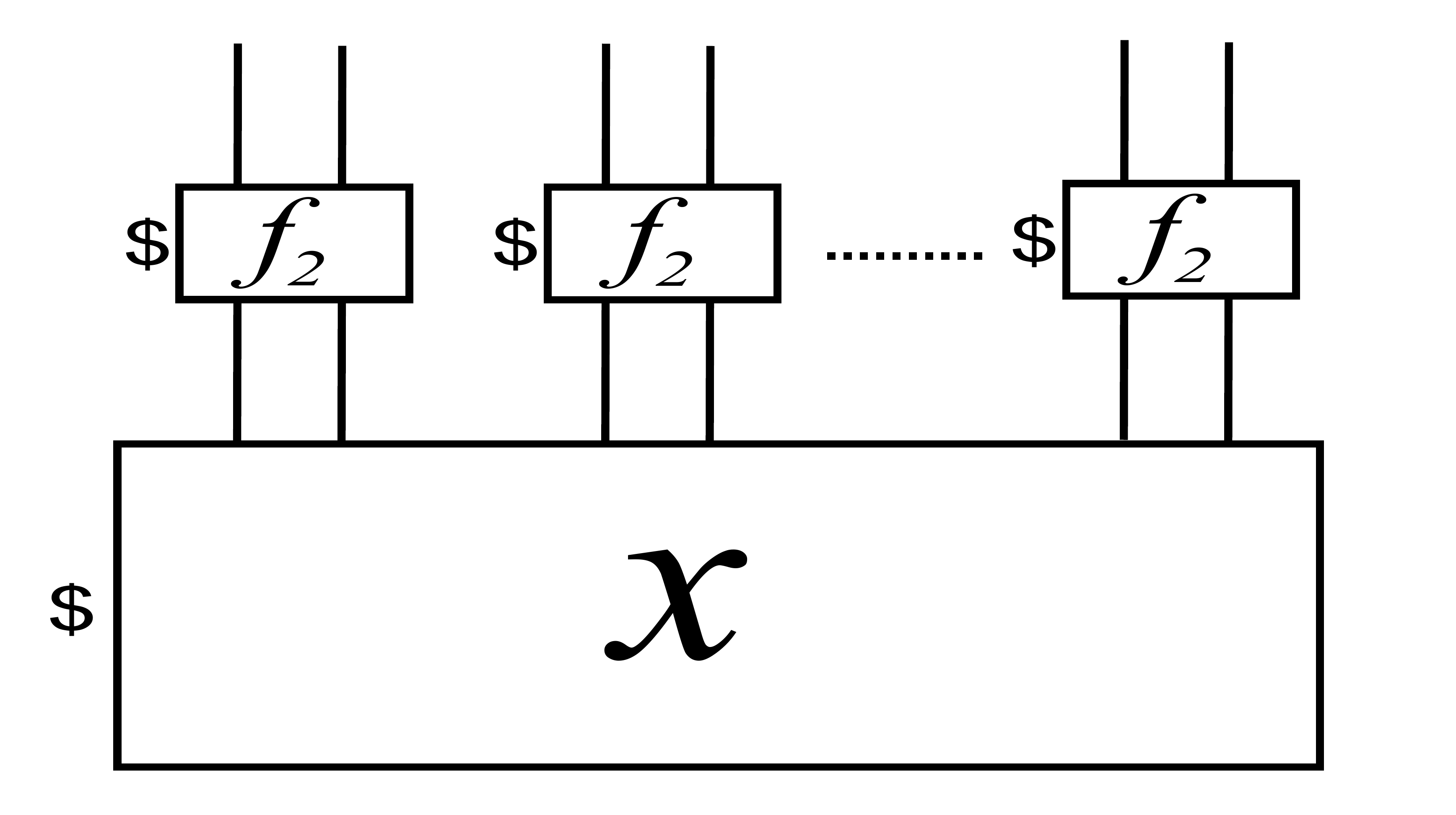} \;.$$
Then the $f_2$-cutdown of $\Pl_{3,+}$ has a basis $\{S, \graa{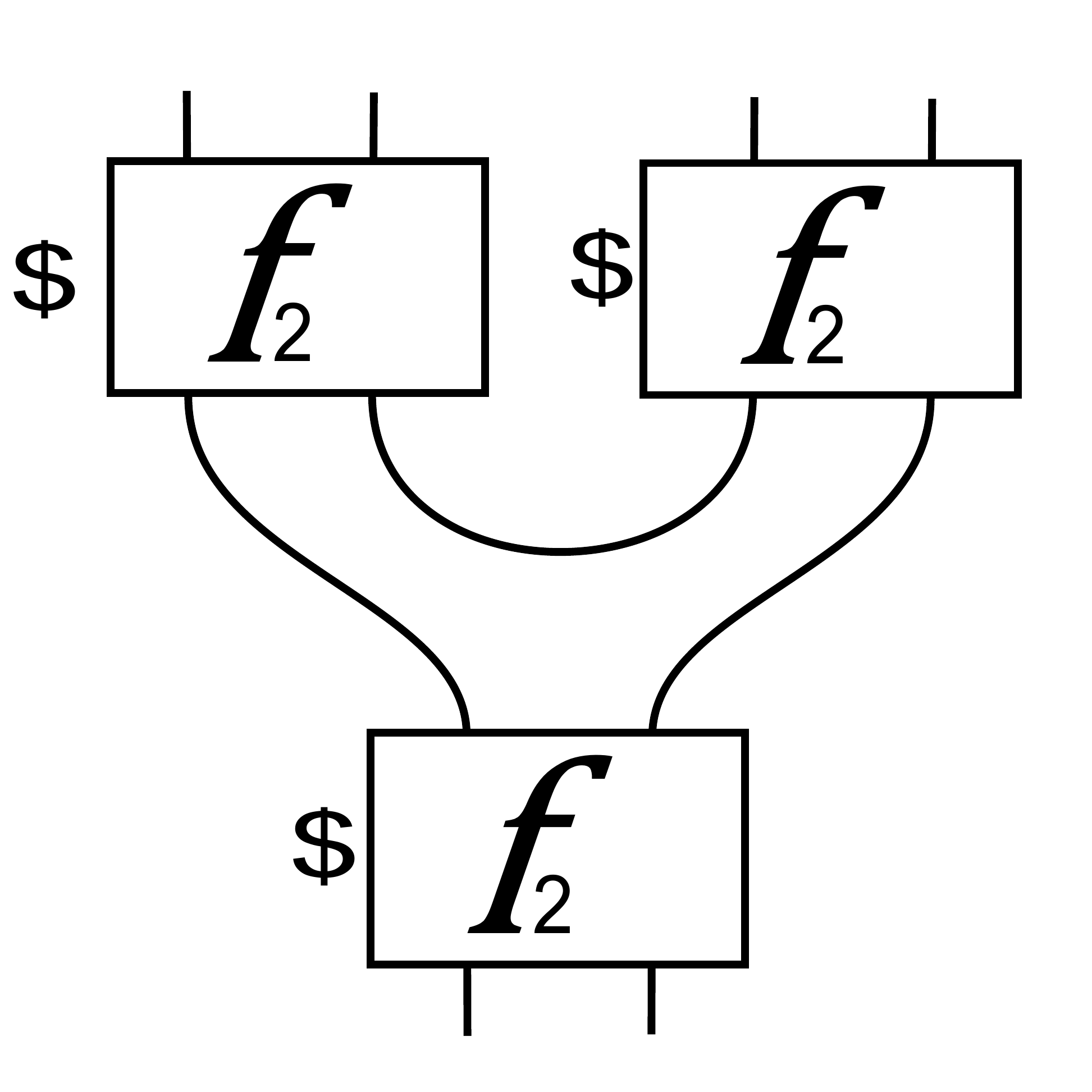}\}$.
We can consider the elements in the $f_2$-cutdown as morphisms in the $N-N$ bimodule category (the even part of the subfactor planar algebra).
In terms of $N-N$ bimodule maps, we rewrite $f_2$ as a single string labelled by $f_2$, and we ignore the label if there is no confusion.
In this setting, we express $S$ and $\graa{TLTL}$ as $N-N$ bimodule maps as follows:
\begin{align}
\gra{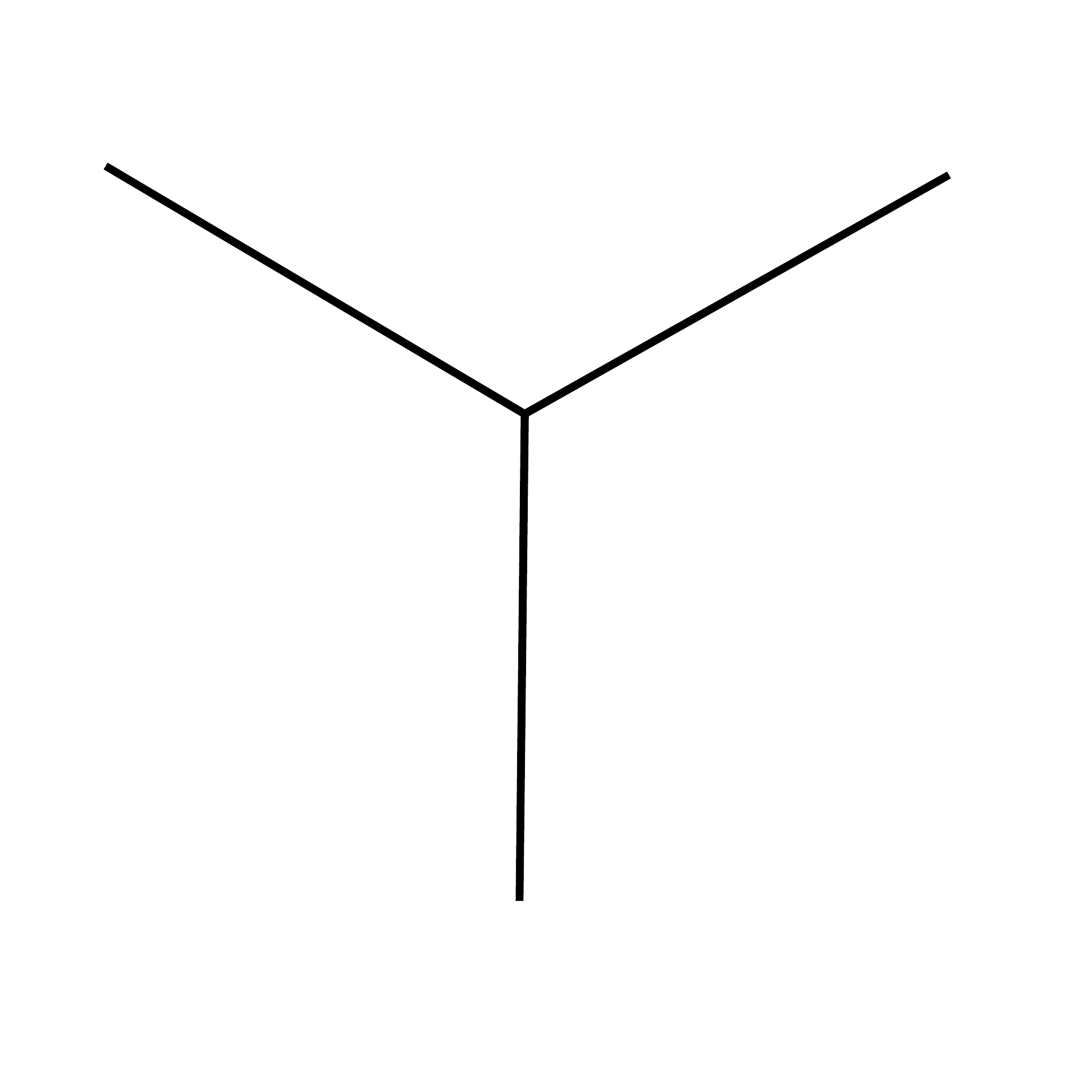}&:=\graa{TLTL} \;,\\
\gra{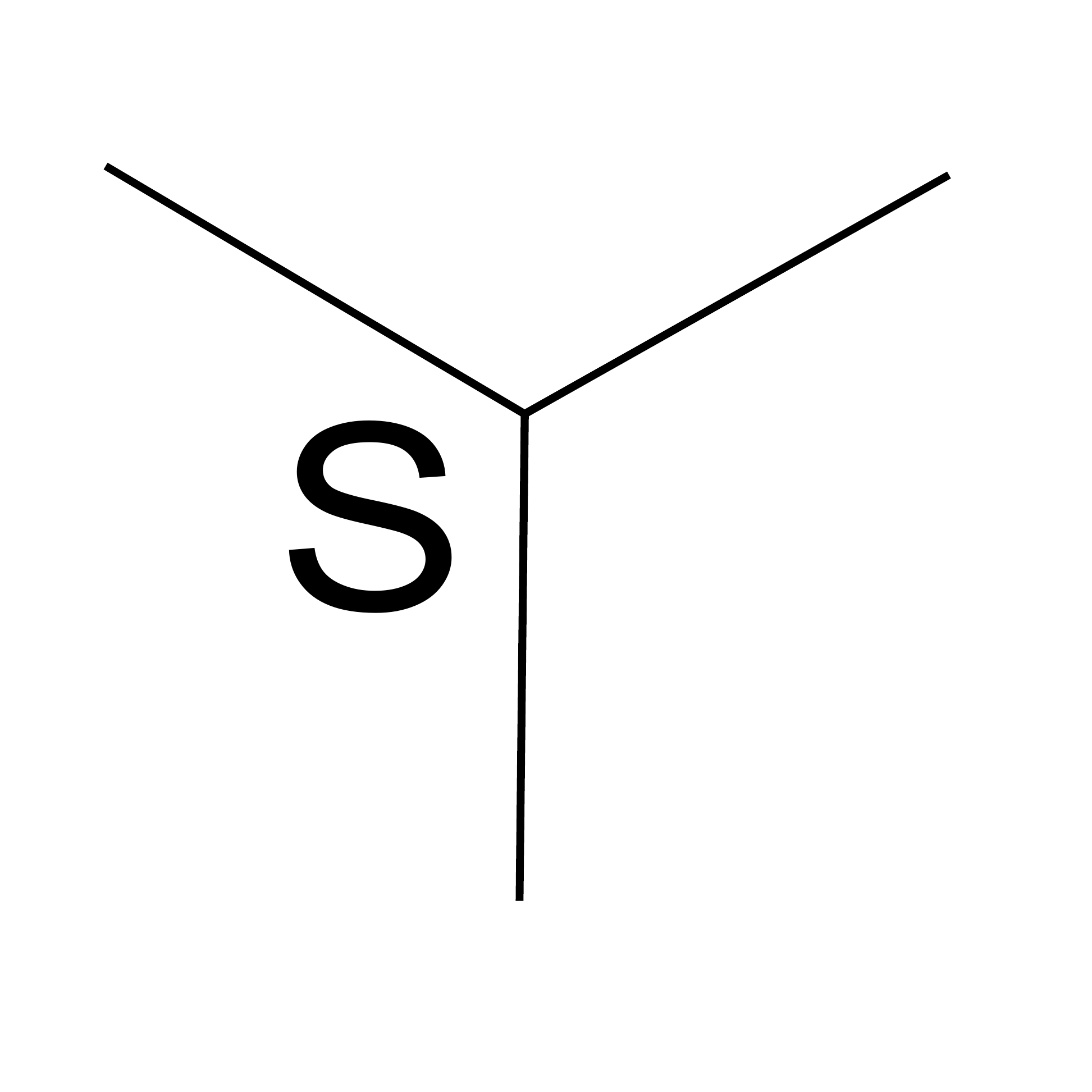}&:=\graa{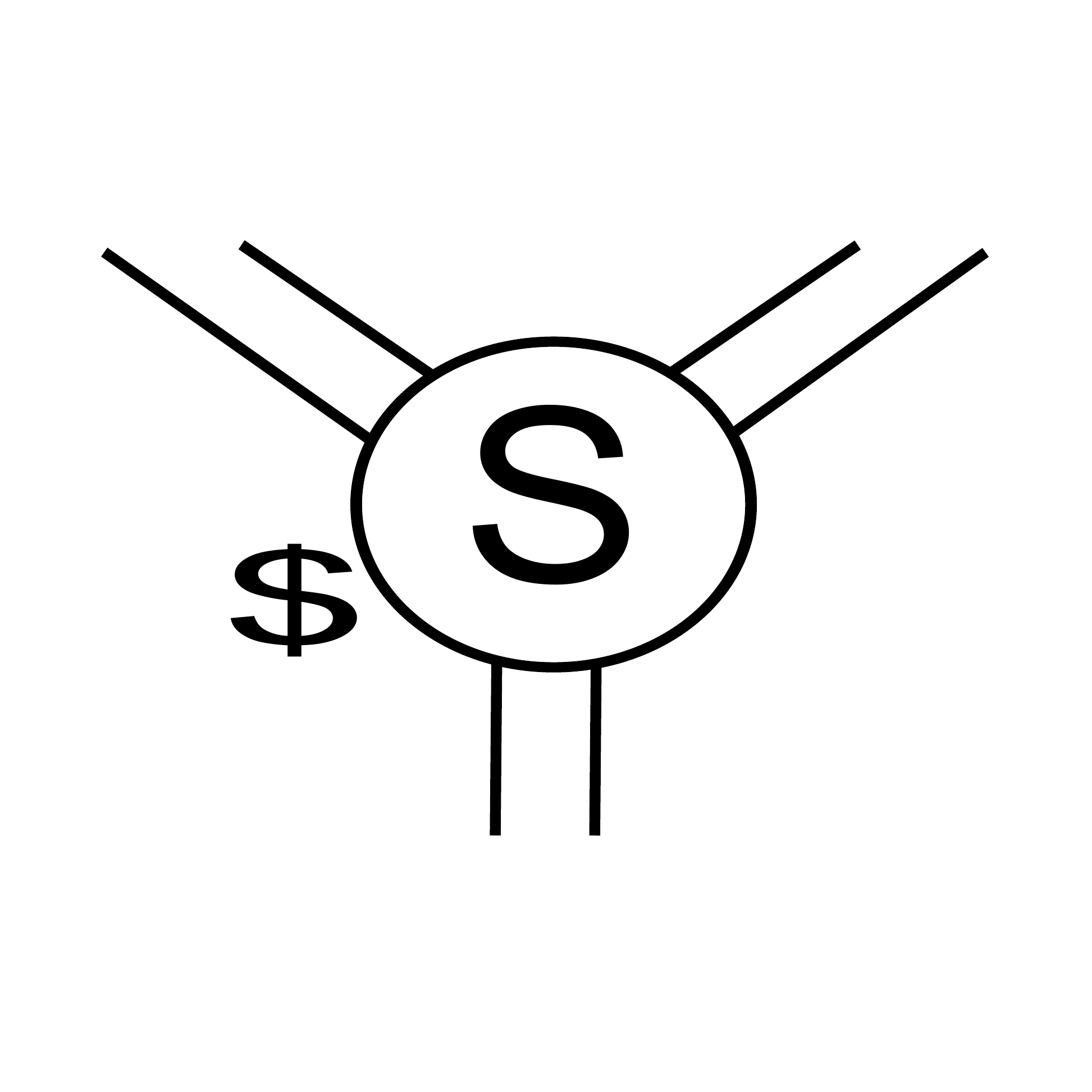}.
\end{align}
where the position of $S$ indicates the position of the $\$$ sign.

\subsection{Relations in 3-boxes}\label{Section:3box}
Now let us set up and simplify the formal variables for the relations of \gra{TL} and \gra{wS} in $\Pl_{3,+}$

\begin{lemma}\label{skein}

We have the following skein relations in the $f_{2}$ cut-down in terms of $\delta,\omega,\gamma$ and one new parameter $\varepsilon$: \\
\begin{itemize}
\item[(i)]$~~~\gra{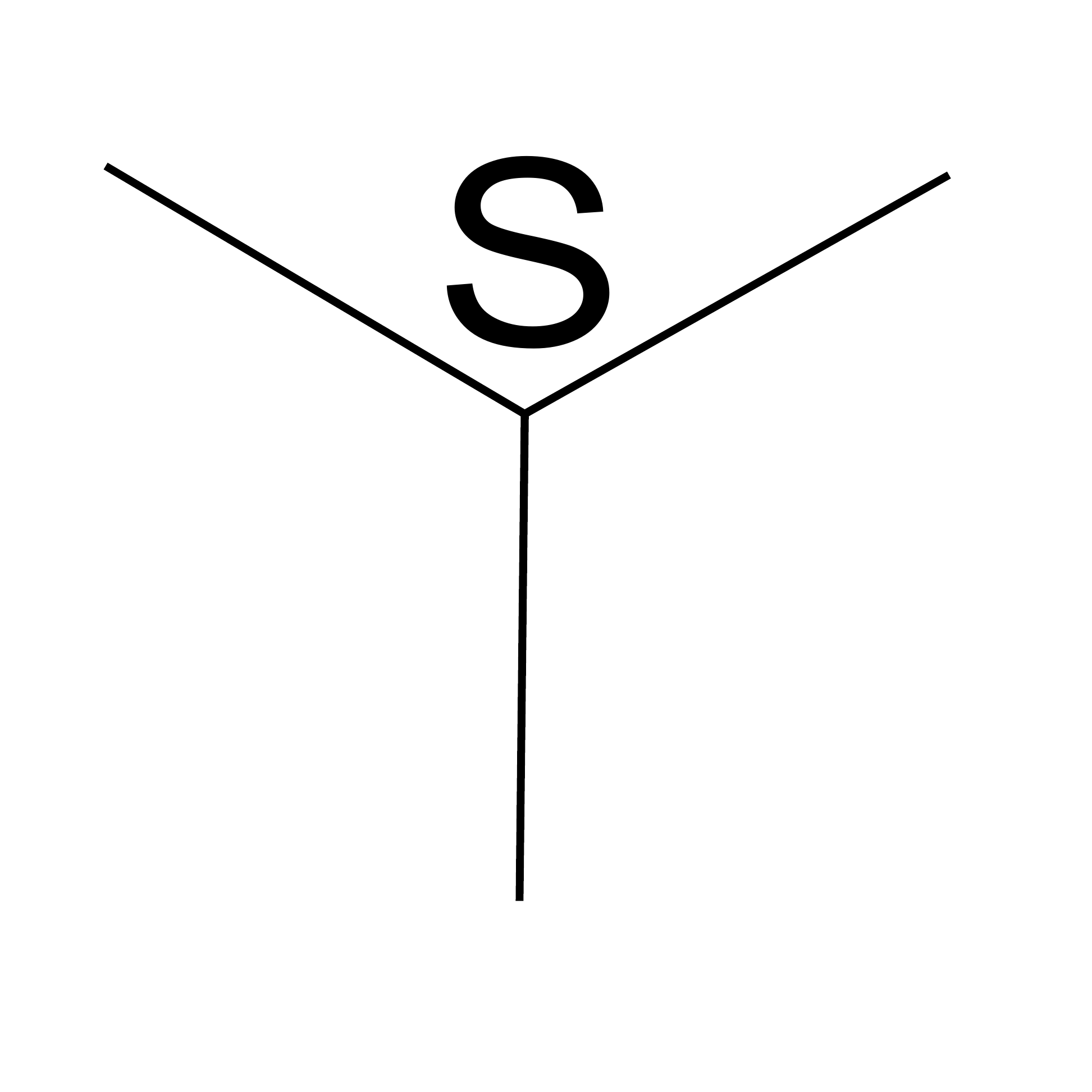}=\omega\gra{wS};$\\
\item[(ii)]$~~\gra{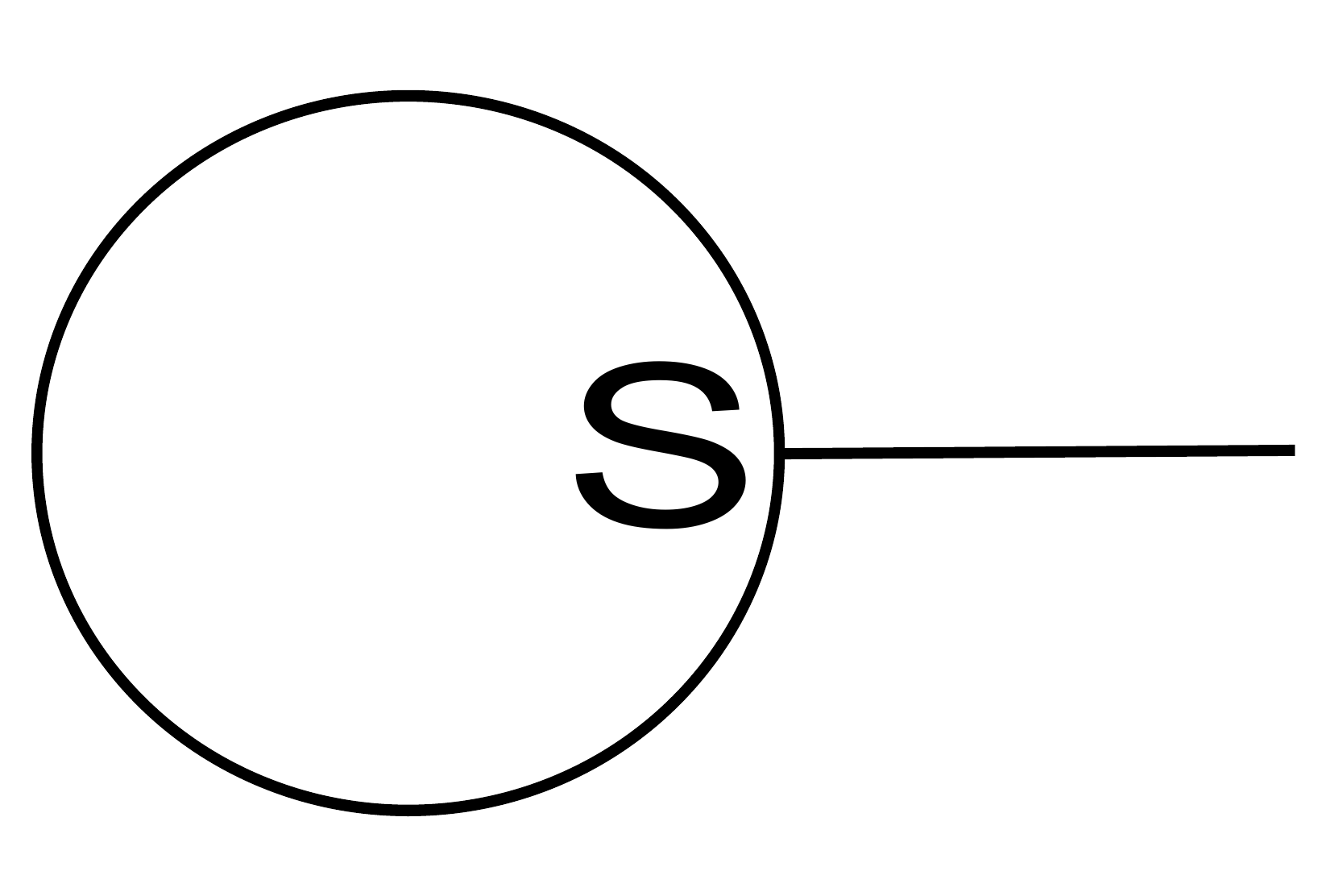}=\gra{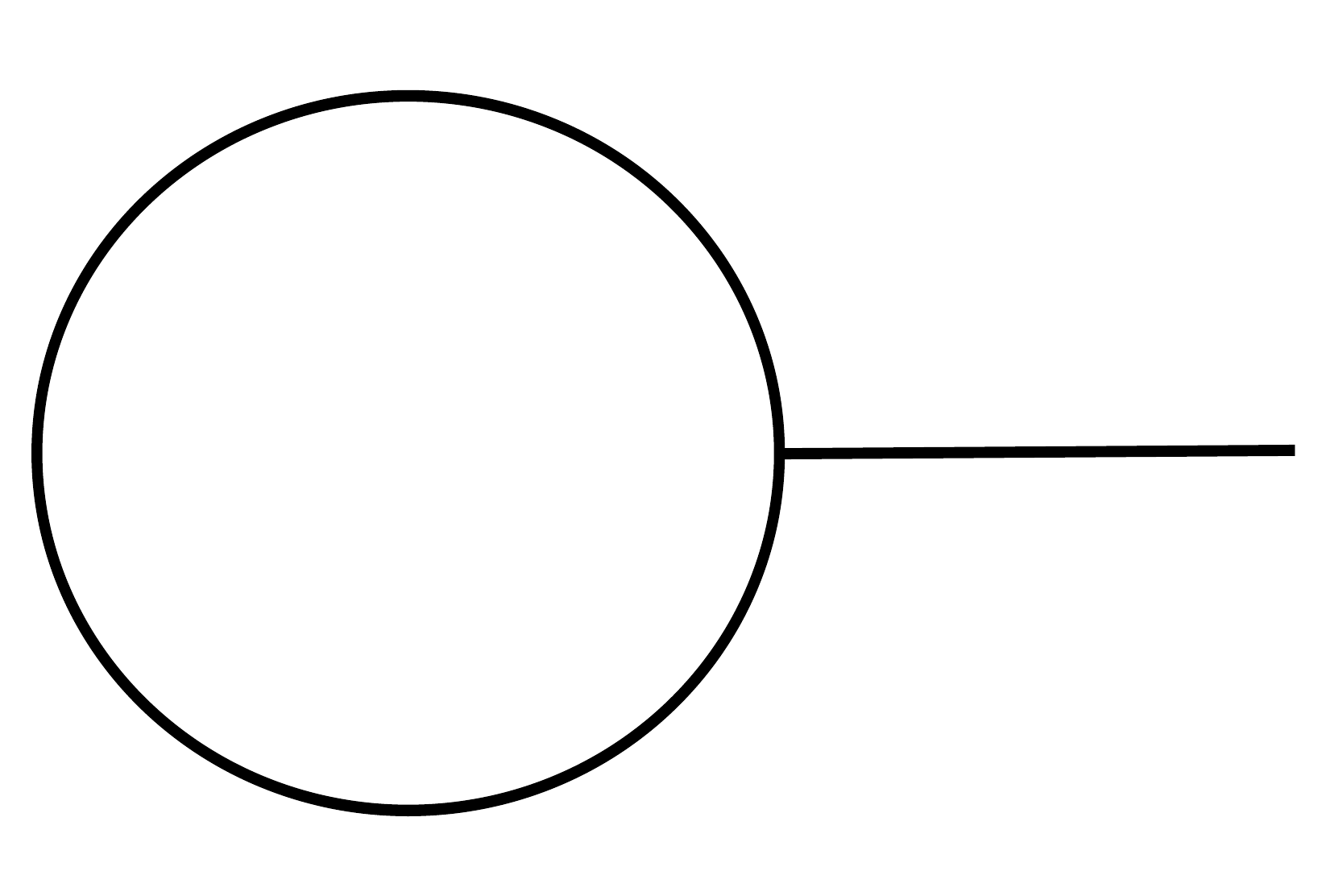}=0;$\\
\item[(iii)]$~\gra{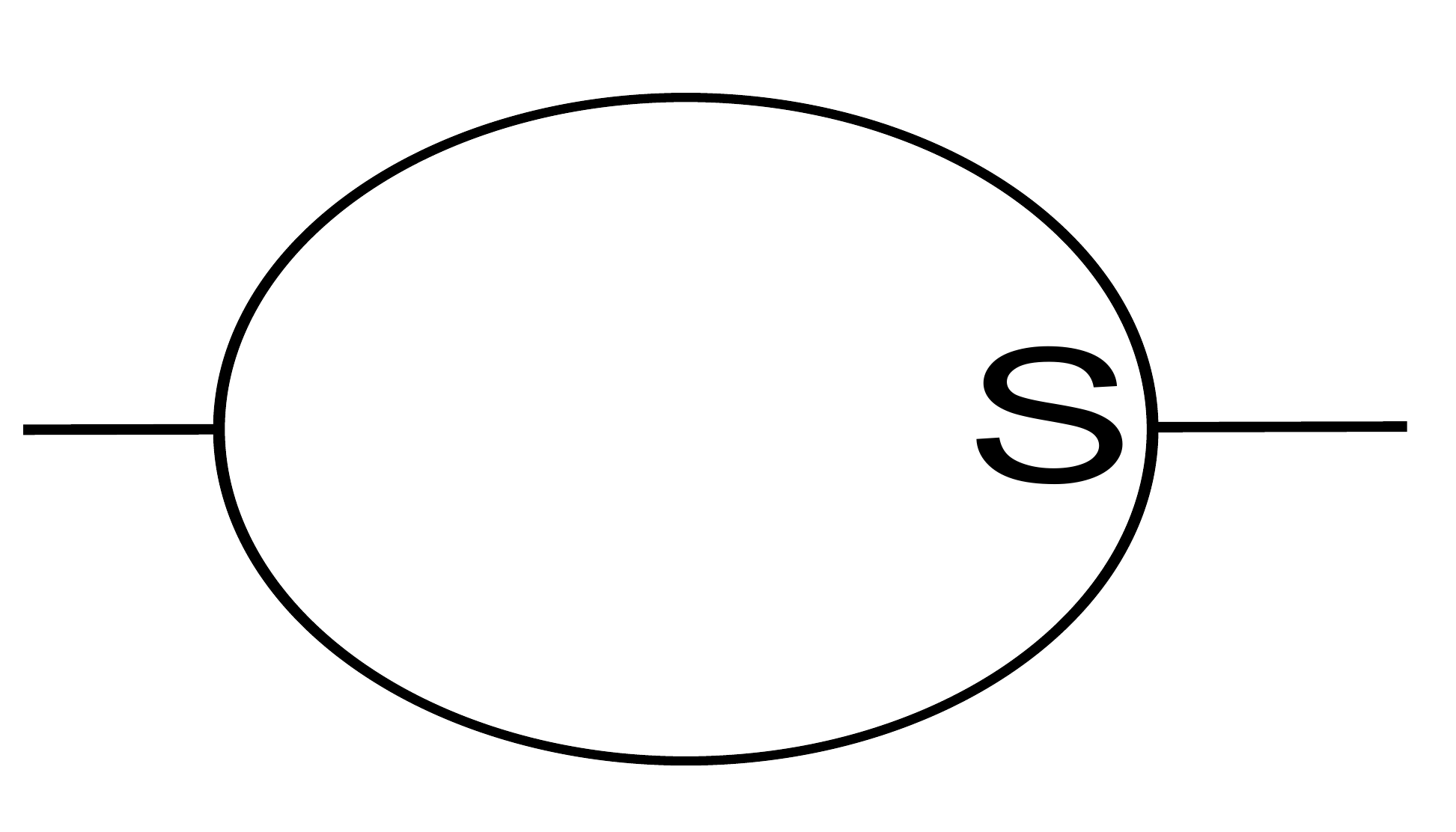}=0;$\\
\item[(iv)]$~~\gra{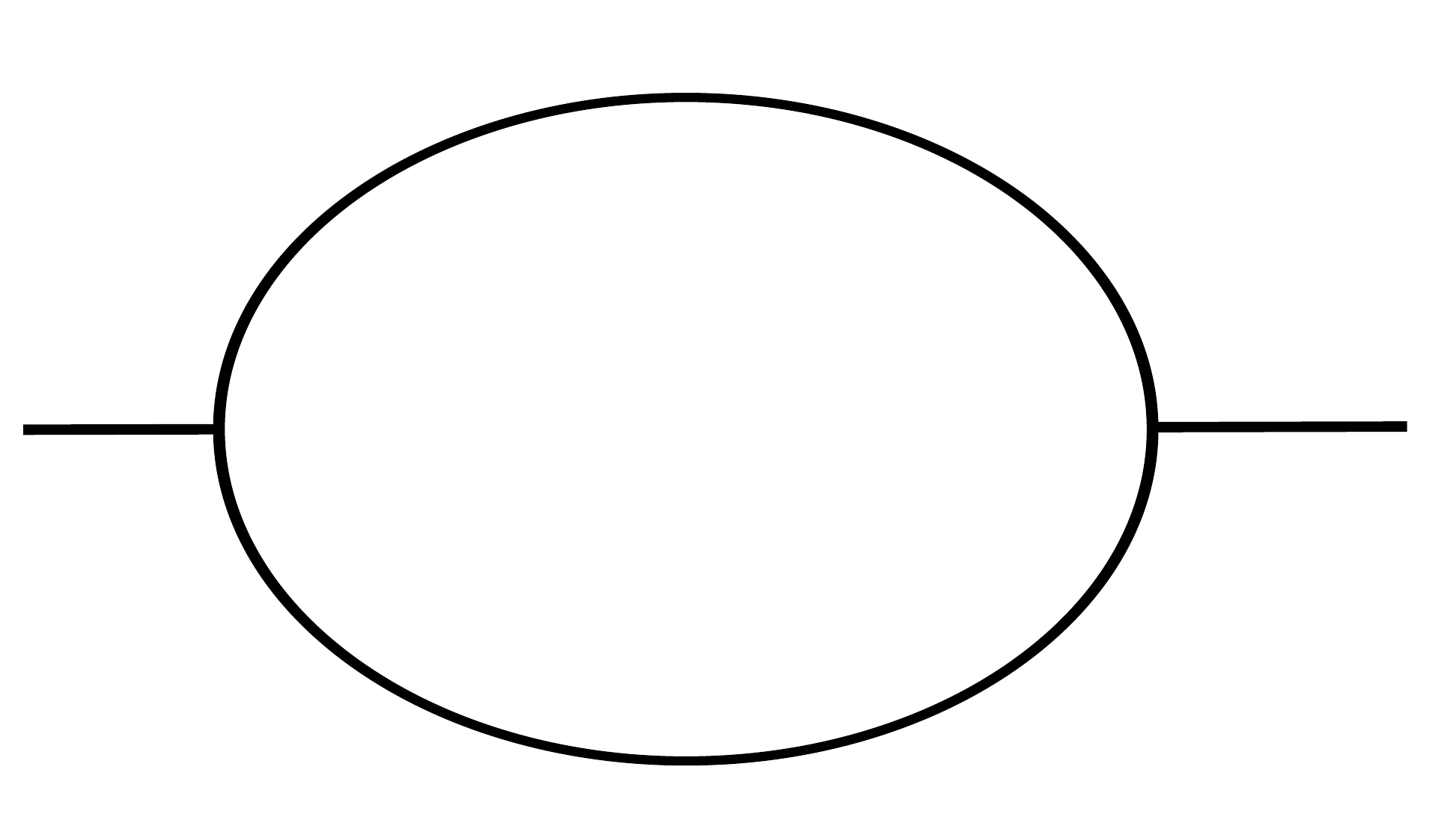}=\frac{\delta^2-2}{\delta}\gra{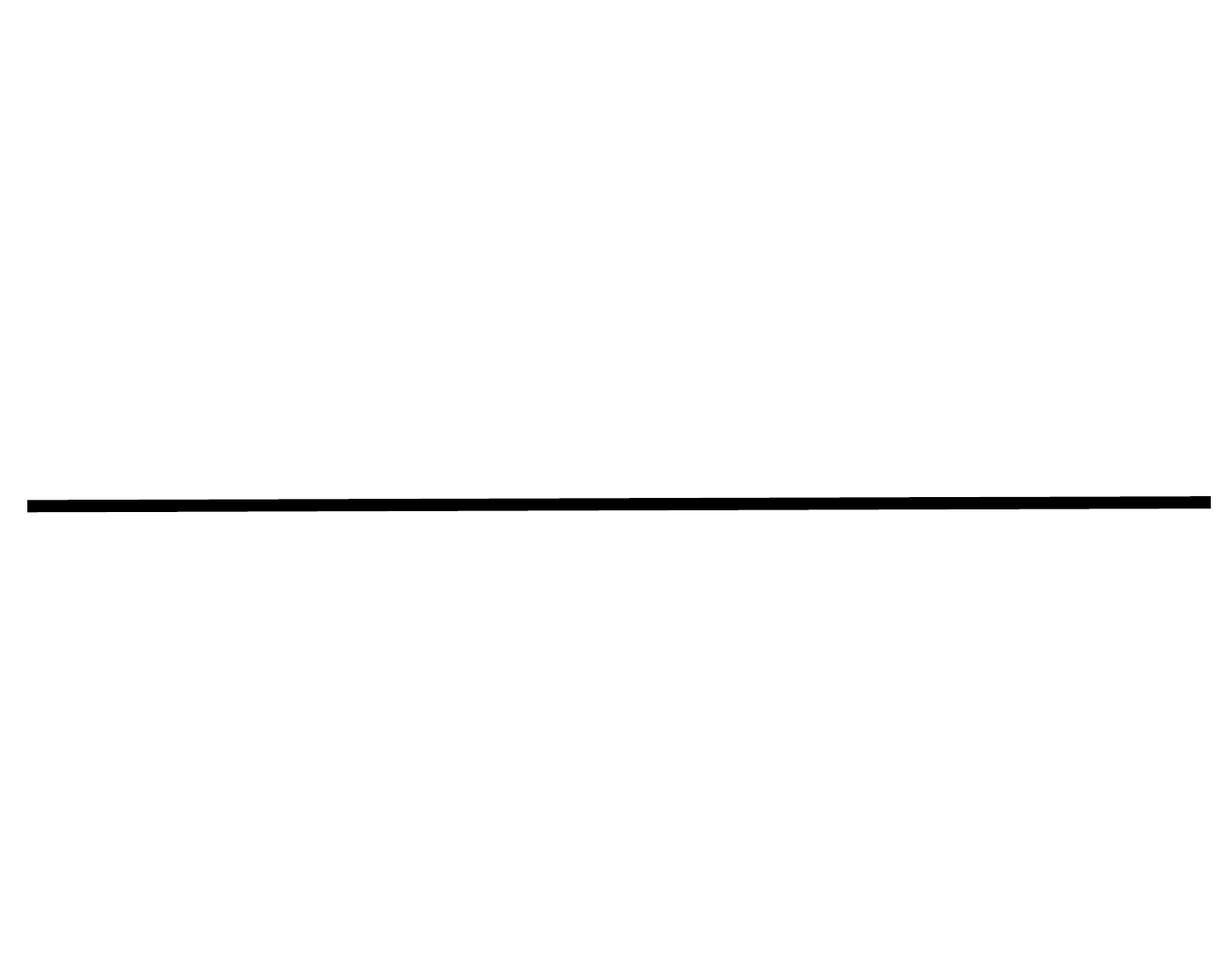};$\\
\item[(v)]$~~~\gra{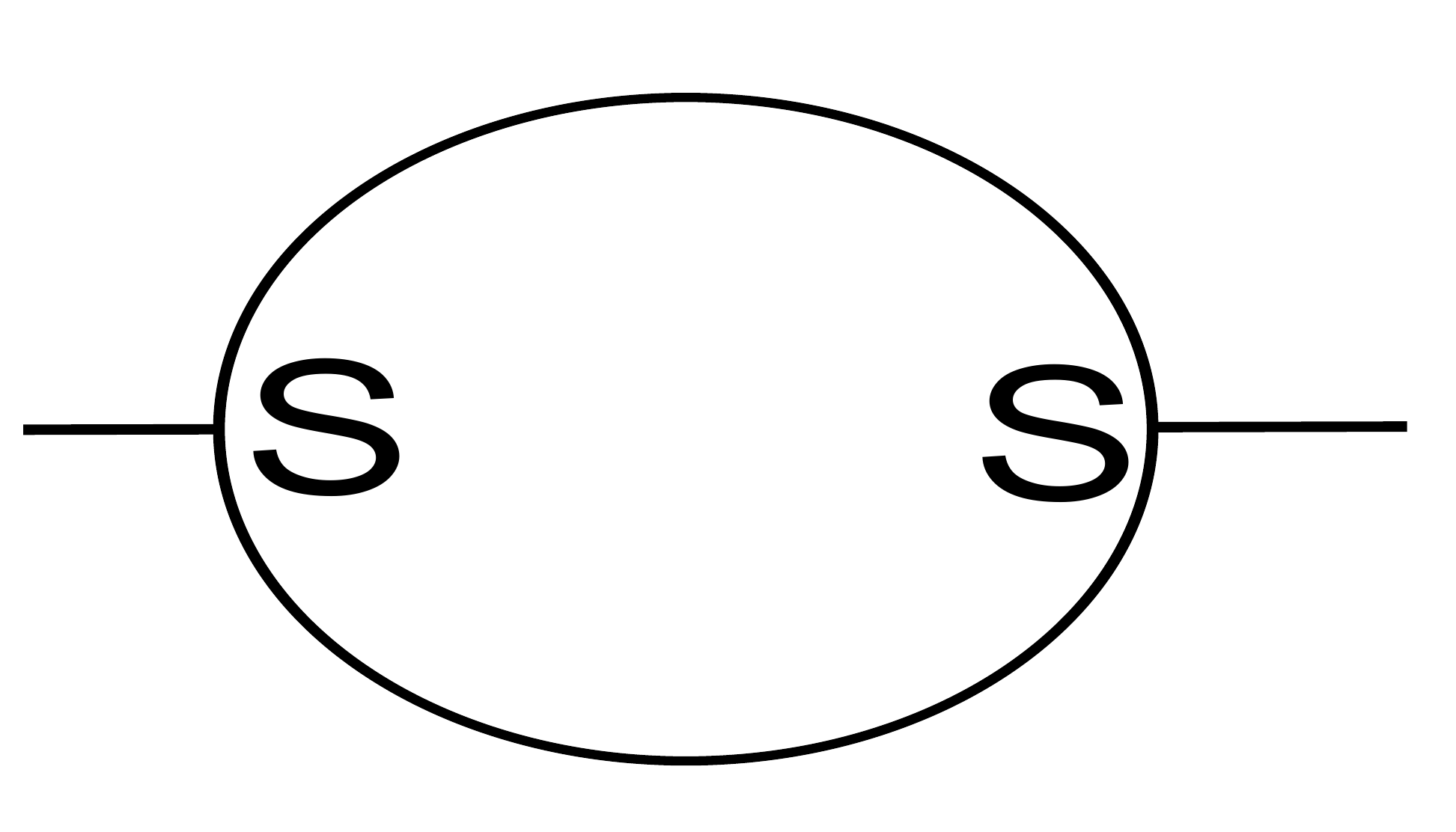}=\gamma\frac{\delta^3-2\delta}{\delta^2-1}\gra{VI};$\\
\item[(vi)]$~~\graa{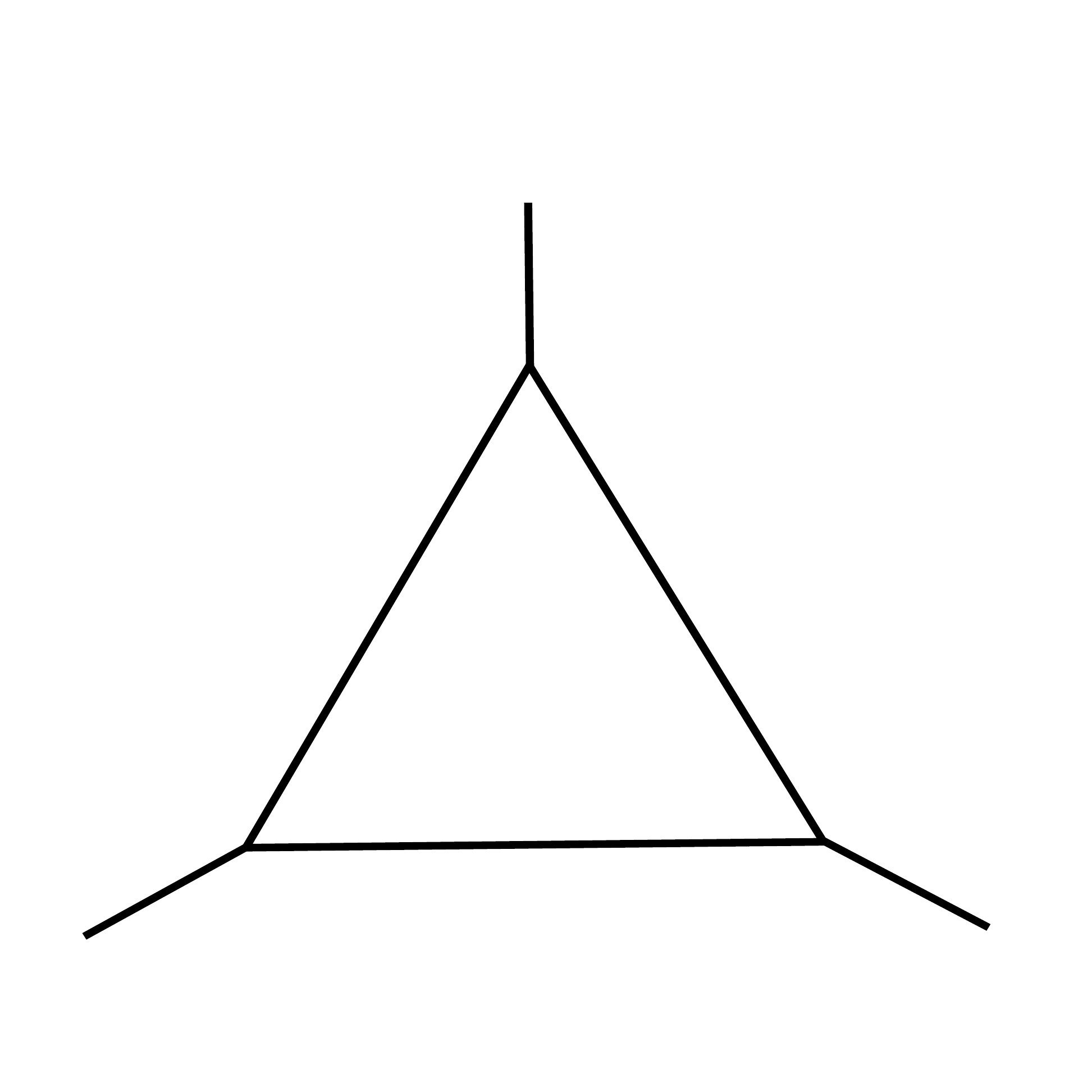}=\frac{\delta^2-3}{\delta}\graa{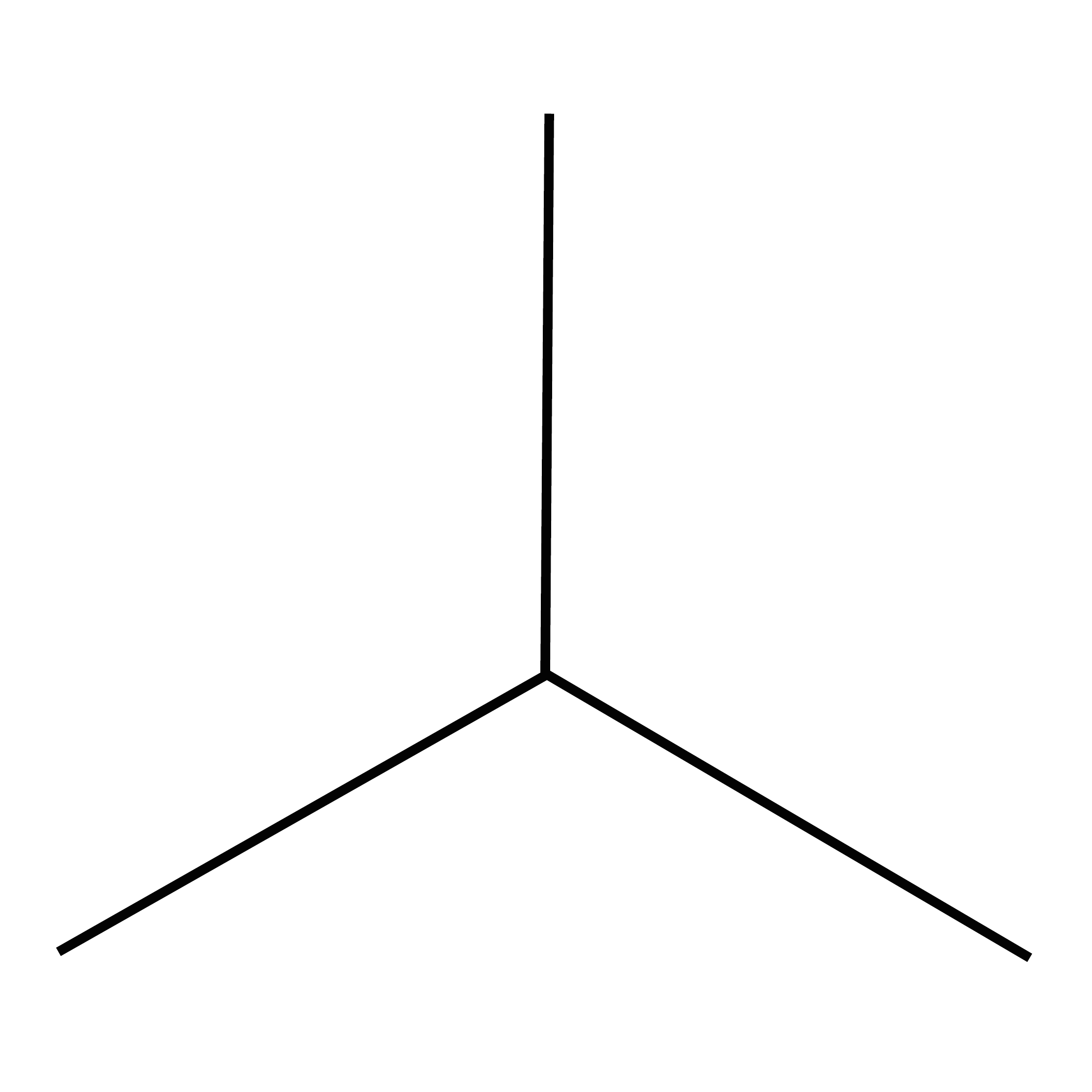};$\\
\item[(vii)]$~\graa{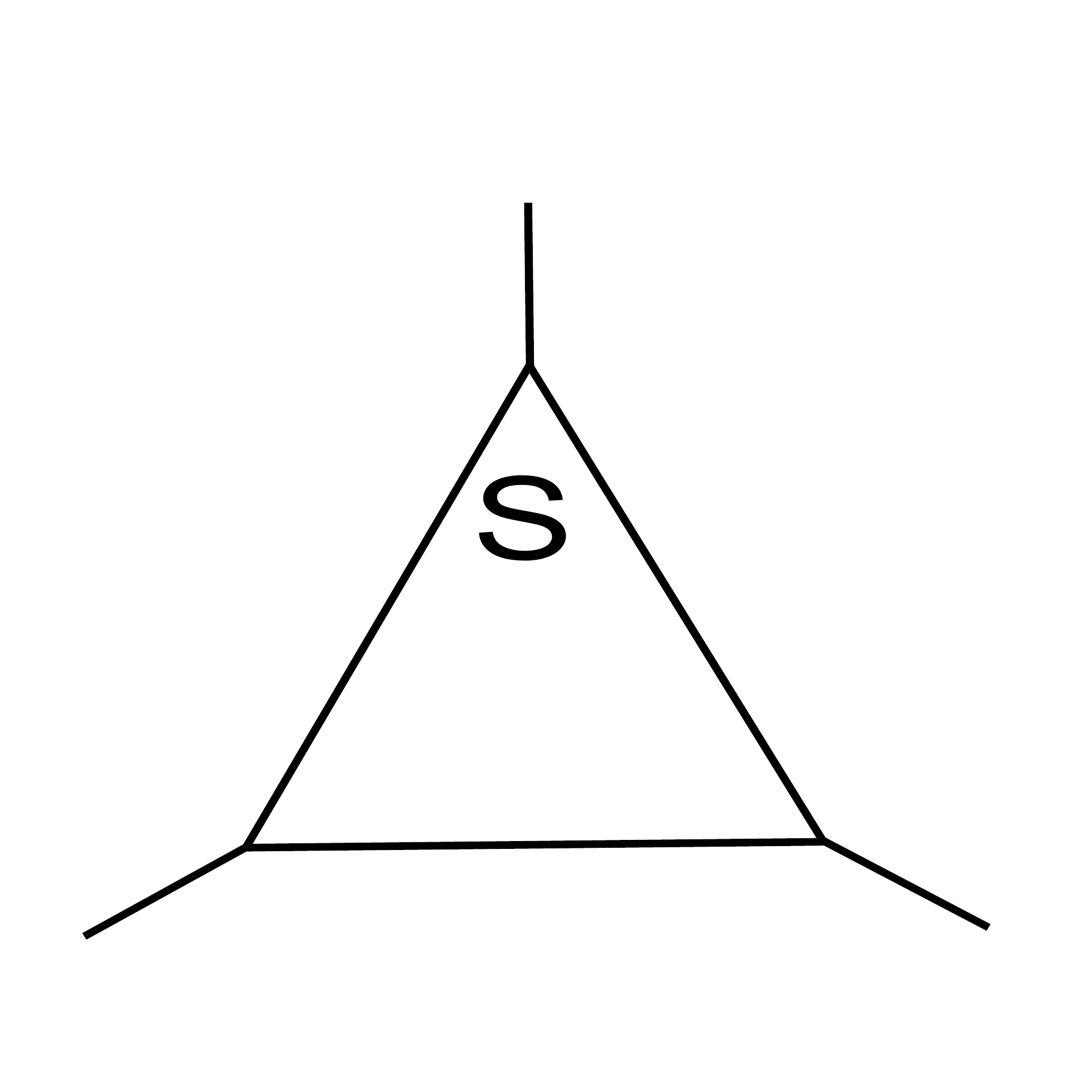}=-\frac{1}{\delta}\graa{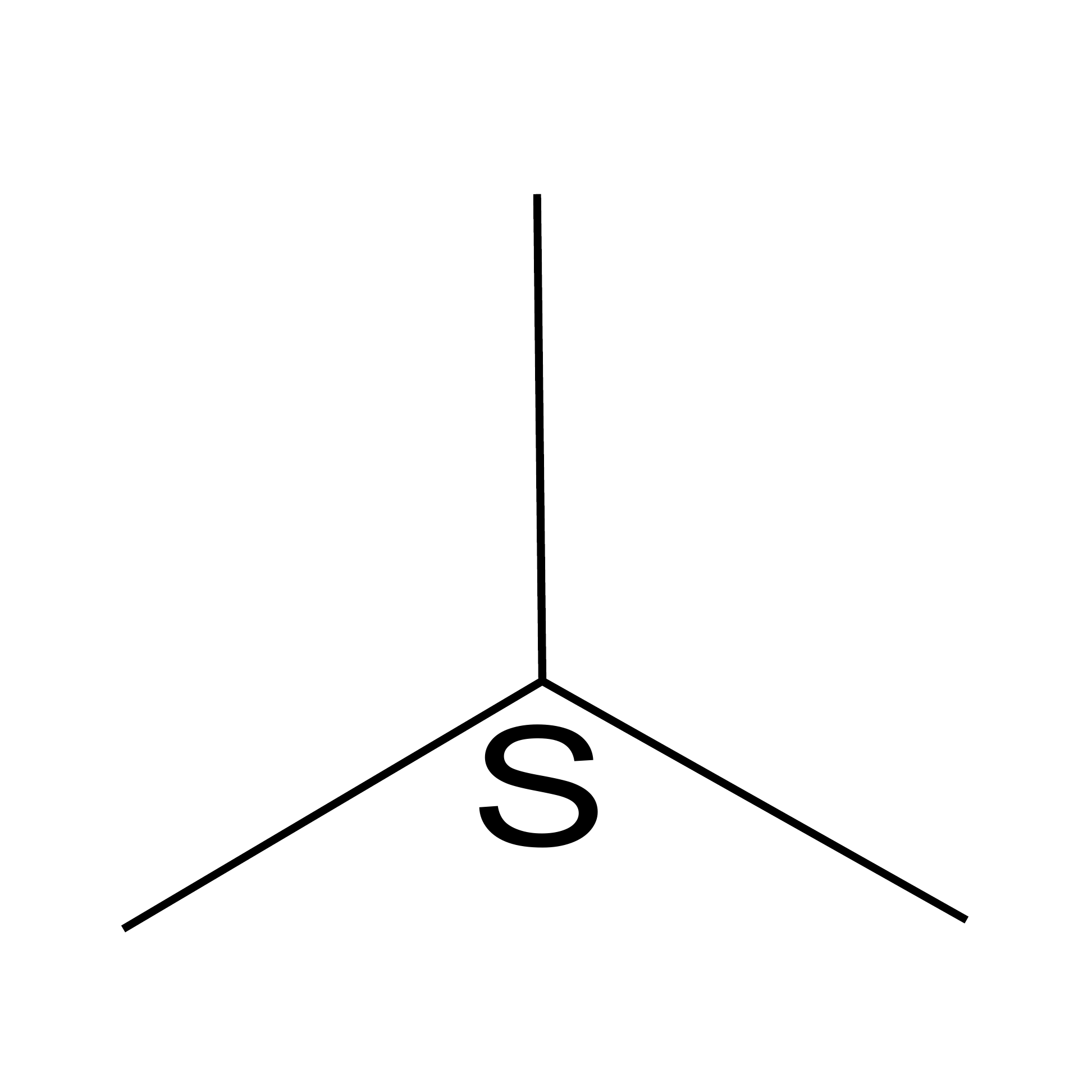};$\\
\item[(viii)]$\graa{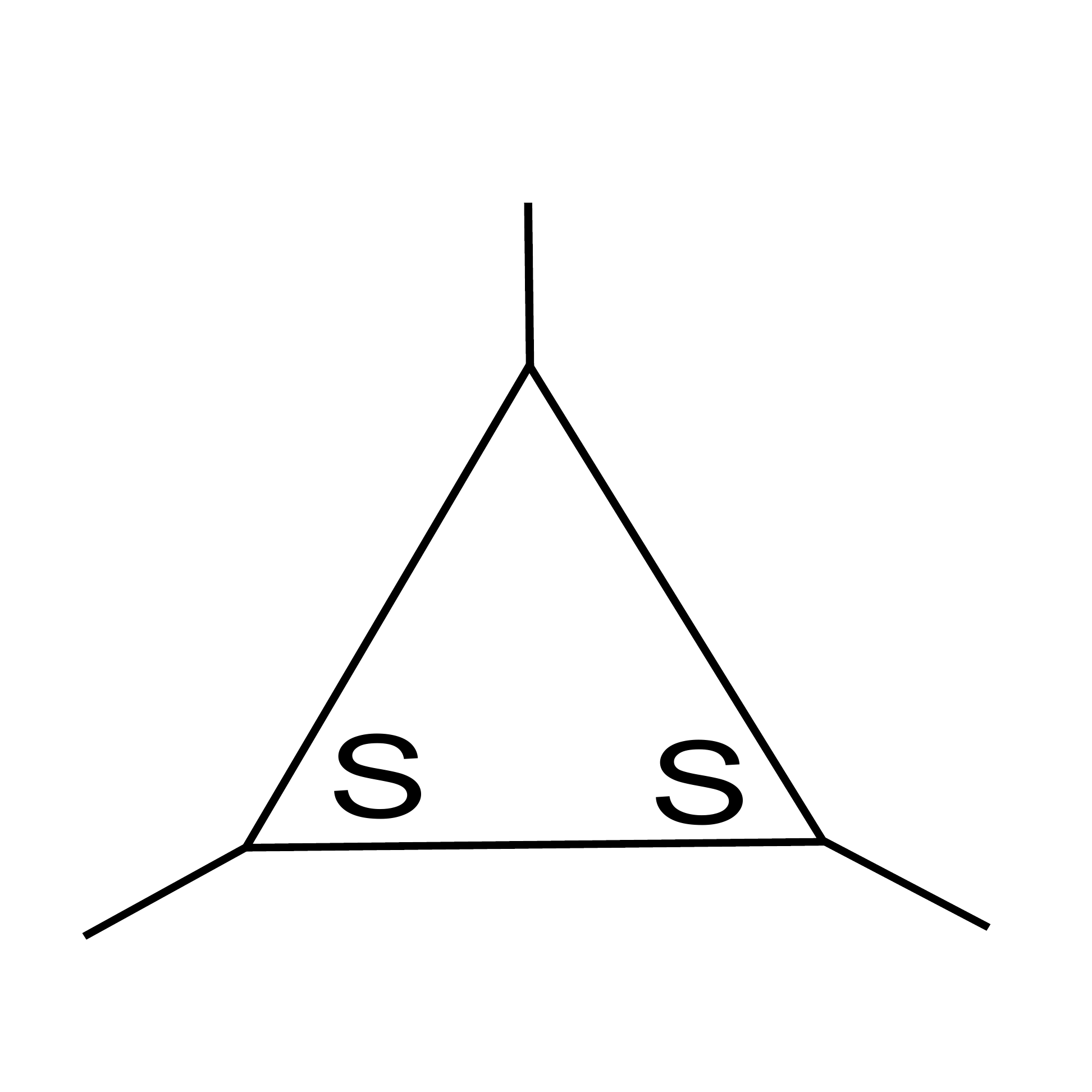}=(\gamma-1)\graa{SV}-\gamma\frac{\delta}{\delta^2-1}\graa{TLTRIV};$\\
\item[(ix)]$~~\graa{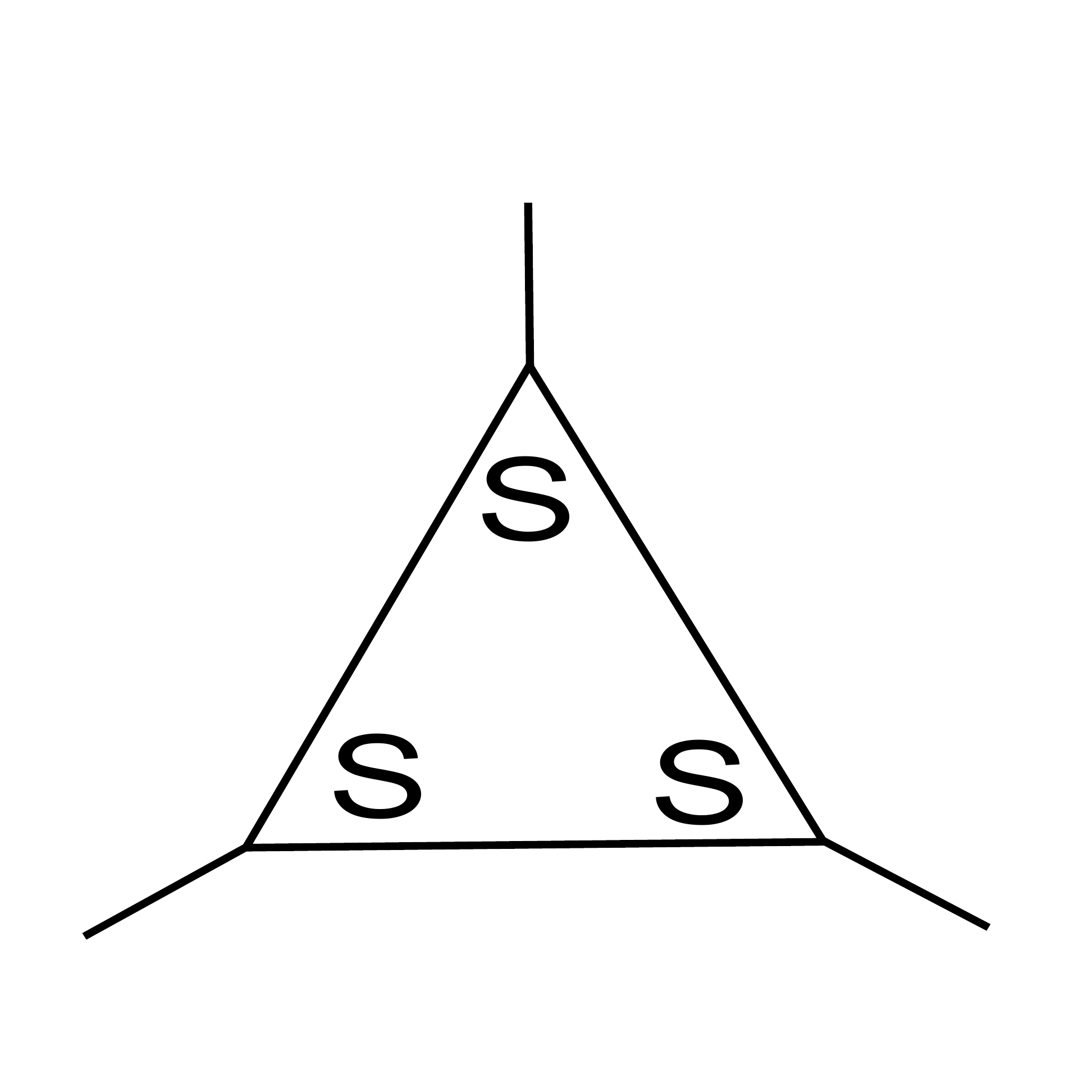}=\delta_{\omega,1}\varepsilon\graa{SV}+\gamma(\gamma-1)\frac{\delta^2}{\delta^2-1}\graa{TLTRIV}$, where $\varepsilon\in\mathbb{C}$ is an undetermined variable.
\end{itemize}
\end{lemma}
\begin{proof}
The relations follow from Relations \ref{Equ:S1}--\ref{Equ:S3}.
Since \gra{TL} and \gra{wS} form an orthogonal basis of the $f_2$-cutdown of $\Pl_{3,+}$, one obtains the coefficients in (viii) and (ix) by computing the inner product.
\end{proof}

\subsection{Relations in 4-boxes}\label{Section:4box}
We proceed to discuss the relations in the $f_2$-cutdown of $\mathscr{P}_{4,+}$.

\begin{notation}
We define
$$\graa{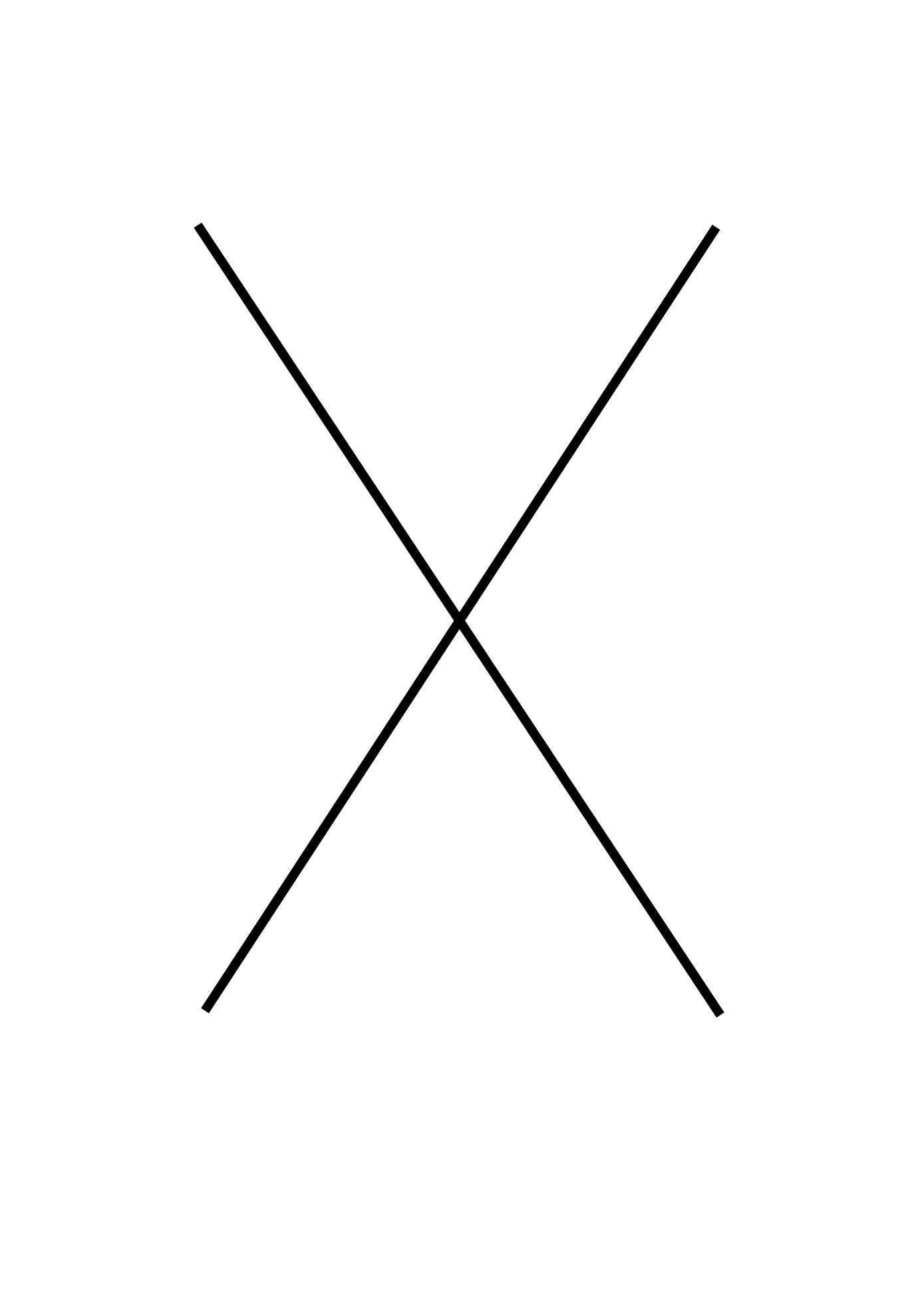}:=\graa{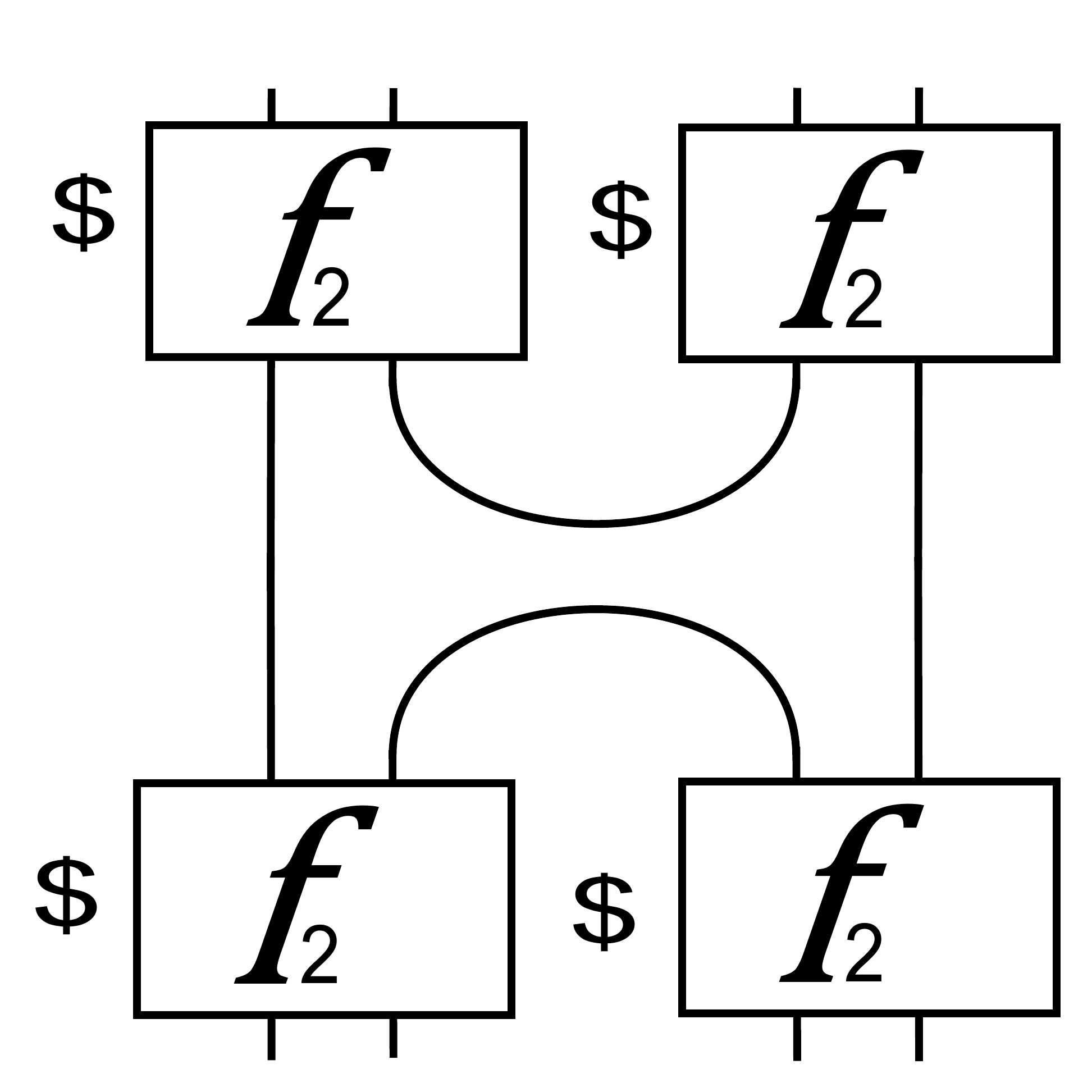} \;.$$
\end{notation}

\begin{lemma}\label{Lem:Generic-cutdown-4-basis}
The set $B$ defined as
$$
\left\{\graa{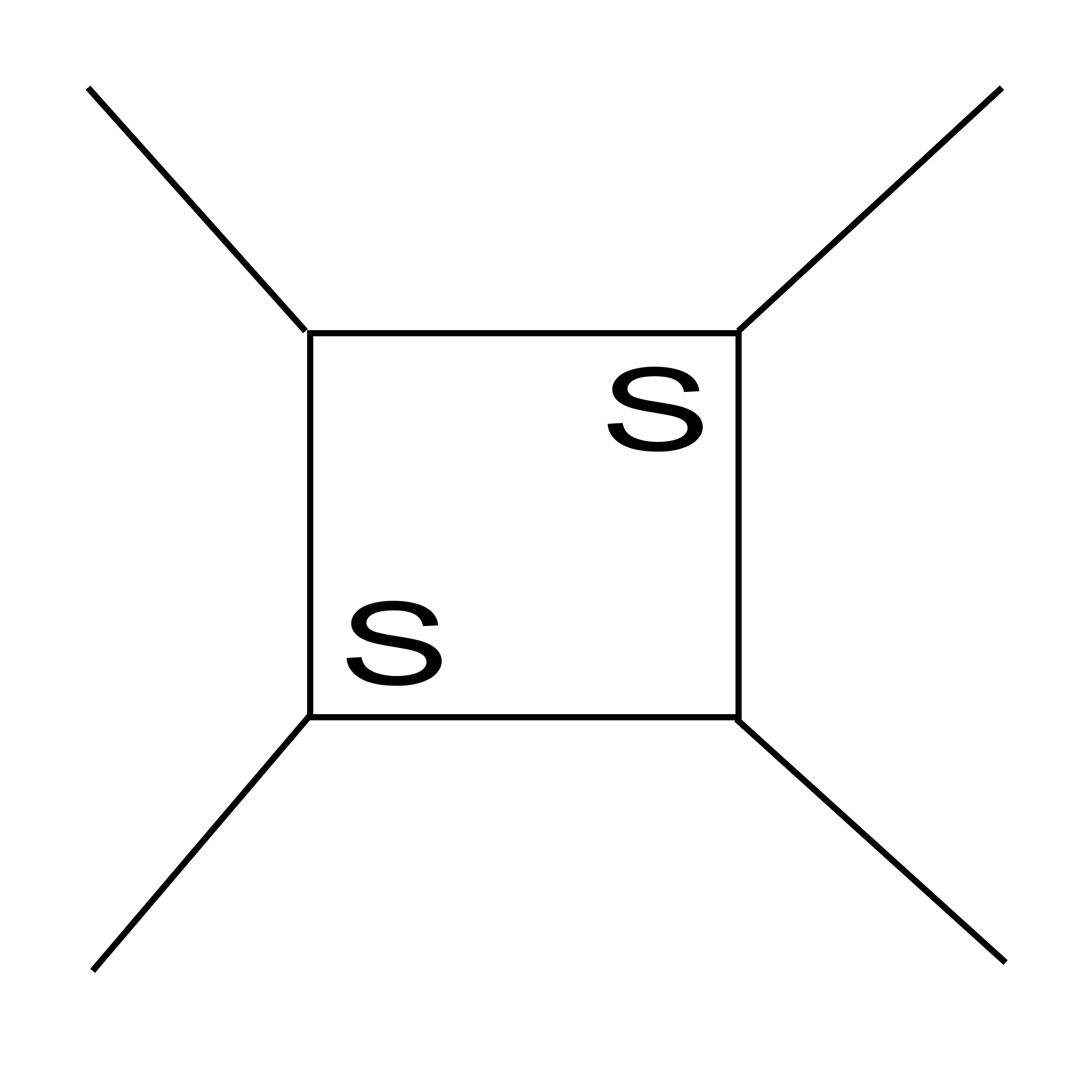},\graa{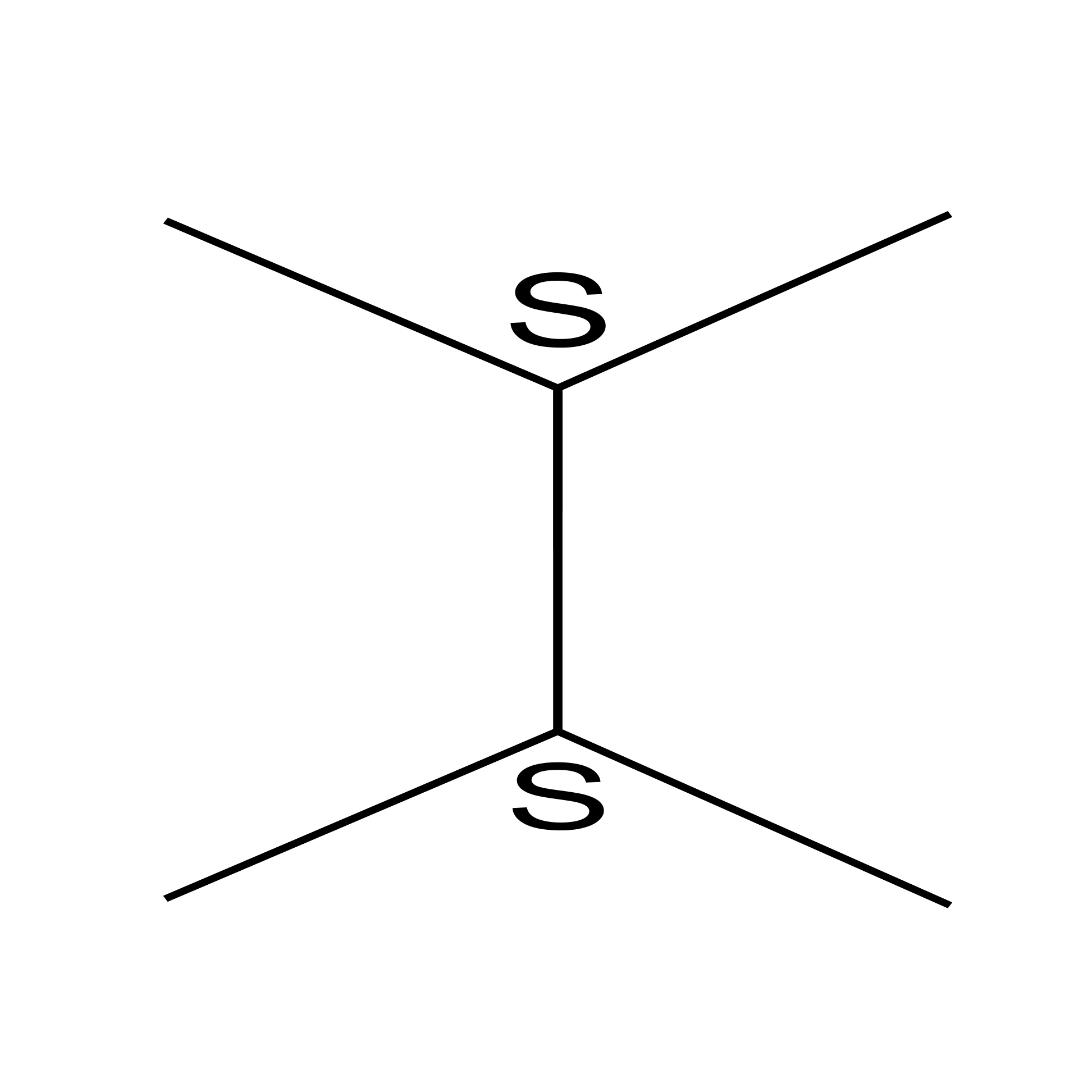},\graa{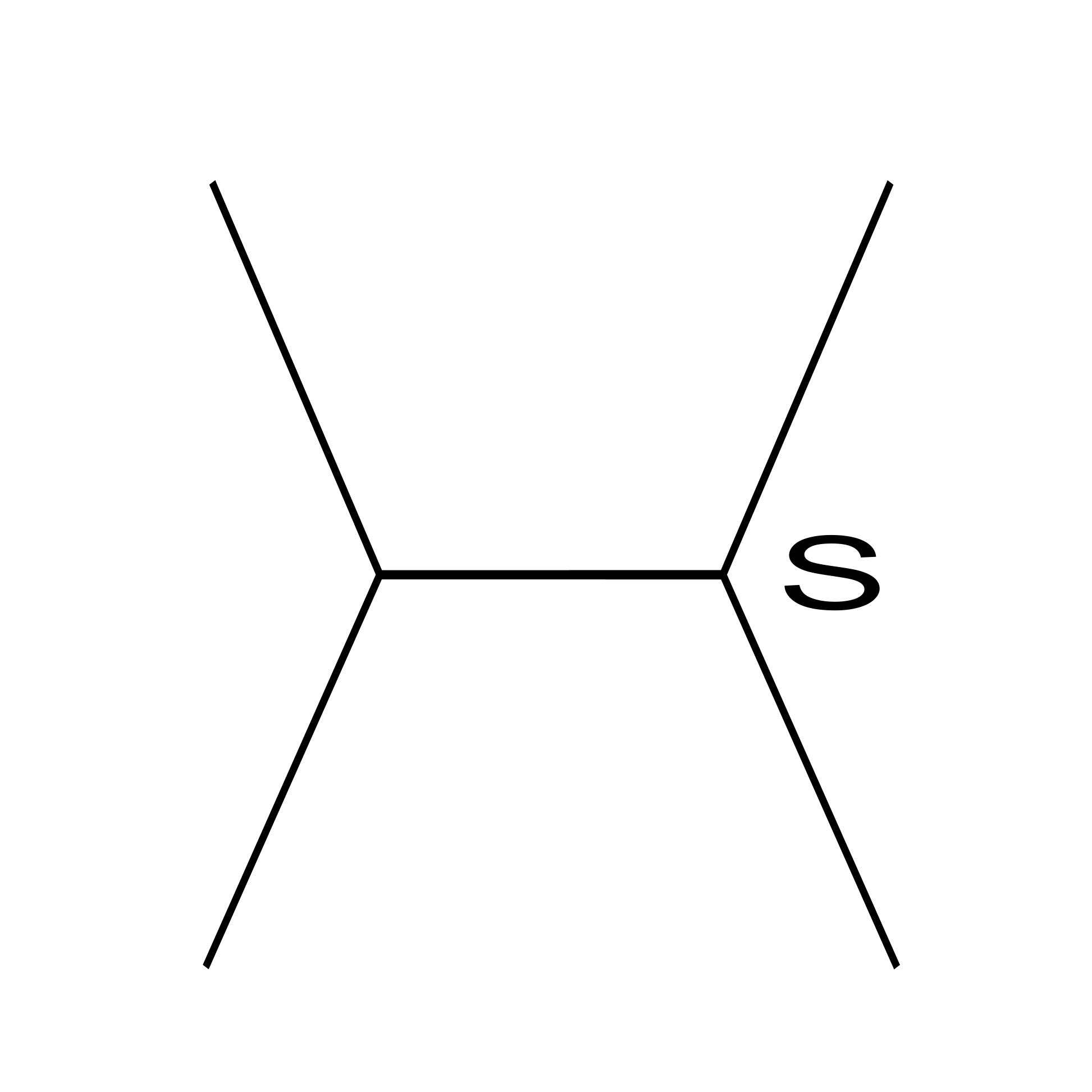},\graa{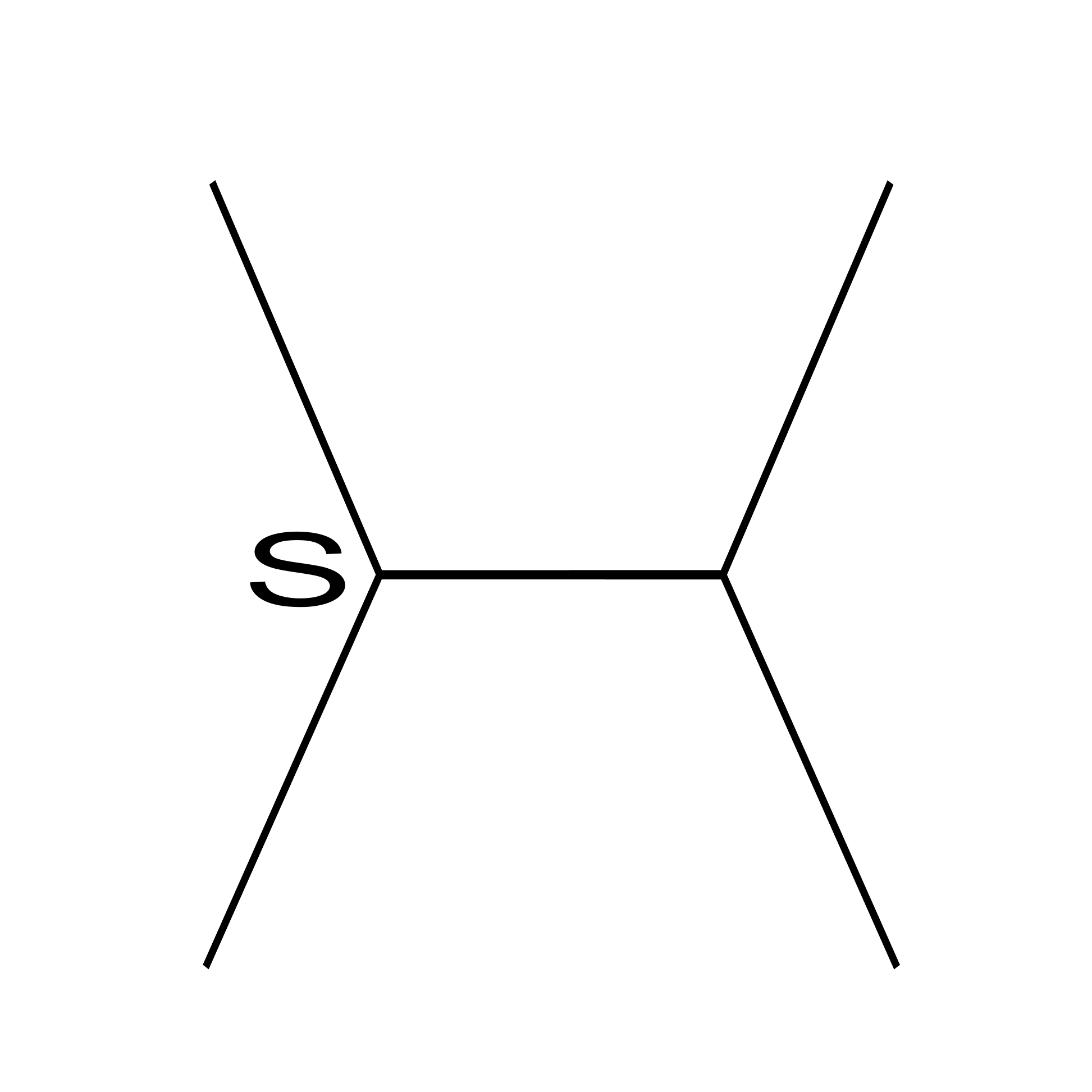},
\graa{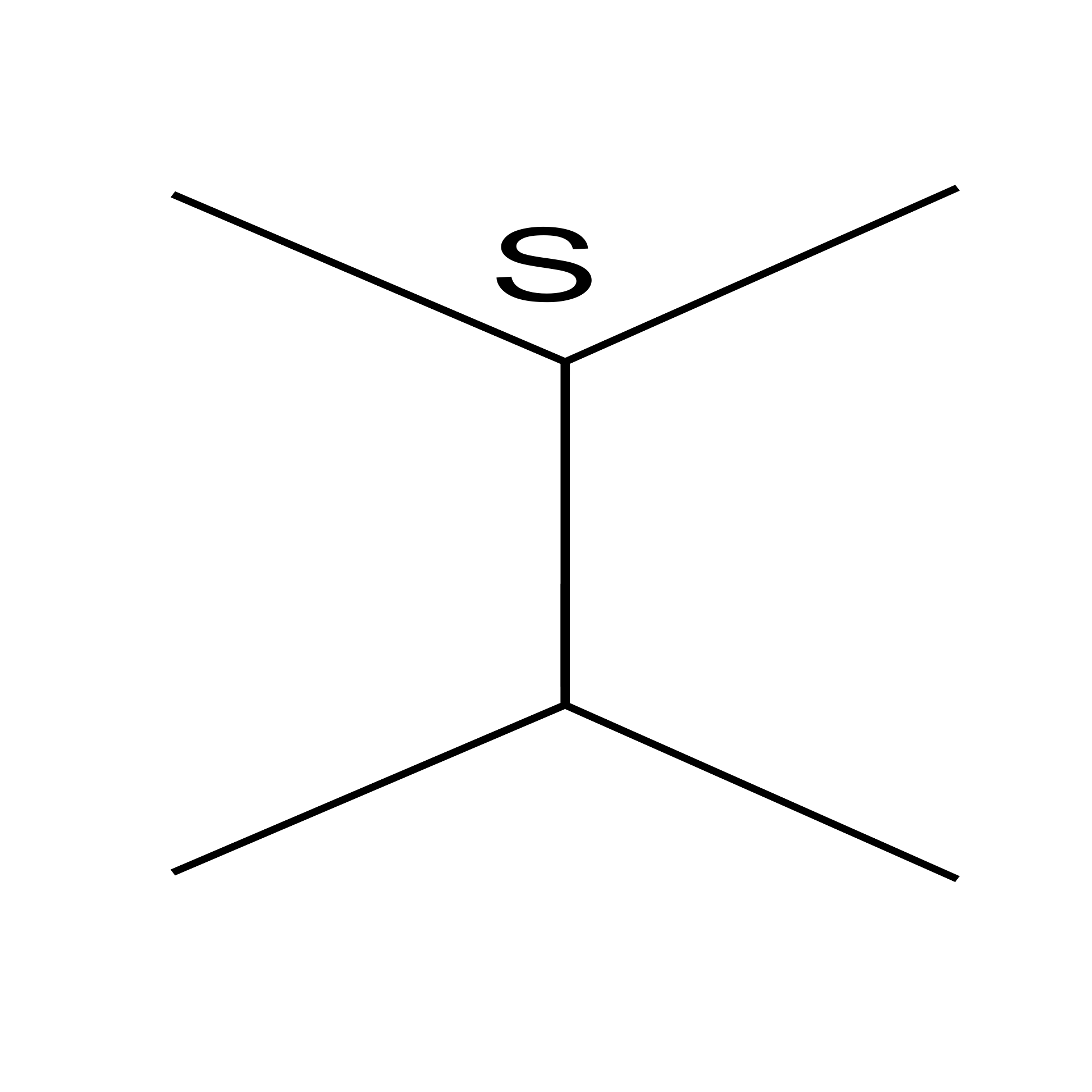},\graa{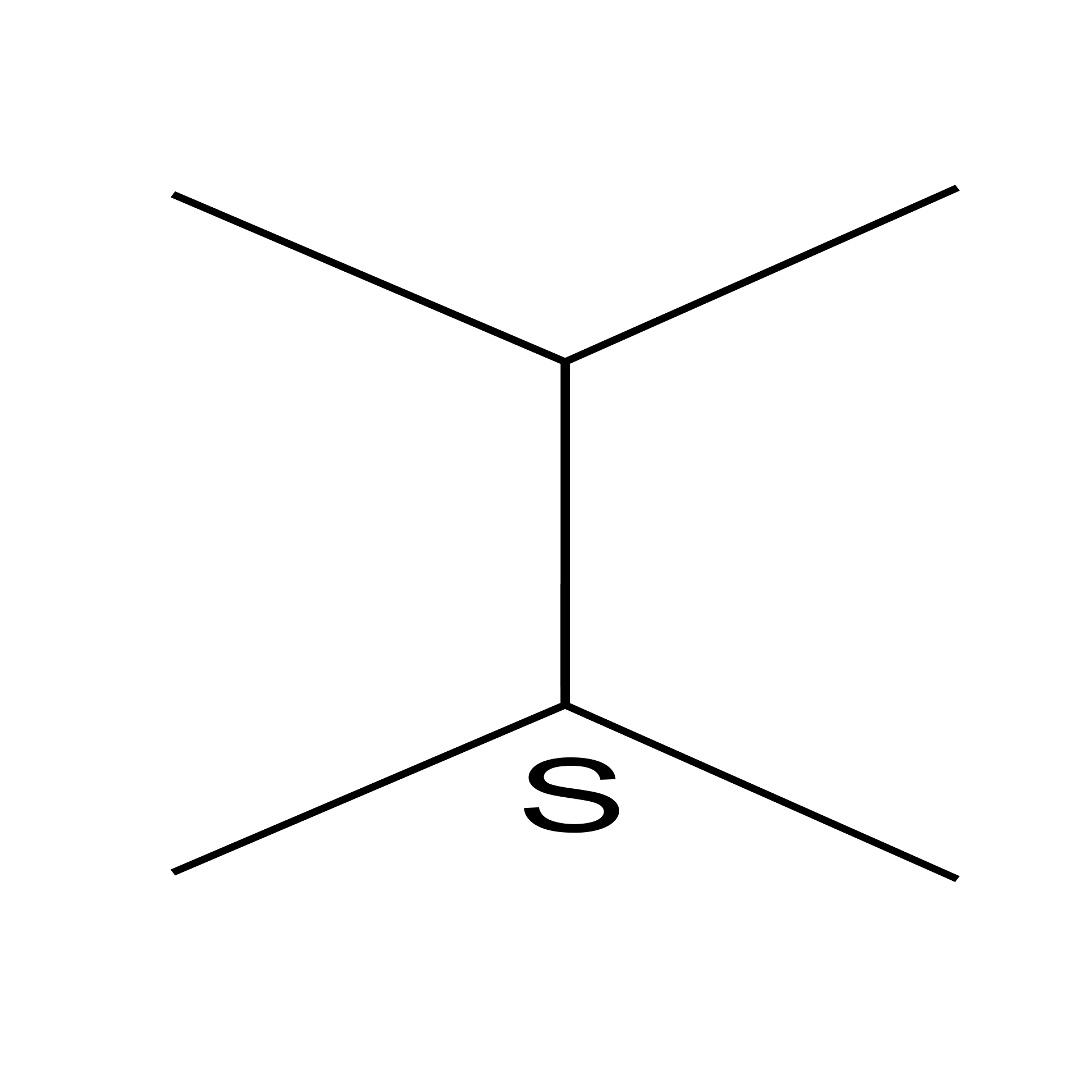},\graa{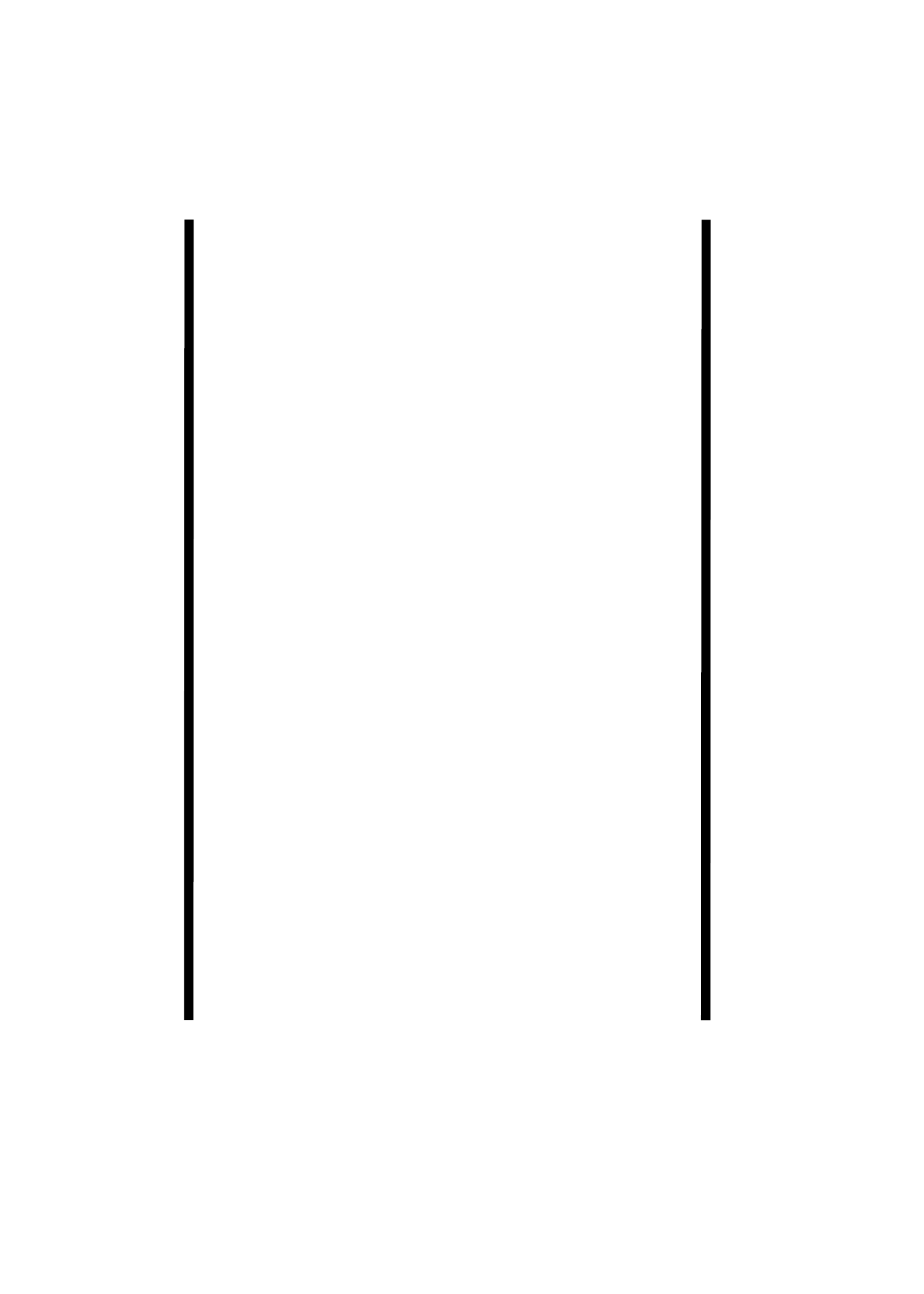},\graa{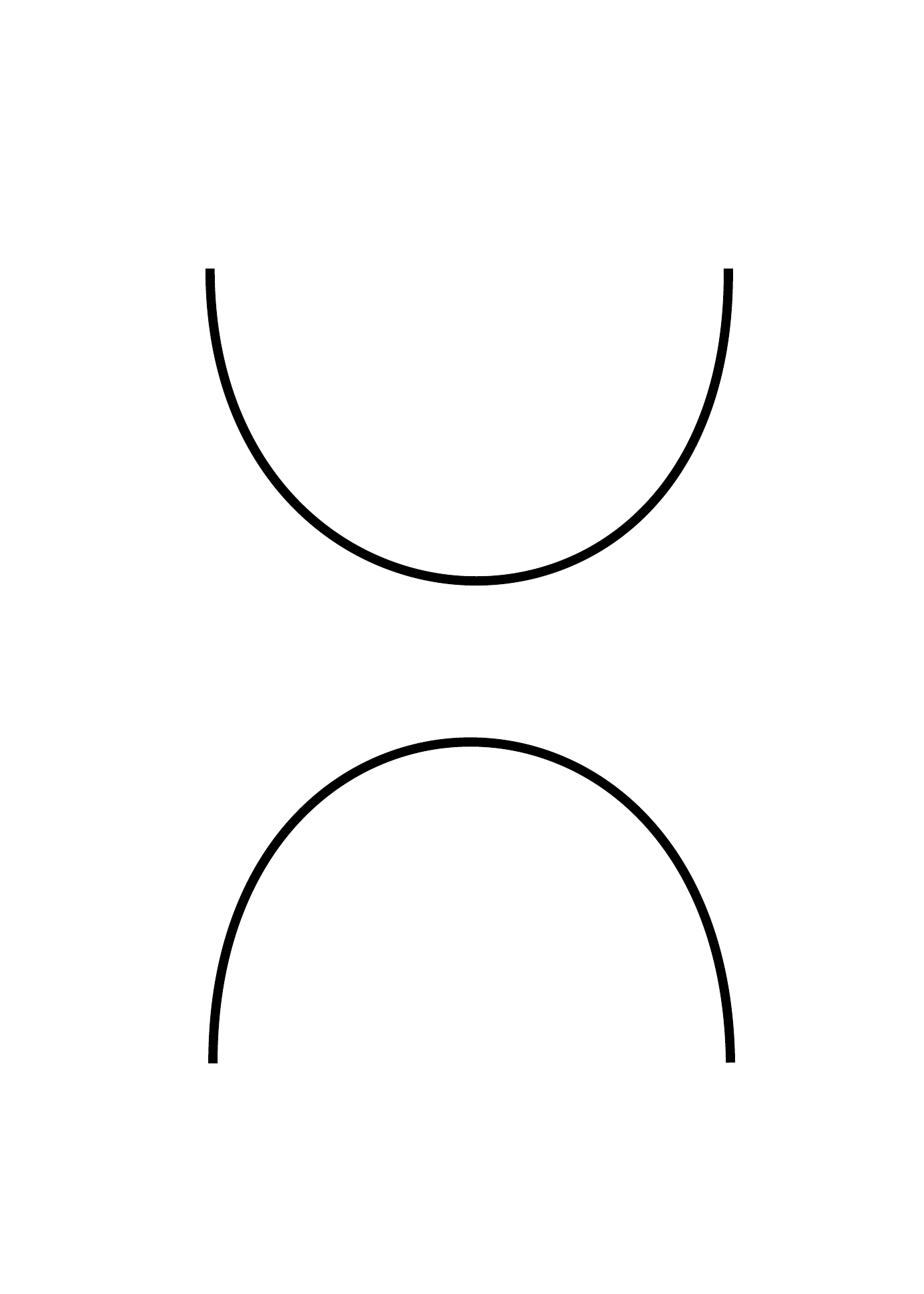},\graa{TLC}\right\}
$$
forms a basis of the $f_2$-cutdown of $\Pl_{4,+}$.
\end{lemma}
\begin{proof}
Note that the $f_2$-cutdown of the 24 diagrams in the basis of $\mathscr{P}_{4,+}$ in Corollary \ref{Cor:24-basis} is in the linear span of $B$, thus $f_2$-cutdown of $\Pl_{4,+}$ is spanned by $B$.

Suppose
	\begin{multline}\label{Equ:basis}
	a\graa{BOX4TLSTLS}+b\graa{ISS}+c\graa{HTLS}+d\graa{HSTL}+e\graa{ISTL}\\
	+f\graa{ITLS}+g\graa{TLID}+h\graa{TLE}+i\graa{TLC}=0,
	\end{multline}
for some $a,b,c,d,e,f,g,h,i\in\mathbb{C}$.
	Recall that each string in Equation \eqref{Equ:basis} represents \gra{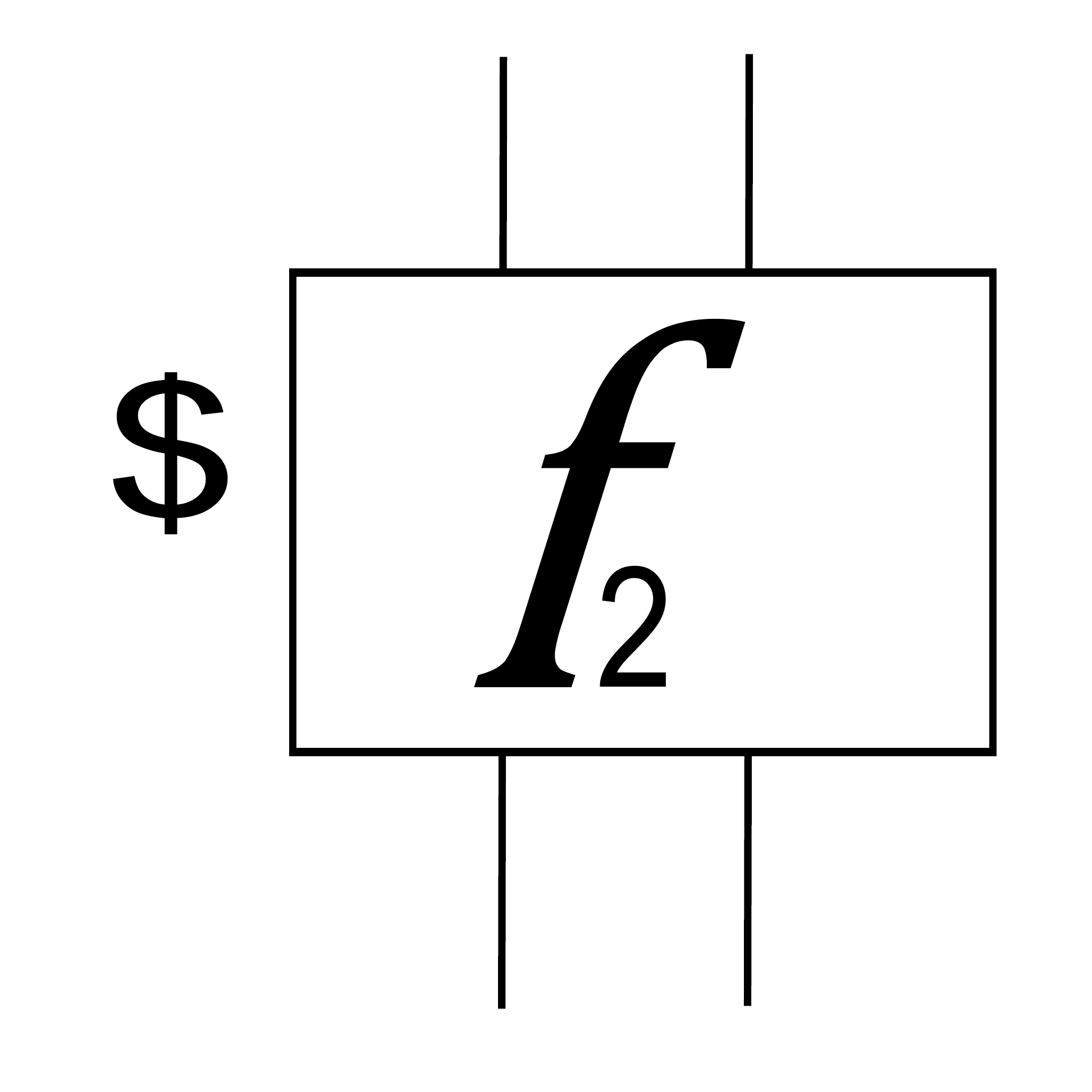} in $\mathscr{P}_{\bullet}$.
We consider it as an equation in $\mathscr{P}_{4,+}$ and rewrite it in terms of the basis of $\mathscr{P}_{4,+}$ in Corollary \ref{Cor:24-basis}.
The coefficients of $\gra{SH}$ and $\gra{I}$ are $a$ and $b$ respectively, thus $a=b=0$.
Furthermore, the coefficients of  $\scalebox{0.8}{\gra{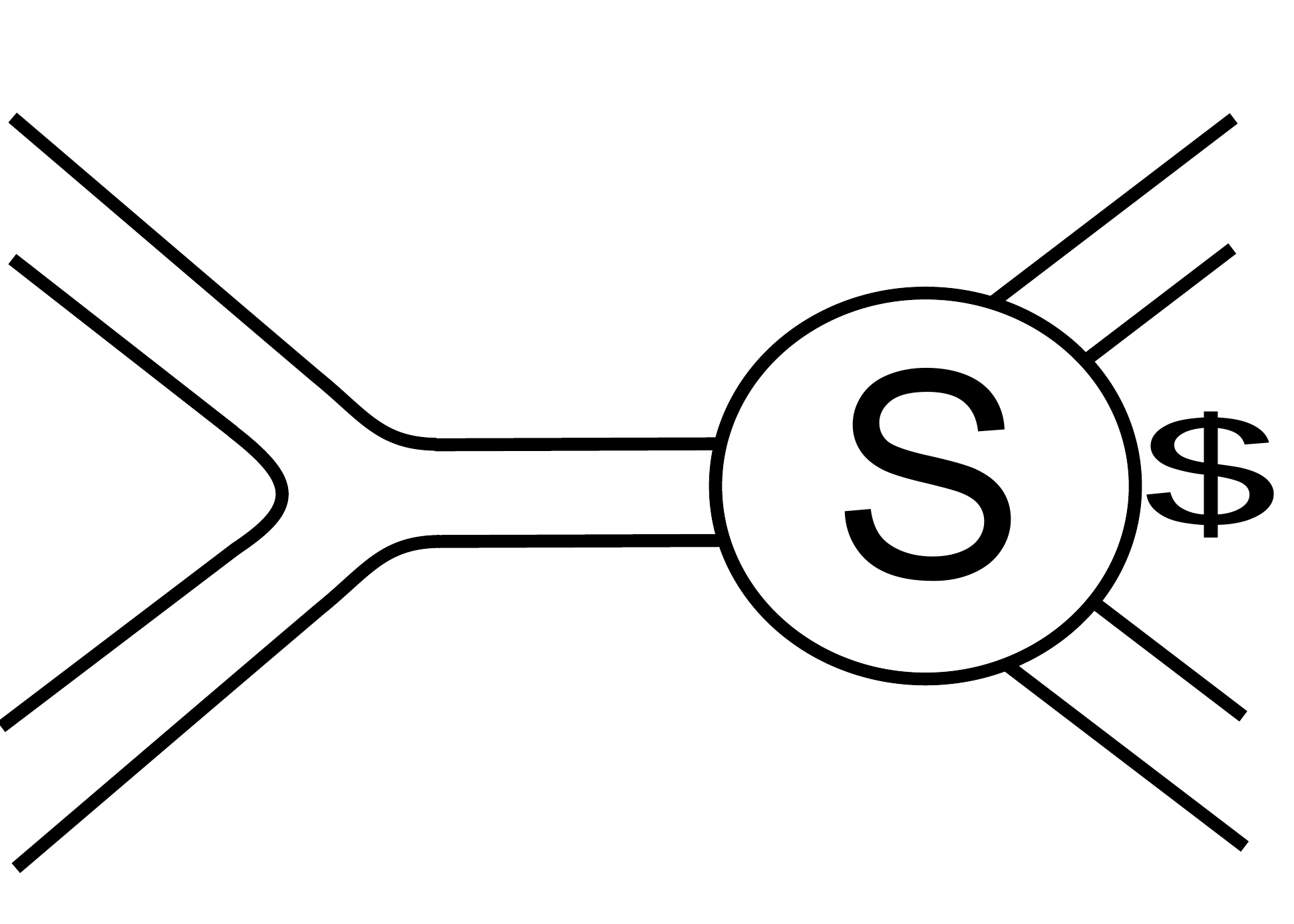}}$, $\scalebox{0.8}{\gra{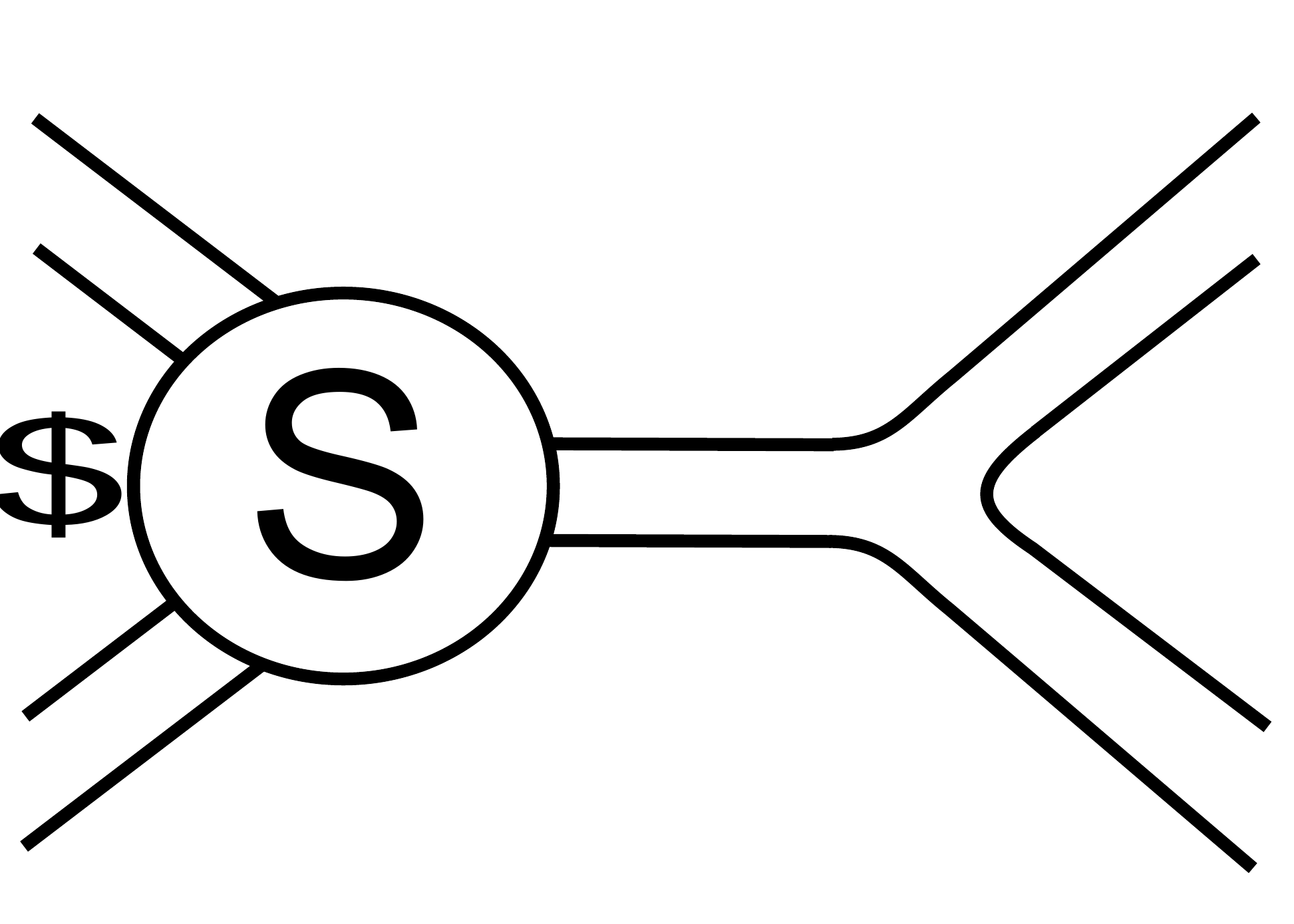}}$, $\gra{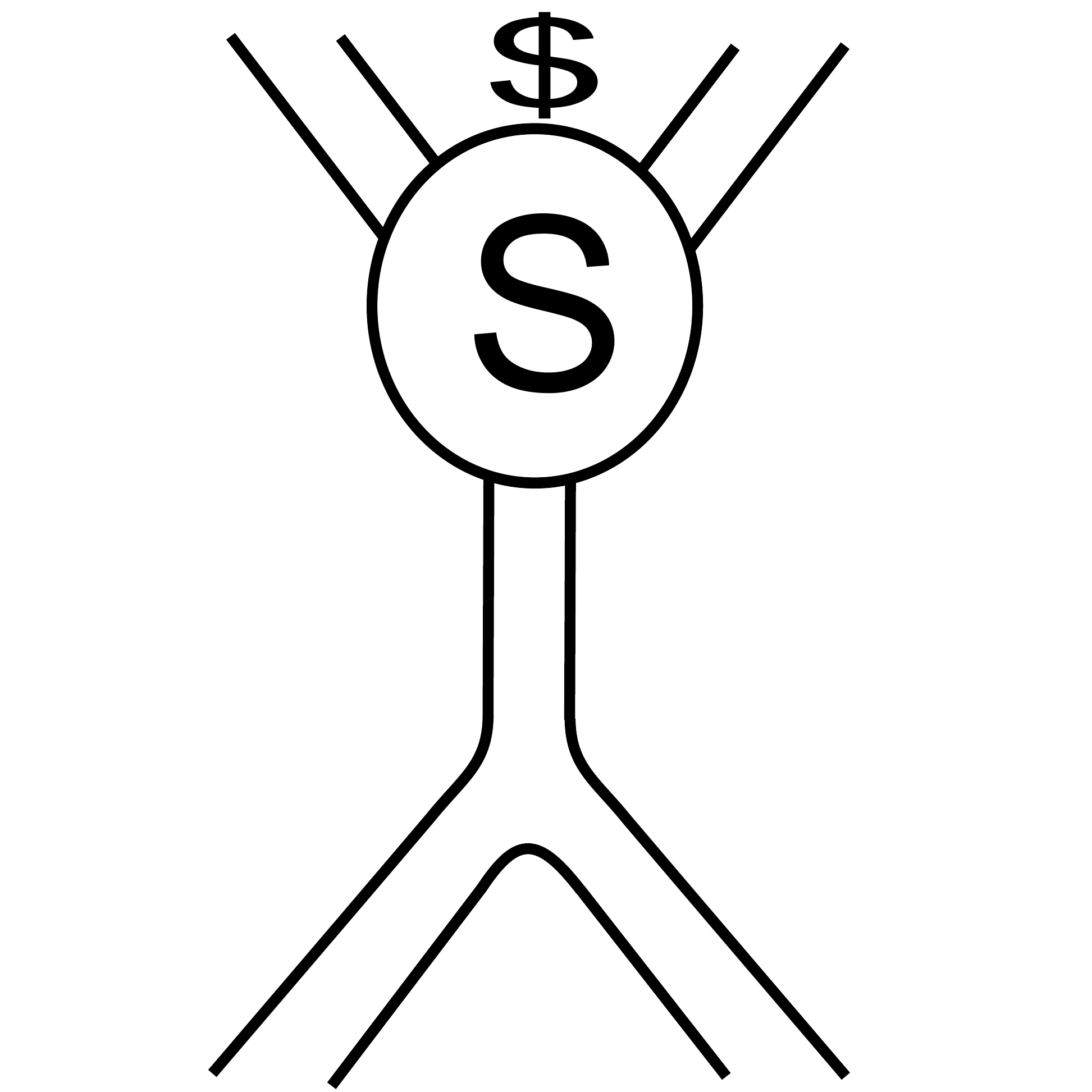}$ and $\gra{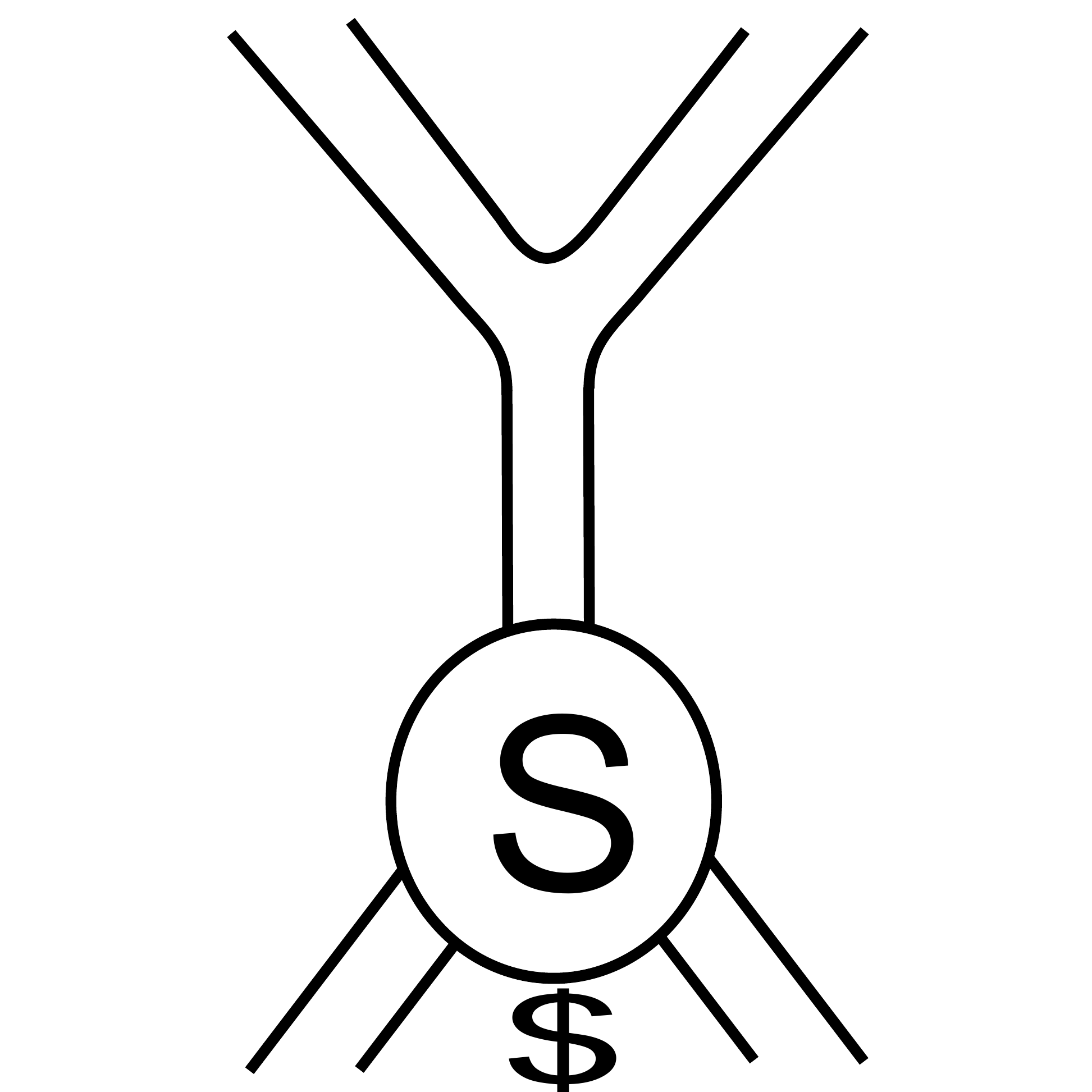}$ are $c$, $d$, $e$ and $f$ respectively, thus $c=d=e=f=0$. Finally, the coefficients of $\gra{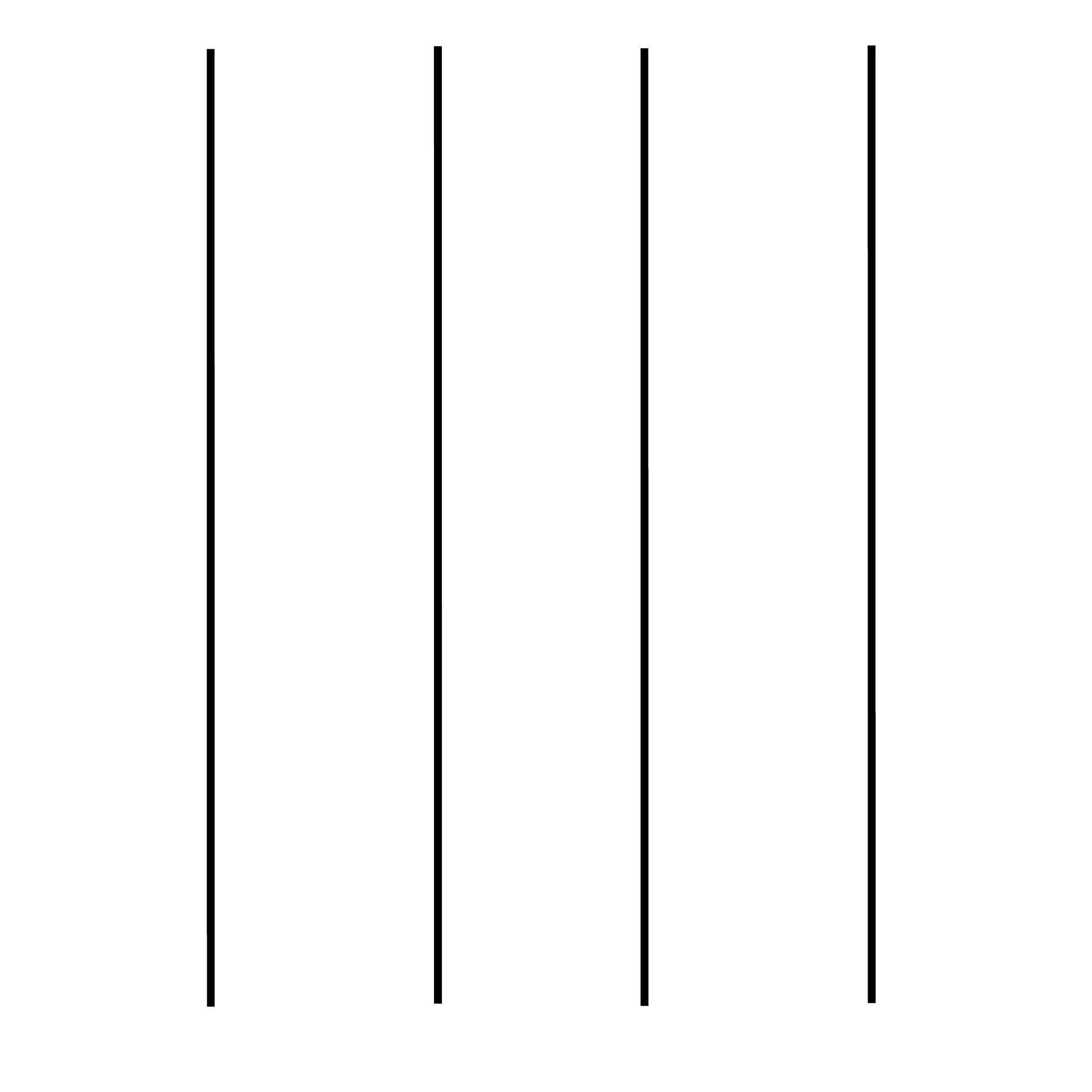}$, $\gra{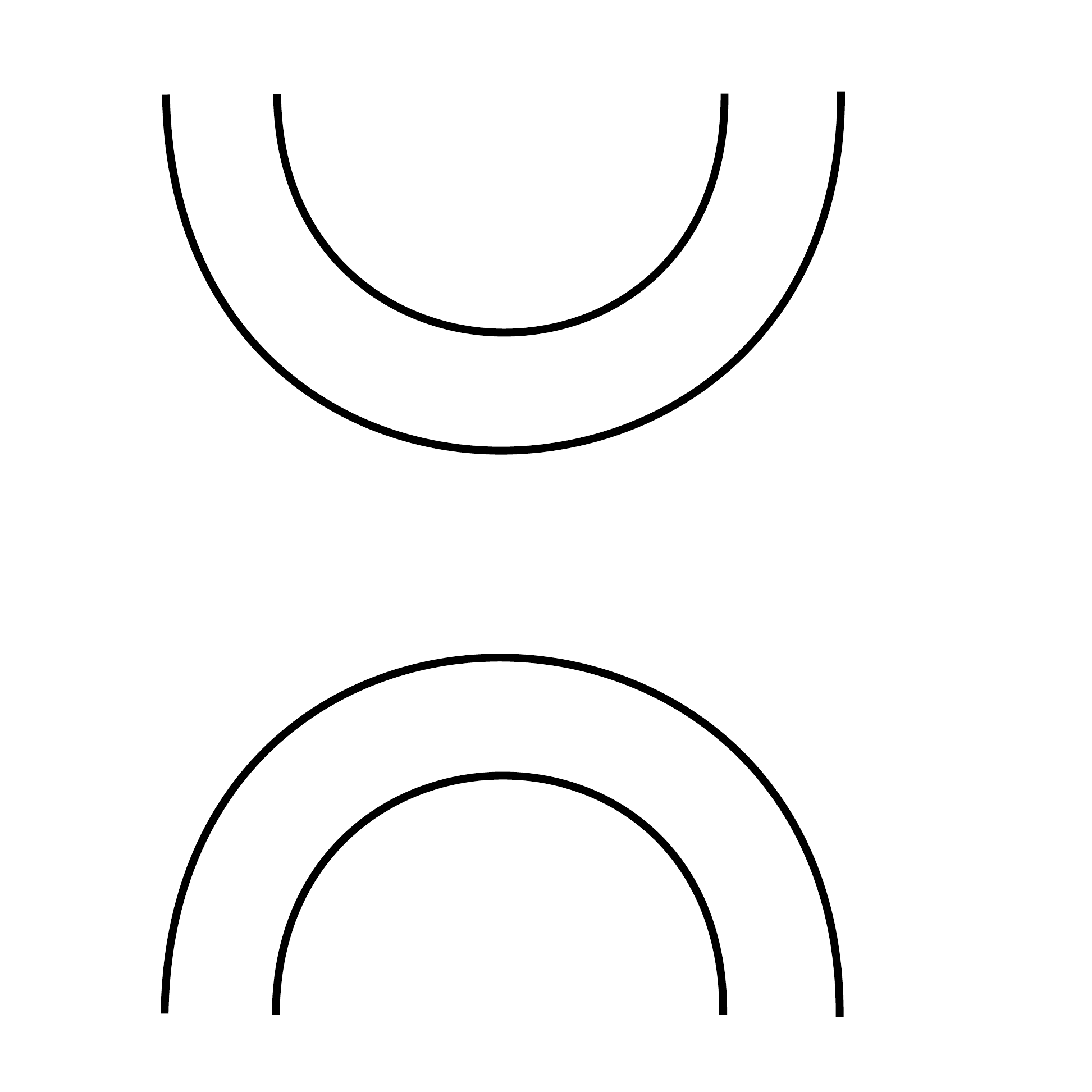}$ and $\gra{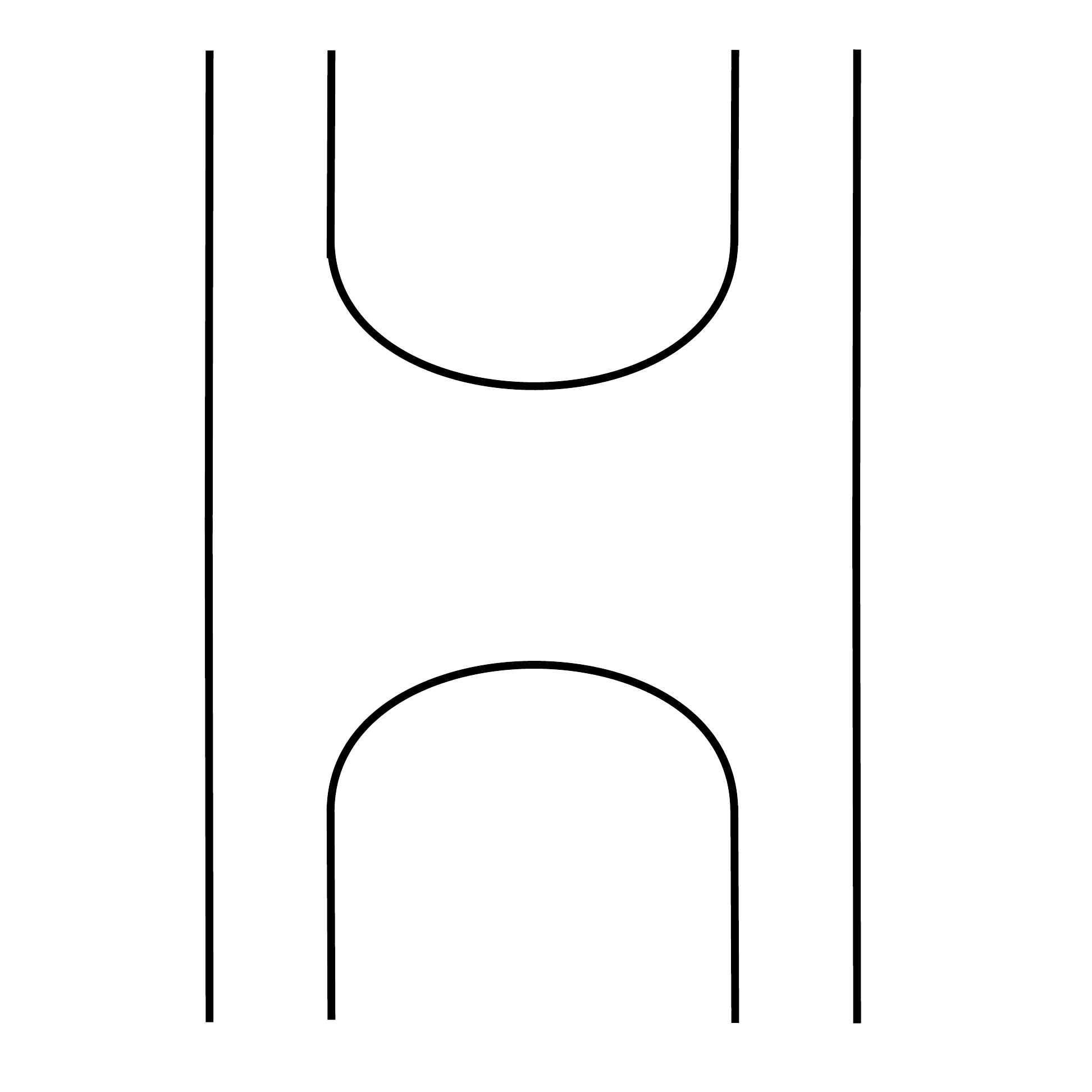}$ are $g$, $h$, and $i$ respectively, thus $g=h=i=0$.
Therefore $B$ is linearly independent, and $B$ is a basis.
\end{proof}

Since $B$ is a basis, we have the unshaded $2\leftrightarrow2$ move \eqref{Equ:H-I1} for $S$ in terms of formal variables:

\begin{align}
\graa{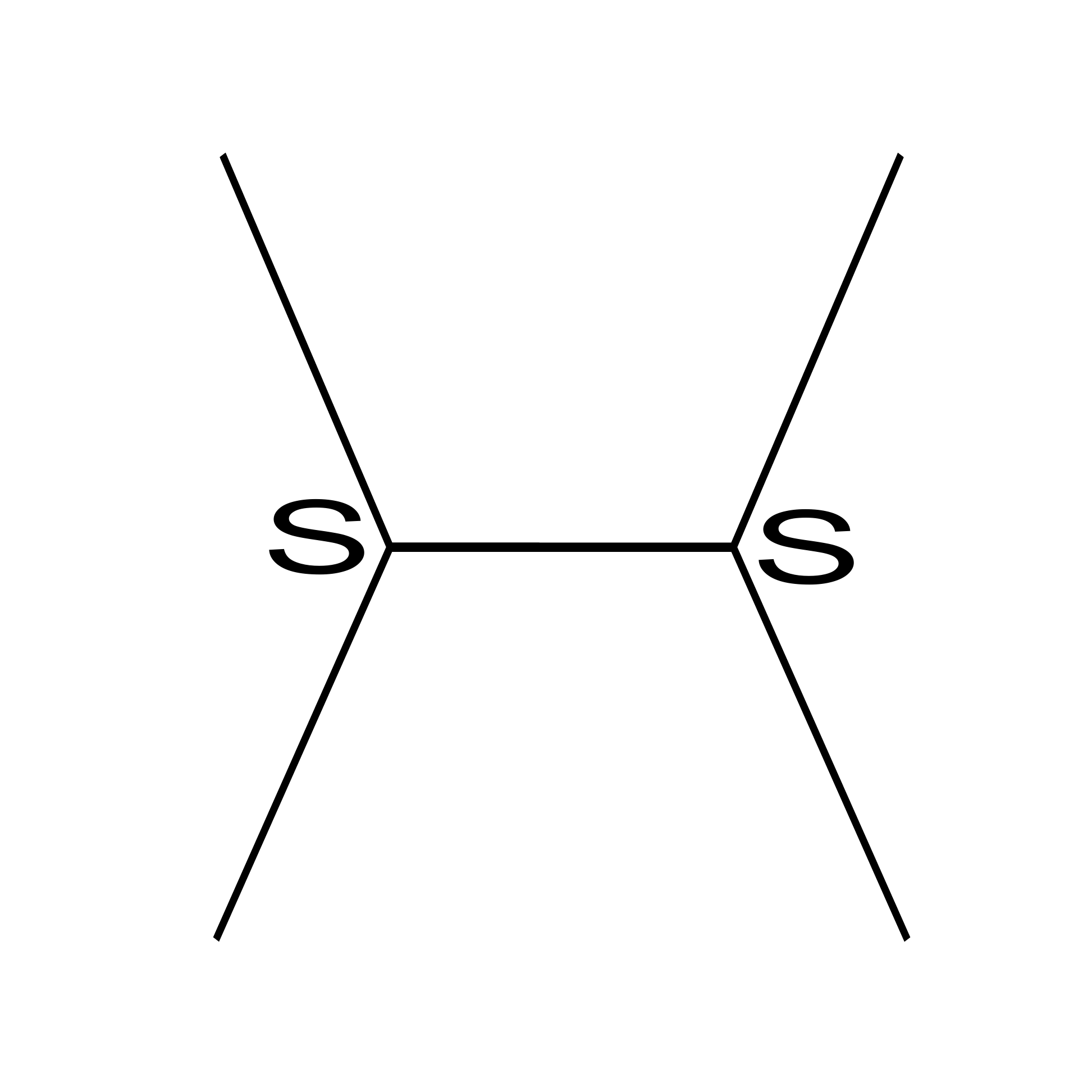}&=a\graa{ISS}+b_1\graa{HSTL}+b_2\graa{ISTL}+ \nonumber\\
&+b_3\graa{HTLS}+b_4\graa{ITLS}+c_1\graa{TLID}+c_2\graa{TLE}+d\graa{TLC} \;. \label{Equ:relation}
\end{align}

Using the rotational symmetry, we can immediately simplify the formal variables.

\begin{lemma}\label{Lem:a}
Either $a=1$, and
\begin{align}
\graa{HSS}&=\graa{ISS}+b\left(\graa{HSTL}-\graa{ITLS}+\graa{HTLS}-\graa{ISTL}\right)+ \nonumber \\
	&+c\left(\graa{TLID}-\graa{TLE}\right) \;. \label{Equ:a=1}
\end{align}
Or $a=-1$, and
\begin{align}
\graa{HSS}&=-\graa{ISS}+b\left(\graa{HSTL}+\graa{ITLS}+\graa{HTLS}+\graa{ISTL}\right)+ \nonumber\\
	&+c\left(\graa{TLID}+\graa{TLE}\right)+d\graa{TLC}. \label{Equ:a=-1}
\end{align}
\end{lemma}

\begin{proof}
  If $a=-1$, then we apply rotations by $\frac{\pi}{2}$ to Equation \eqref{Equ:relation} and obtain
\begin{align}
\graa{ISS}&=-\graa{HSS}+b_4\graa{HSTL}+b_1\graa{ISTL}+ \nonumber\\
&+b_2\graa{HTLS}+b_3\graa{ITLS}+c_2\graa{TLID}+c_1\graa{TLE}+d\graa{TLC} \;. \label{Equ:relation1}
\end{align}
Subtracting Equations \eqref{Equ:relation} from \eqref{Equ:relation1} yields
\begin{align}
0&=(b_4-b_1)\graa{HSTL}+(b_1-b_2)\graa{ISTL}+ \nonumber\\
&+(b_2-b_3)\graa{HTLS}+(b_3-b_4)\graa{ITLS}+(c_1-c_2)(\graa{TLID}+\graa{TLE}) \;. \label{Equ:relation2}
\end{align}
Thus $b_1=b_2=b_3=b_4$ and $c_1=c_2$ and we obtain Equation \eqref{Equ:a=-1}.

If $a\neq-1$, then we apply rotation by $\frac{k\pi}{2}, k=0,1,2,3,$ to Equation \eqref{Equ:relation} and take the alternating sum of the four resultant equations.
We have
\begin{align}
2(1+a)\graa{HSS}&=2(1+a)\graa{ISS}+\nonumber\\
&+(b_1-b_2+b_3-b_4)\left(\graa{HSTL}-\graa{ITLS}+\graa{HTLS}-\graa{ISTL}\right)+ \nonumber\\
	&+2(c_1-c_2)\left(\graa{TLID}-\graa{TLE}\right).
\end{align}
Taking the quotient by $2(1+a)$, we obtain Equation \eqref{Equ:a=1}.
\end{proof}

Furthermore, we can determine the coefficients $b,c,d$ in Lemma \ref{Lem:a} in terms of $\delta,\gamma,\omega$.

\begin{lemma}\label{Thurston}
If $a=1$, then
\begin{align}\label{Solution:a=1}
\left\{
\begin{aligned}
 b&=-(\gamma-1)\frac{\delta}{\delta^2-2+\omega+\omega^{-1}} \;,\\
 c&=-\gamma\frac{\delta}{\delta^2-1} \;.
\end{aligned}
\right.
\end{align}

If $a=-1$, then
\begin{align} \label{Solution:a=-1}
\left\{
\begin{aligned}
b&=\frac{(\gamma-1)\delta}{\delta^2-2-\omega-\omega^{-1}} \;,\\
c&=\gamma\delta(2\frac{\delta^2-2}{\delta^4-3\delta^2+1}-\frac{1}{\delta^2-1}) \;,\\
d&=-2\gamma\frac{\delta^2}{\delta^4-3\delta^2+1} \;.
\end{aligned}
\right.
\end{align}
\end{lemma}
\begin{proof}
When $a=1$, applying \gra{TL} to the bottom on both sides of Equation \eqref{Equ:a=1}, we obtain $$(\gamma-1)\gra{S}-\gamma\frac{\delta}{\delta^2-1}\gra{TL}=b\left((\omega+\omega^{-1})(-\frac{1}{\delta})\gra{S}-\frac{\delta^2-2}{\delta}\gra{S}\right)+c\gra{TL}.$$
By comparing the coefficient in \gra{TL} and \gra{S}, we solve $b,c$ as in Eqution \eqref{Solution:a=1}.

When $a=-1$, we apply $\gra{TL}$ to the bottom on both sides of Equation \eqref{Equ:a=-1}.
By a similar computation, we solve $b,c,d$ as in Equation \eqref{Solution:a=-1}.
\end{proof}

We remark that the variable $\varepsilon$ in Lemma $\ref{skein}$ does not appear in the solutions \ref{Solution:a=1}, \ref{Solution:a=-1}. We can solve for $\varepsilon$ in terms of $\delta,\gamma,\omega$ by applying \gra{wS} to the bottom of both sides of Equations \eqref{Equ:a=1}, \eqref{Equ:a=-1}. We will determine $\varepsilon$ later.

\begin{theorem}\label{cor1}
	Suppose $\mathscr{P}_\bullet$ is a planar algebra generated by a non-trivial 3-box satisfying Thustron relations, then $\mathscr{P}_\bullet$ is determined by $(\delta,\gamma,\omega,a,a')$, where $\delta^2$ is the index, $\gamma$ the ratio of the trace of the two orthogonal minimal idempotents, $\omega$ is the rotation eigenvalue and $a,a'$ are the signs in the unshaded and shaded $(2\leftrightarrow2)$ moves.
\end{theorem}

\begin{proof}
  Suppose $\mathscr{P}_{\bullet}$ and $\mathscr{P}^{\prime}_{\bullet}$ are two planar algebras and there exists orthogonal minimal idempotents $P,Q\in \Pl_{\bullet}$ and $P',Q'\in \mathscr{P}^{\prime}_{\bullet}$. Then we can construct uncappable rotation eigenvectors $S$ and $S'$ as in section $3$ satisfying the same quadratic relation. Since the planar algebra is generated by a 3-box, we define a map $\phi:\mathscr{P}_3\rightarrow\mathscr{P}^{\prime}_3$ by sending $S$ to $S'$. Since all the coefficients in the skein theory are determined $(\delta,\gamma,\omega)$, $\phi$ extends to a planar algebra isomorphism.
\end{proof}

With the structure of the Thurston-relation coefficients determined, we are ready to turn to the classification of Thurston-relation planar algebras.

\subsection{Classification}\label{sec:classification}

In this section, we prove our classification result (Theorem \ref{Thm:main}) for the generic case.

In Theorem \ref{cor1}, we show that $\mathscr{P}_{\bullet}$ is determined by $\delta>2,\gamma,\omega,a$ and $a'$ for the unshaded and shaded $2\leftrightarrow 2$ moves.
First we prove that $\omega=1$ and $a,a'=1$ in the $2\leftrightarrow 2$ move.
Thus $\mathscr{P}_{\bullet}$ is determined by $\delta,\gamma$.
Then we identify  $\mathscr{P}_{\bullet}$ with $\mathscr{P}_\bullet^H(q,r)$.

\begin{lemma}\label{thm1}
If $\mathscr{P}_{\bullet}$ is a planar algebra generated by a non-trivial 3-box satisfying Thurston's relation with parameters $(\delta,\gamma,\omega,a,a')$, $\delta>2$, then the rotation eigenvalue $\omega=1$.
\end{lemma}

\begin{proof}

A direct computation using Lemma \ref{skein} shows that in general,

$$\graa{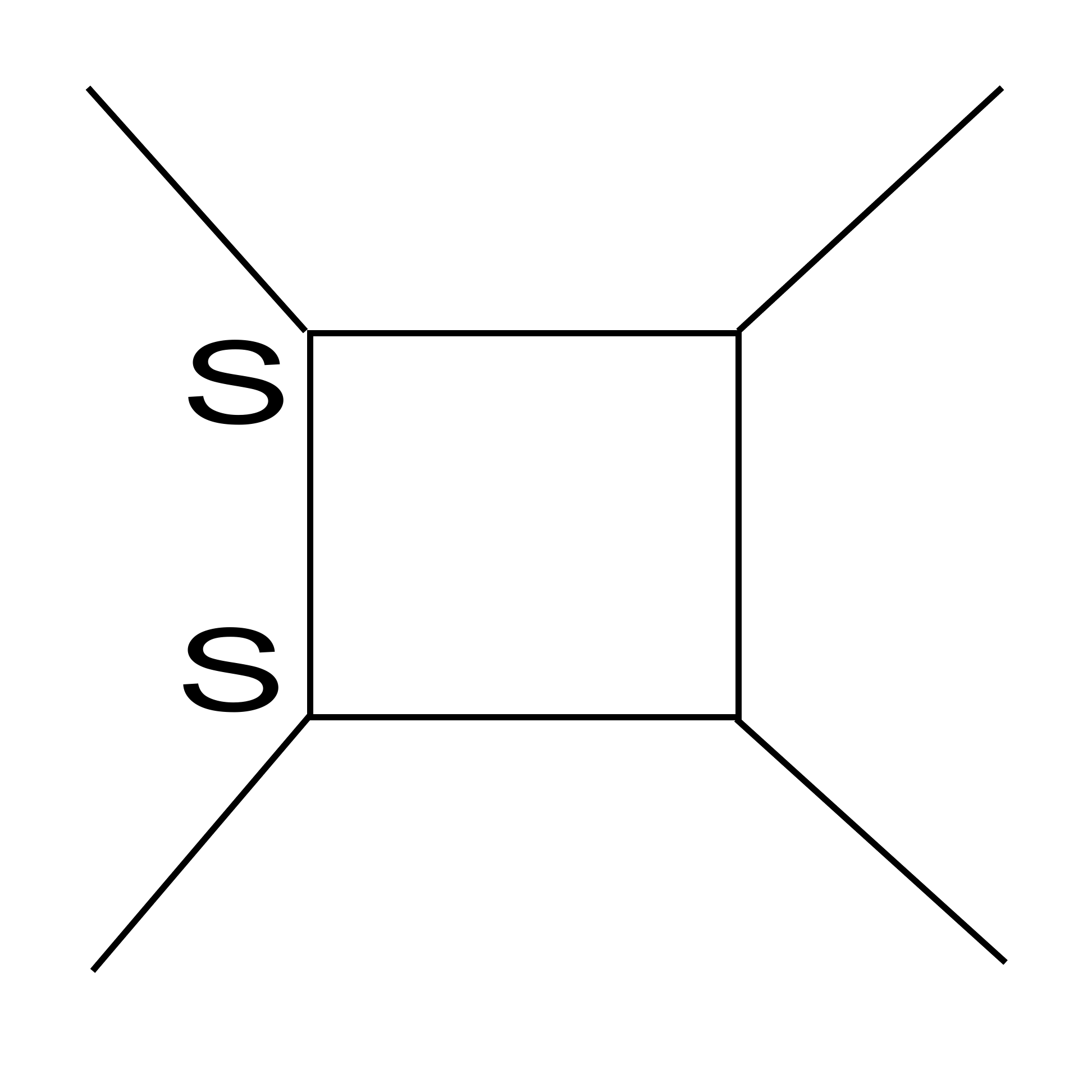}=(\gamma-1)\graa{HSTL}+\gamma\graa{TLID}-\frac{\gamma\delta}{\delta^2-1}\graa{TLC}-\frac{1}{\delta}\graa{ISS}$$

Suppose $a=1$,
	Notice that $\graa{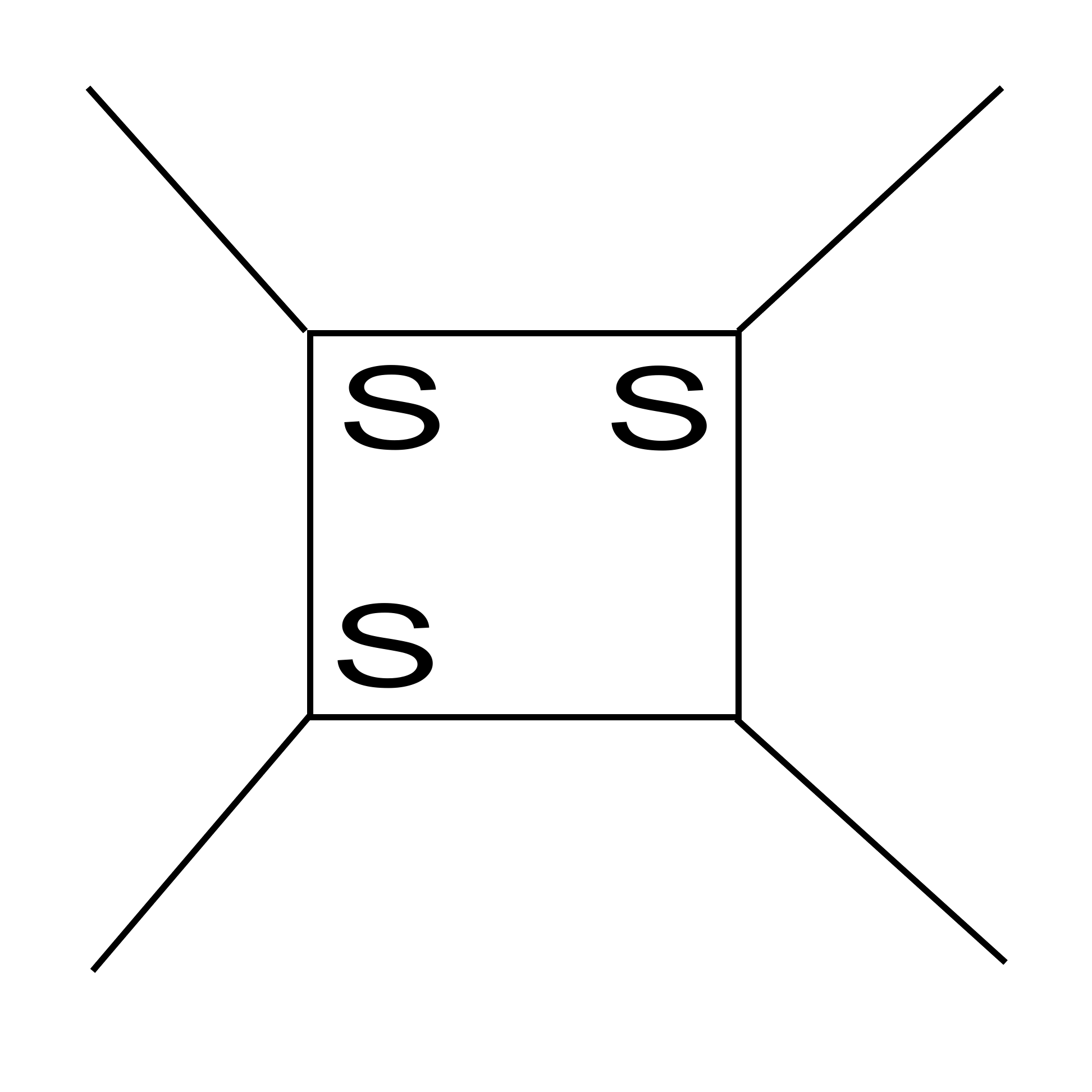}=\omega\omega^{-1}\graa{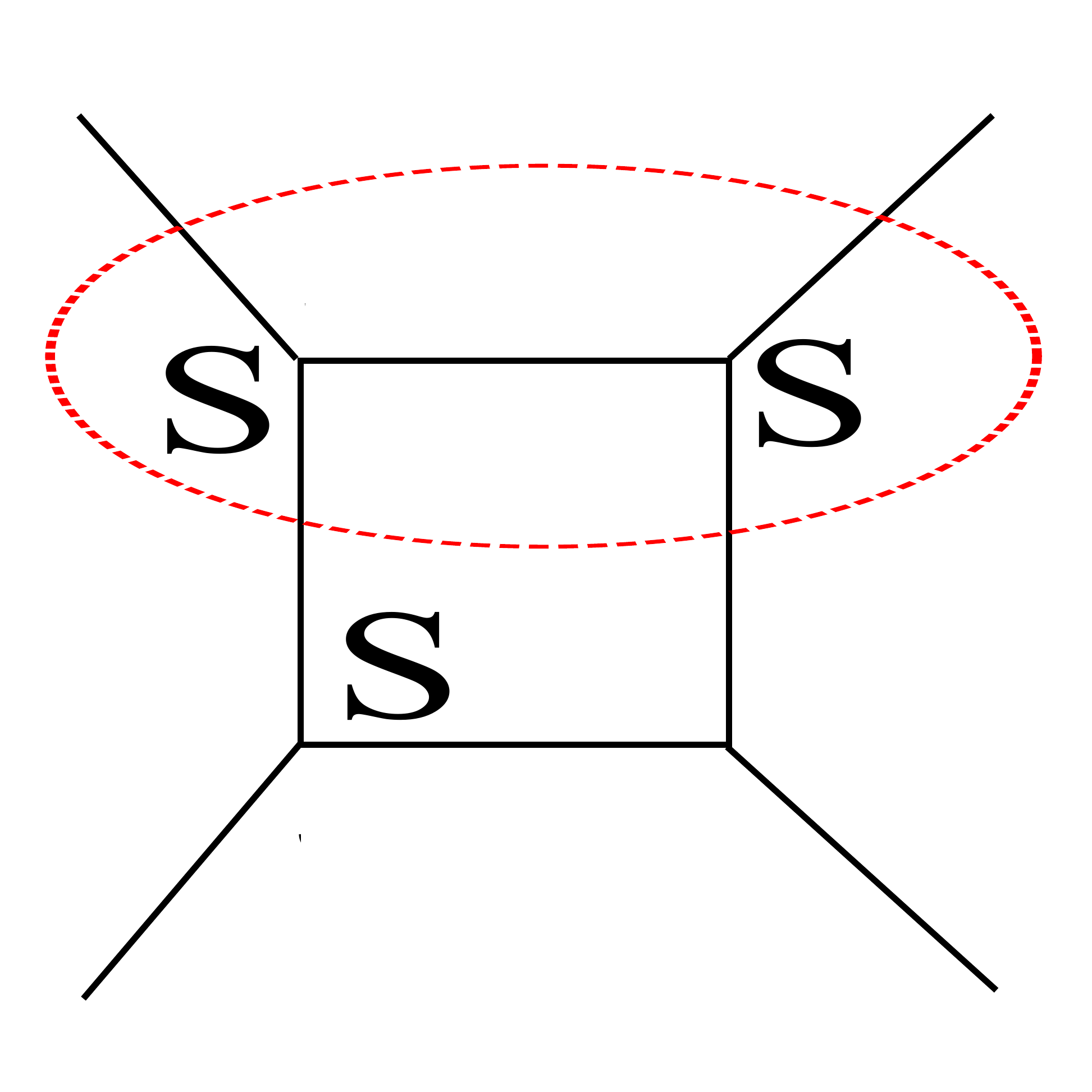}$,\\
	and apply the ($2\leftrightarrow2$) move in the dotted circle, we can rewrite with respect to the basis in Corollary \ref{Cor:24-basis}. The coefficient of \graa{BOX4TLSTLS} in the linear representation of $\graa{BOX4SSTLS}$ with respect to this basis is $b\omega^{-1}$.\\
	Similarly
	$\graa{BOX4SSTLS}=\omega\omega^{-1}\graa{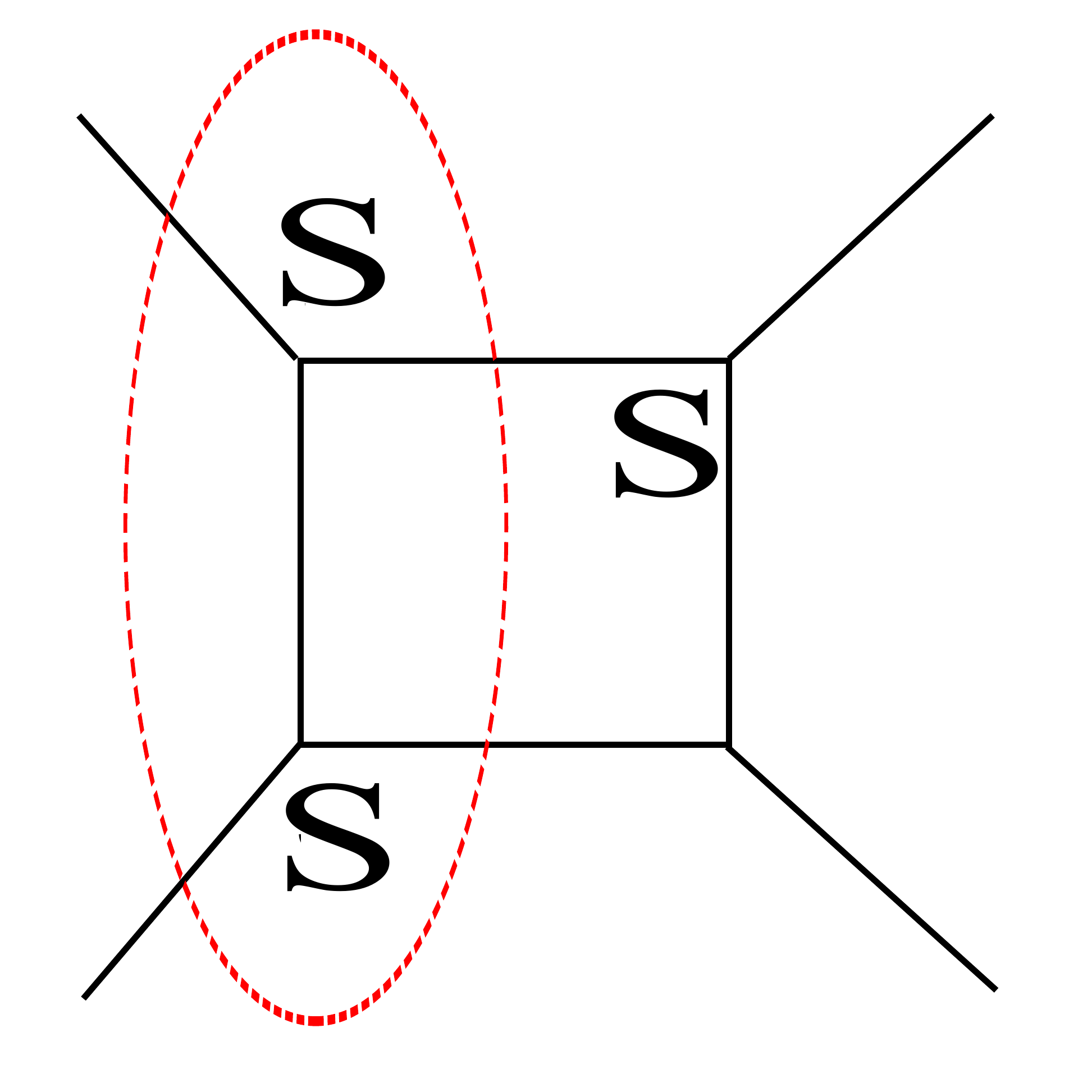}$,\\
	so applying the ($2\leftrightarrow2$) move in the dotted circle, we see that the coefficient of \graa{BOX4TLSTLS} is $b\omega$.

    Equating these coefficients from the two different computations yields
	
	$$b(\omega-\omega^{-1})=0.$$
	Since $\omega\neq 1$, $\omega-\omega^{-1}\neq0$, and so
    $$b=-\frac{(\gamma-1)\delta}{\delta^2-3}=0,$$
	hence $\gamma=1$.\\
	Note that from Thurston's relation, we obtain
	$$\graa{BOX3SSS}=(\frac{\gamma\delta}{\delta^2-1}(\delta^2-3)-2(\gamma-1)^2\frac{\delta}{\delta^2-3})\graa{SV}+\gamma(\gamma-1)\frac{\delta}{\delta^2-1}\graa{TLTRIV}$$
	Since the coefficient before \gra{SV} is 0 with $\omega\neq1$, we obtain $\frac{\delta(\delta^2-3)}{\delta^2-1}=0$, which contradicts the fact that $\delta>2$.\\

    Now suppose $a=-1$.  Again, we can evaluate \graa{BOX4SSTLS} as a sum of basis diagrams in two different ways as above, which will again give us that $\gamma=1$.\\
    Similarly by checking $\graa{BOX3SSS}$, we obtain
$$\frac{\delta^2-2}{\delta^2-1}=2\frac{\delta^2-2}{\delta^4-3\delta^2+1}-\frac{1}{\delta^2-1},$$
which yields the equation
$$1=2\frac{\delta^2-2}{\delta^4-3\delta^2+1}.$$
This has no real solutions with $\delta>2$.\\

This concludes the proof of the theorem.

\end{proof}

\begin{theorem}\label{thm2}
If $\mathscr{P}_{\bullet}$ is a planar algebra generated by a non-trivial 3-box satisfying Thurston's relation with parameters $(\delta,\gamma,1,a,a')$, $\delta>2$, then $a=a'=1$.
\end{theorem}
\begin{proof}
    Suppose $a=-1$. First note that by \ref{thm1}, we have $\omega=1$.

	Applying ($2\leftrightarrow2$) move in the dotted circle, we obtain the relation
	$$\graa{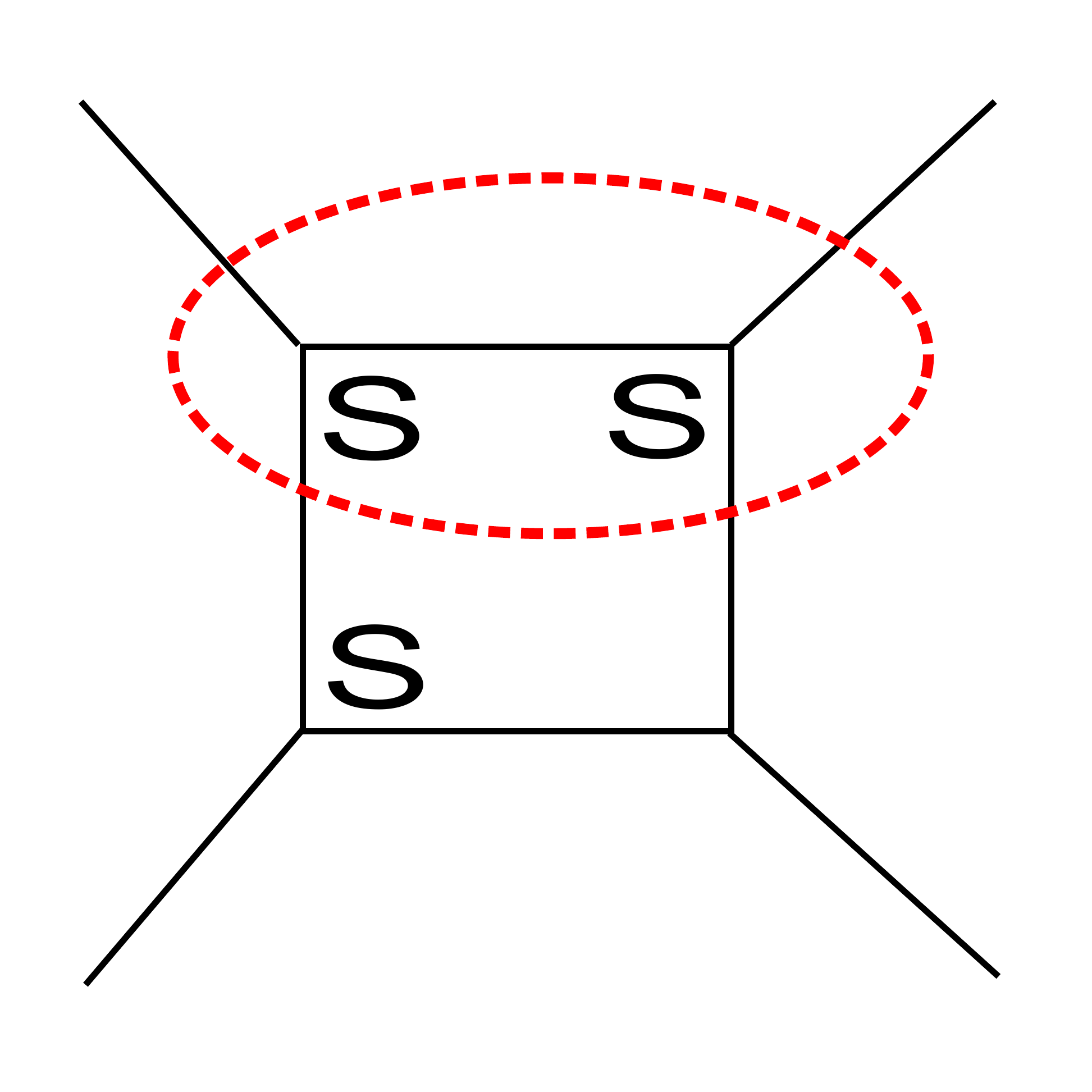}=b\graa{BOX4TLSTLS}+(-(\gamma-1)+b(-\frac{2}{\delta}))\graa{ISS}+\frac{\gamma\delta}{\delta^2-1}\graa{ISTL}$$
	$$~~~~~~~~~+(b(\gamma-1)+c)\graa{HSTL}+(b(\gamma-1)+d(-\frac{1}{\delta}))\graa{ITLS}$$
	$$~~+b\gamma\graa{TLID}+\frac{b\gamma}{\delta^2-1}\graa{TLE}+2b(-\gamma\frac{\delta}{\delta^2-1})\graa{TLC}$$

	Applying ($2\leftrightarrow2$) move in the dotted circle, we obtain
	$$\graa{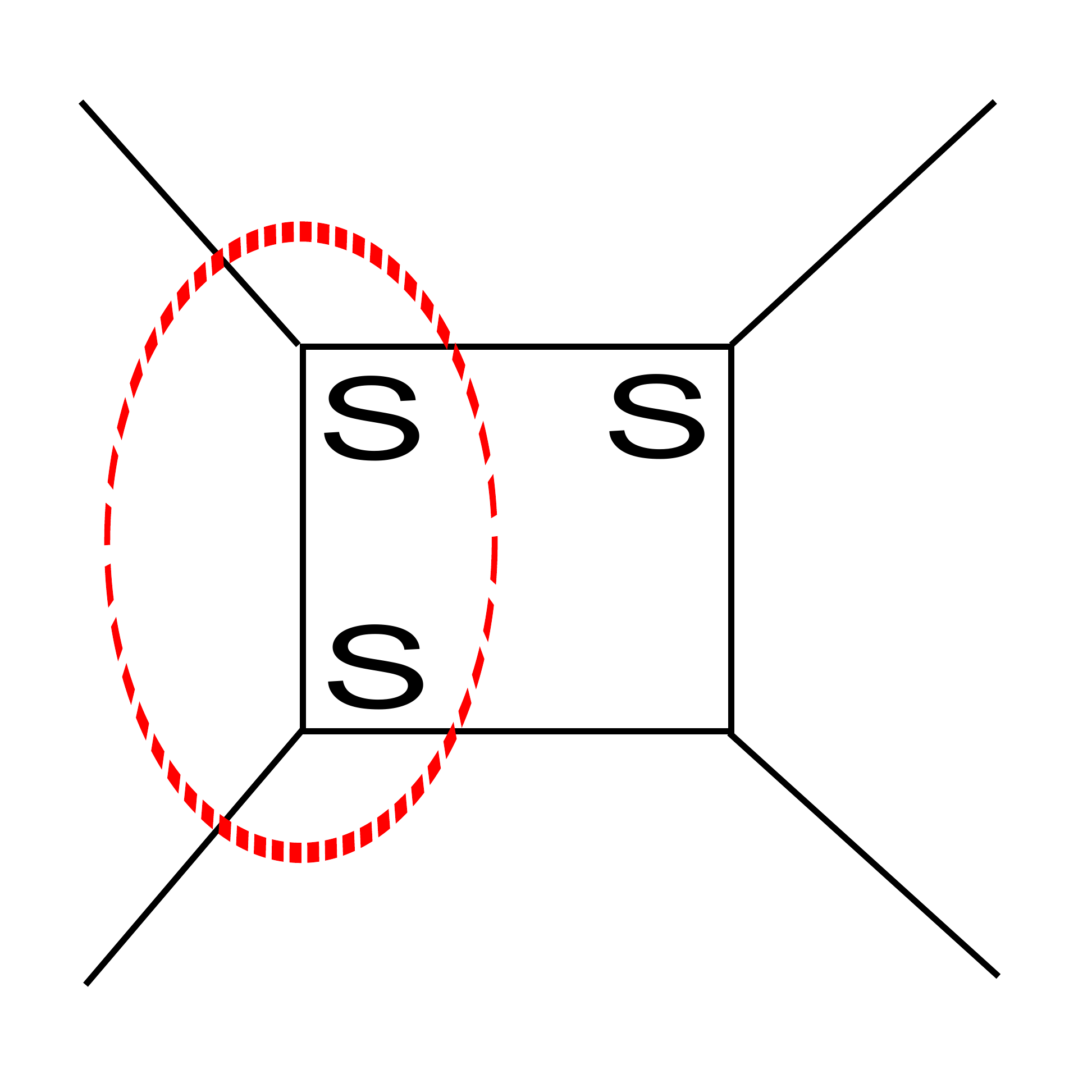}=b\graa{BOX4TLSTLS}+(-(\gamma-1)+b(-\frac{2}{\delta}))\graa{HSS}+\frac{\gamma\delta}{\delta^2-1}\graa{HSTL}$$
	$$~~~~~~~~~+(b(\gamma-1)+c)\graa{ISTL}+(b(\gamma-1)+d(-\frac{1}{\delta}))\graa{HTLS}$$
	$$~~+b\gamma\graa{TLE}+\frac{b\gamma}{\delta^2-1}\graa{TLID}+2b(-\gamma\frac{\delta}{\delta^2-1})\graa{TLC}.$$
	
    Let $\lambda$ be the coefficient of \graa{ISS} in the linear expansion of \graa{BOX4TLSTLS} with respect to basis.  (Note that if $\dim{\Pl_{4,+}=24}$, then $\lambda=0$. We do not use this assumption, so that the proof also work for the case $\dim{\Pl_{4,+}=23}$.)
    Equating the coefficients of $\graa{ISS}$ in the two expressions yields
	$$b\lambda-(\gamma-1)+b(-\frac{2}{\delta})=b\lambda +(\gamma-1)+b(\frac{2}{\delta}),$$

 thus
 $$-(\gamma-1)+b(-\frac{2}{\delta})=0.$$

 Note that $b=\frac{(\gamma-1)\delta}{\delta^2-4}$, when $\omega=1$, so
 $$(\gamma-1)(-1-\frac{\delta}{\delta^2-4}\frac{2}{\delta})=0$$

 hence

$$(\gamma-1)(\frac{\delta^2-2}{\delta^2-4})=0$$

$\delta>2$ implies that $\gamma=1$ and thus $b=0$. Therefore,
	$$\graa{BOX4SSTLS}=\frac{\delta}{\delta^2-1}\graa{ISTL}+c\graa{HSTL}+d(-\frac{1}{\delta})\graa{ITLS}$$ $$~~~~~~~~~~~~~=\frac{\delta}{\delta^2-1}\graa{HSTL}+c\graa{ISTL}+d(-\frac{1}{\delta})\graa{HTLS}$$
	Hence, $d=-2\gamma\frac{\delta^2}{\delta^4-3\delta^2+1}=0$. Note that this contradicts $\gamma=1$.
\end{proof}

These two theorems show that a singly genrated Thurston-relation planar algebra with parameters $(\delta,\gamma,\omega,a,a')$ satisfies $\omega=a=a'=1$. Thus the planar algebra is parameterized by $(\delta,\gamma)$. With the following lemma and theorem, we will identify any Thurston-relation planar algebra as $\mathscr{P}_\bullet^H(q,r)$ for some $(q,r)$.

\begin{lemma}\label{lemma:idem}
 	There exist two idempotents
 	$P,Q\in \mathscr{P}^H(q,r)_{3,+}$ with $P+Q=f_3$ and	
 	\begin{align}
 	tr(P)&=\frac{(r-r^{-1})(rq-r^{-1}q^{-1})(rq^{-2}-r^{-1}q^2)}{(q+q^{-1})(q-q^{-1})^3}\\
 	tr(Q)&=\frac{(r-r^{-1})(rq^2-r^{-1}q^{-2})(rq^{-1}-r^{-1}q)}{(q+q^{-1})(q-q^{-1})^3}
 	\end{align}
\end{lemma}

\begin{proof}

   Consider the basis $\left\{\gra{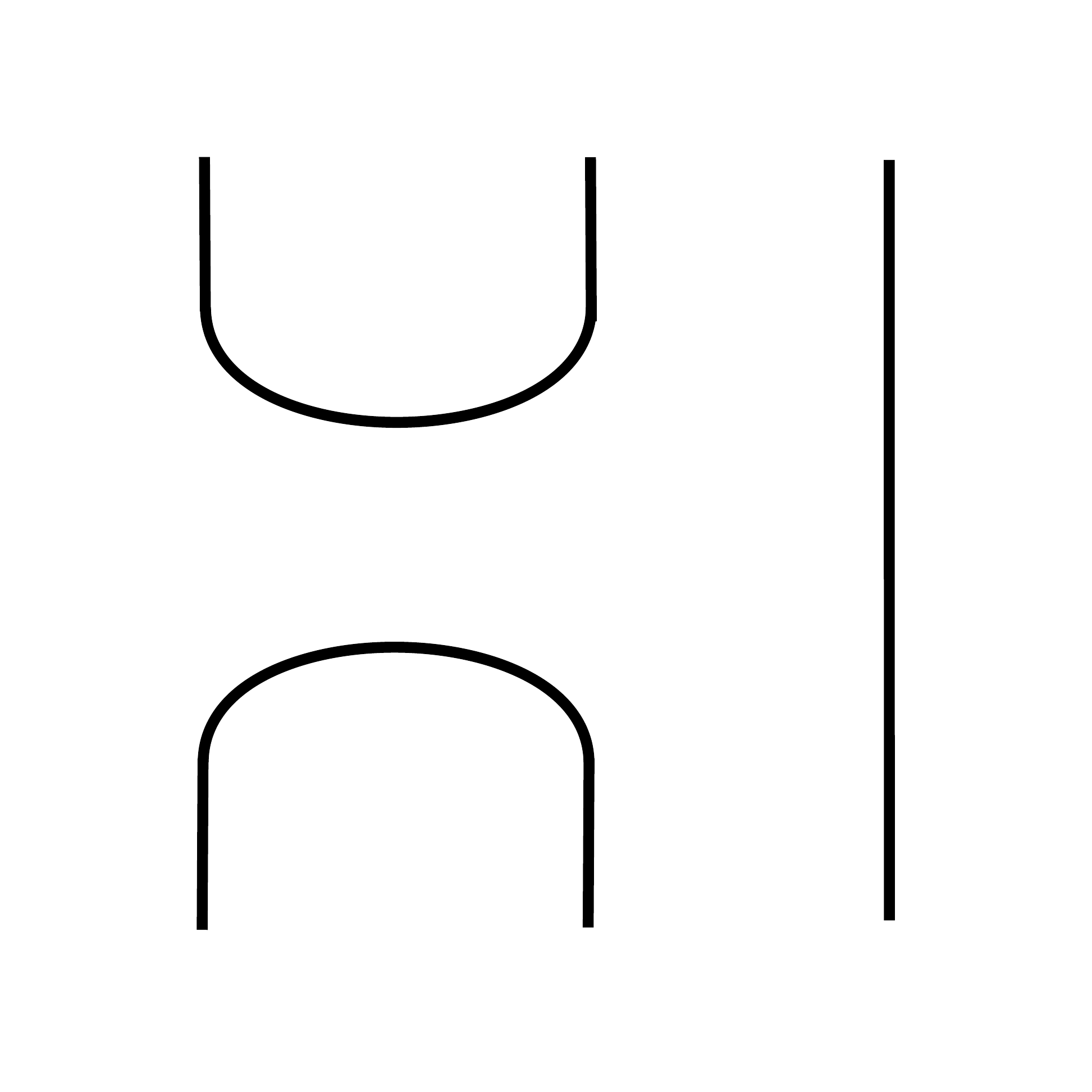},\gra{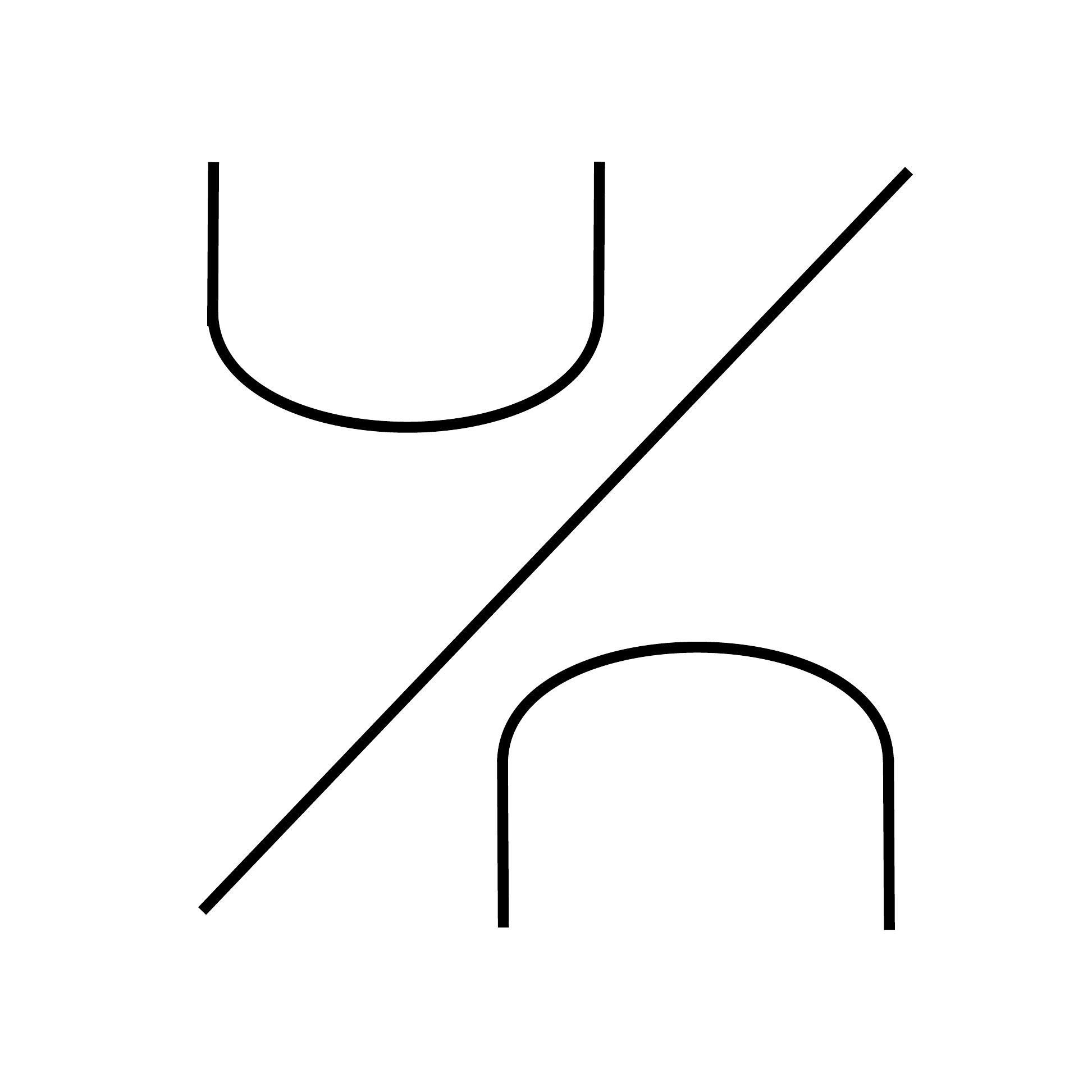},\gra{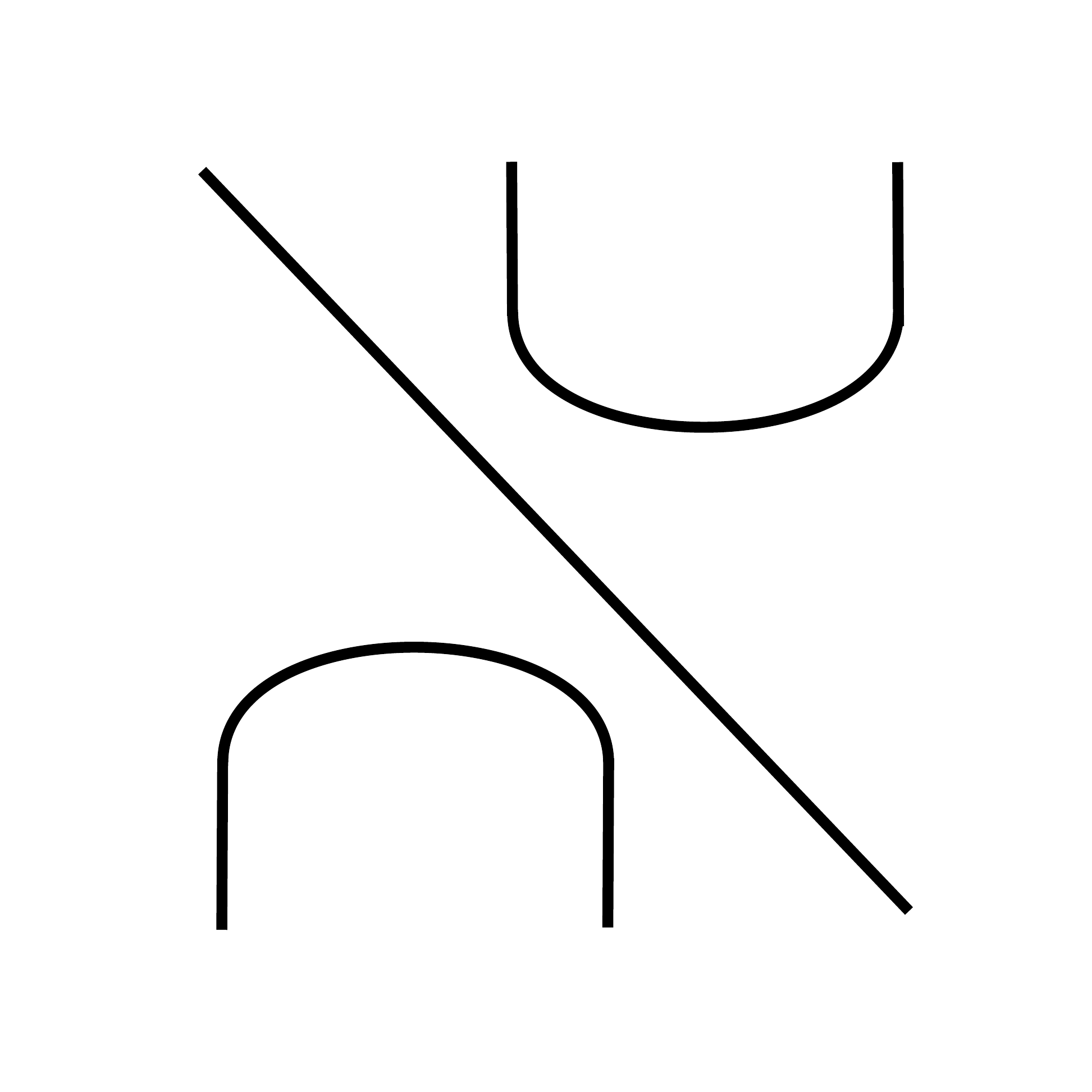}\right\}$ of $\mathscr{P}^H(q,r)_{3,+}$.\\
 	
    First set $Q=aR+b\gra{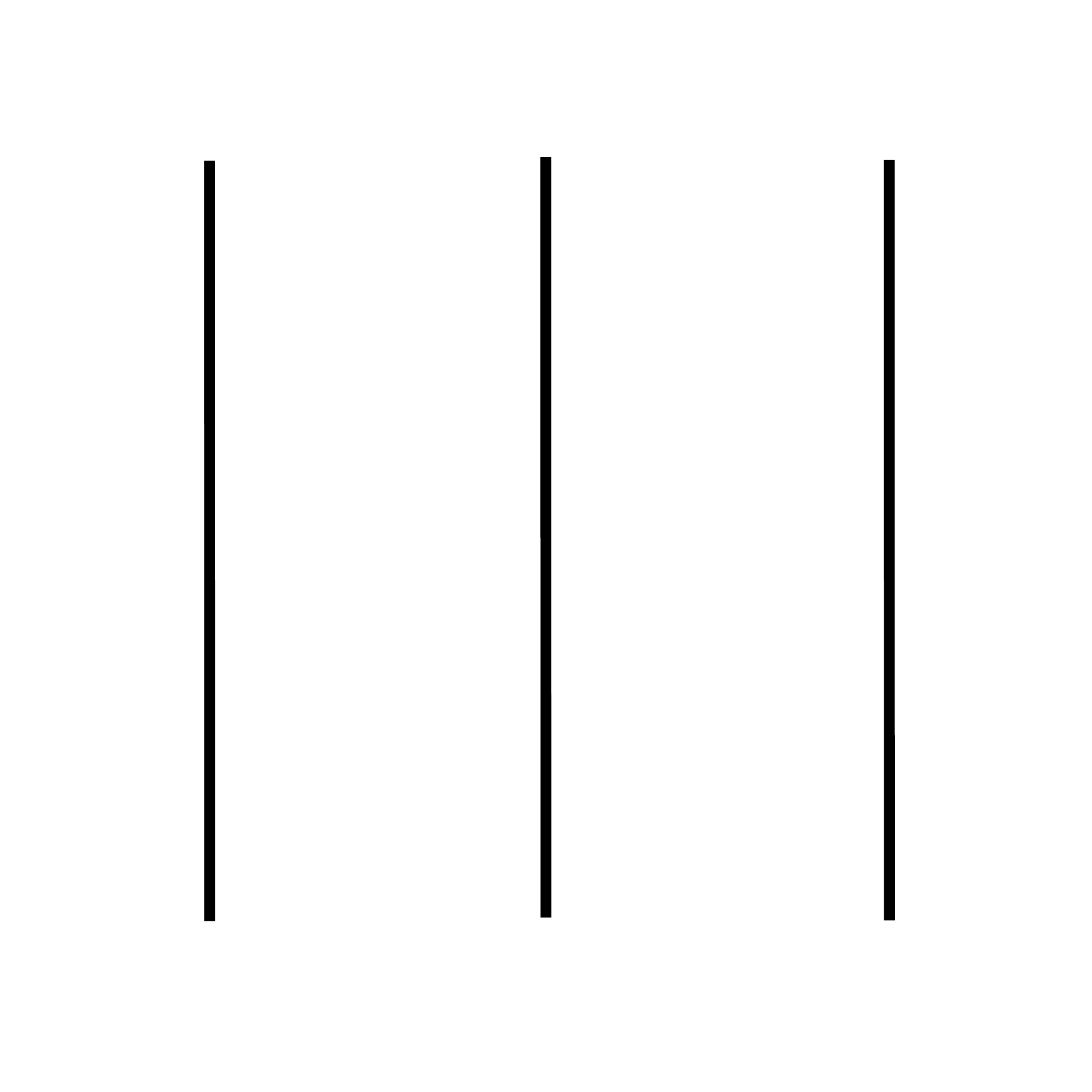}+c\left(\gra{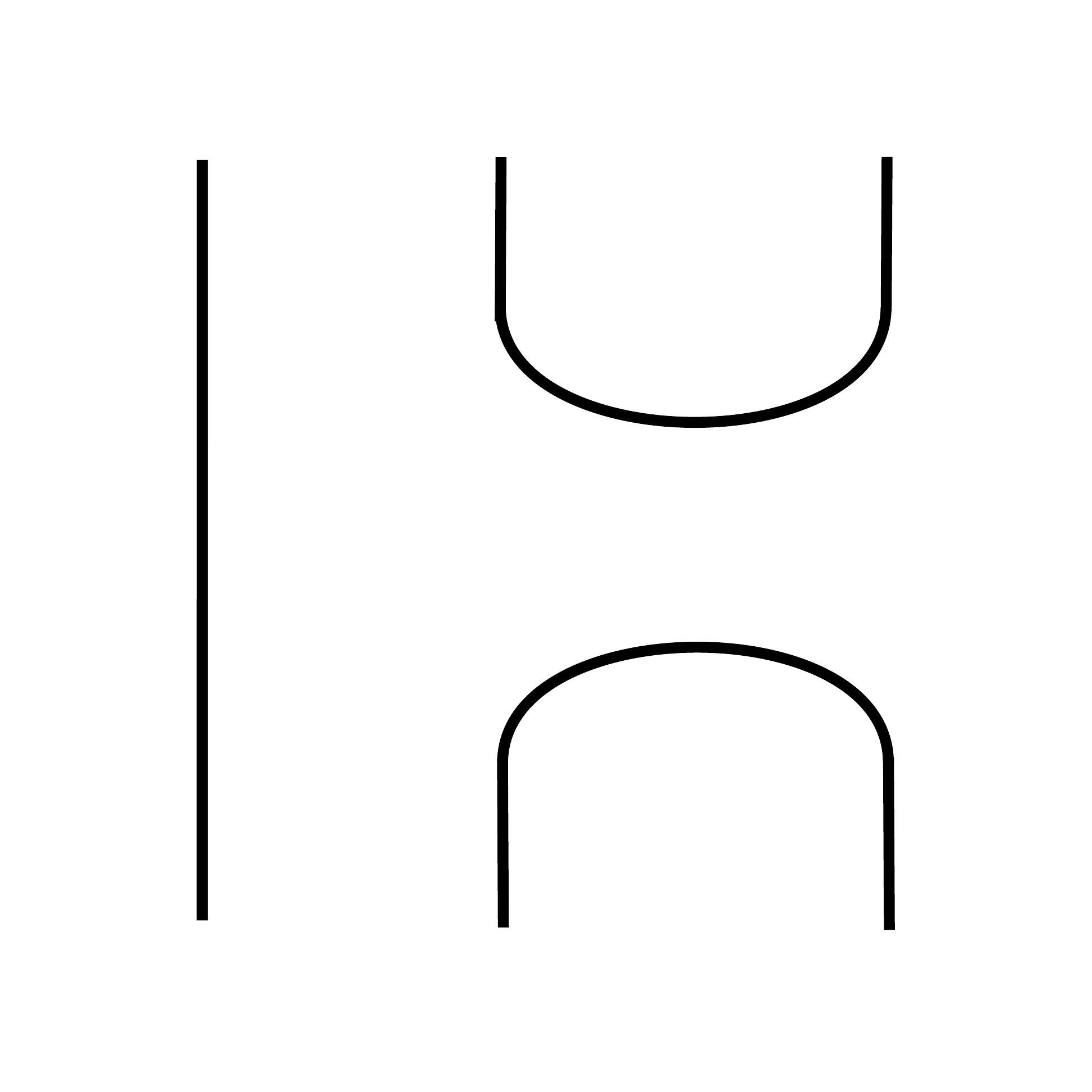}+\gra{TL3c2}\right)+e\gra{TL3e}+f\gra{TL3f}$.
 	Since we want $Q\in\mathcal{I}(\mathscr{P}^H(q,r)_{3,+})$, comparing the coefficients with respect to $\mathcal{B}$ of $e_k Q=Q e_k=0$ for $k=1,2$ we obtain the following equations:
 	\begin{equation}
 	\begin{cases}
 	ar+c+e\frac{r-r^{-1}}{q-q^{-1}}=0\\
 	b+c\frac{r-r^{-1}}{q-q^{-1}}+f=0\\
 	ar^{-1}+c+f\frac{r-r^{-1}}{q-q^{-1}}=0\\
 	a(q-q^{-1})+b+c\frac{r-r^{-1}}{q-q^{-1}}+e=0
 	\end{cases}	
 	\end{equation}
 	Which have solutions in terms of $a,b,r$ and $q$:
 	\begin{equation}
 	\begin{cases}
 	c=a\frac{r^{-1}(q-q^{-1})^2}{r^2+r^{-2}-q^2-q^{-2}}-b\frac{(r-r^{-1})(q-q^{-1})}{r^2+r^{-2}-q^2-q^{-2}}\\
 	e=-\frac{q-q^{-1}}{r-r^{-1}}(ar+ar^{-1}\frac{(q-q^{-1})^2}{r^2+r^{-2}-q^2-q^{-2}}-b\frac{(r-r^{-1})(q-q^{-1})}{r^2+r^{-2}-q^2-q^{-2}})\\
 	f=-\frac{q-q^{-1}}{r-r^{-1}}(ar^{-1}\frac{(r-r^{-1})^2}{r^2+r^{-2}-q^2-q^{-2}}+b\frac{(r-r^{-1})(q-q^{-1})}{r^2+r^{-2}-q^2-q^{-2}})
 	\end{cases}
 	\end{equation}
 	We want $Q$ to be an idempotent, so we set up equations by comparing the coefficients with respect $\mathscr{B}$ of $Q^2$ and $Q$ and we find two possible solutions:
 	\begin{equation}
 	\begin{cases}
 	a=\frac{1}{q+q^{-1}},~b=\frac{q^{-1}}{q+q^{-1}}~~~( ~1)\\
 	a=-\frac{1}{q+q^{-1}},~b=\frac{q}{q+q^{-1}}~~~(~ 2)
 	\end{cases}
 	\end{equation}
 	Taking (1), we have
 	\begin{equation}
 	\begin{split}
 	tr(Q)=&ar(\frac{r-r^{-1}}{q-q^{-1}})^2+b(\frac{r-r^{-1}}{q-q^{-1}})^3+c(2(\frac{r-r^{-1}}{q-q^{-1}})^2)+(e+f)\frac{r-r^{-1}}{q-q^{-1}}\\
 	=&\frac{r}{q+q^{-1}}(\frac{r-r^{-1}}{q-q^{-1}})^2+\frac{q^{-1}}{q+q^{-1}}(\frac{r-r^{-1}}{q-q^{-1}})^3\\
 	&+\frac{q-q^{-1}}{q+q^{-1}}\frac{r^{-1}q-r^{-1}q}{r^2+r^{-2}-q^2-q^{-2}}(2(\frac{r-r^{-1}}{q-q^{-1}})^2)\\
 	&+(-\frac{r+r^{-1}}{q+q^{-1}}-2\frac{q-q^{-1}}{q+q^{-1}}\frac{r^{-1}q-r^{-1}q}{r^2+r^{-2}-q^2-q^{-2}})\\
 	=&\frac{r}{q+q^{-1}}(\frac{r-r^{-1}}{q-q^{-1}})^2+\frac{q^{-1}}{q+q^{-1}}(\frac{r-r^{-1}}{q-q^{-1}})^3\\
 	&-\frac{r+r^{-1}}{q+q^{-1}}+2\frac{q-q^{-1}}{q+q^{-1}}\frac{r^{-1}q-r^{-1}q}{r^2+r^{-2}-q^2-q^{-2}}((\frac{r-r^{-1}}{q-q^{-1}})^2-1)\\
 	=&\frac{1}{(q+q^{-1})(q-q^{-1})^3}(r(r-r^{-1})^2(q-q^{-1})+q^{-1}(q-q^{-1})^3\\
 	&-(r+r^{-1})(q-q^{-1})^3+2(r^{-1}q-rq^{-1})(q-q^{-1})^2)\\
 	=&\frac{(r-r^{-1})(r^2q+r^{-2}q^{-1}-q^3-q^{-3})}{(q^2-q^{-2})(q-q^{-1})}\\
 	=&\frac{(r-r^{-1})(rq^2-r^{-1}q^{-2})(rq^{-1}-r^{-1}q)}{(q+q^{-1})(q-q^{-1})^3}
 	\end{split}
 	\end{equation}
 	Applying a similar computation, we see that $P=f_3-Q$ is an idempotent with
 	$$tr(P)=\frac{(r-r^{-1})(rq-r^{-1}q^{-1})(rq^{-2}-r^{-1}q^2)}{(q+q^{-1})(q-q^{-1})^3}$$
 	We see that taking the second case (2) simply switches the role of $P$ and $Q$.
 \end{proof}
 \begin{remark}
 	Another strategy to show the above trace formula is to use the trace formula \eqref{equ:traceformula} for $q$ a root of unity, and proving that the trace formula for idempotents are rational functions on $q,r$.
 \end{remark}

We are now ready for the following theorem:

\begin{theorem}\label{Thm:main1}
	Suppose $\mathscr{P}_\bullet$ is a singly generated Thurston-relation planar algebra and $\dim{\mathscr{P}_{4,\pm}}=24$. Then it is isomorphic to the semisimple quotient of $\mathscr{P}_\bullet^H(q,r)$ for some $(q,r)$.
\end{theorem}

\begin{proof}
	Suppose $\mathscr{P}_\bullet$ is a singly generated Thurston-relation planar algebra with parameter $(\delta>2,\gamma)$. We find $(q,r)$ such that $\mathscr{P}_\bullet^H(q,r)$ has the same parameter $(\delta,\gamma)$.
	
	Case 1: $\gamma=\displaystyle\frac{(\delta+2)(\delta-1)}{(\delta-2)(\delta+1)}$.
	
	In this case, we show that that $\mathscr{P}_\bullet^H(1,1)$ with a circle parameter $\delta$ is a solution. First note that this planar algebra has the desired $\delta$. Therefore, we only need to show that the ratio of two traces of the two minimal idempotents in $\mathscr{P}_3^H(1,1)$ equals to $\gamma$.
	By Lemma \ref{lemma:idem}, we know that $\mathscr{P}_3^H(1,1)$ with circle parameter $\delta$ has two minimal idempotents $P,Q$ with
	\begin{align}
			tr(P)&=\frac{\delta(\delta+2)(\delta-1)}{2}\\
			tr(Q)&=\frac{\delta(\delta-2)(\delta+1)}{2}
	\end{align}
	The ratio equals to $\displaystyle\frac{(\delta+2)(\delta-1)}{(\delta-2)(\delta+1)}$. Therefore, $\mathscr{P}_\bullet^H(1,1)$ with circle parameter $\delta$ gives a solution.
	
	Case 2: $\gamma\neq\displaystyle\frac{(\delta+2)(\delta-1)}{(\delta-2)(\delta+1)}$.
	Let $(q,r)$ be the solution of
		\begin{align} q^2+q^{-2}&=2(\delta^2-2)/(\delta^2-(\frac{\delta^2-2}{\delta}\frac{\gamma-1}{\gamma+1})^2)-2\label{Equ:qsolutions}\\
		r-r^{-1}=\delta(q-q^{-1})	\label{equ:circparameter}
		\end{align}
		(Note that the assumption  $\gamma\neq\displaystyle\frac{(\delta+2)(\delta-1)}{(\delta-2)(\delta+1)}$implies that $q\neq\pm1$.)
	From Equation \eqref{Equ:qsolutions}, we see that

	\begin{align}
	\delta^2(q^2+q^{-2}-2)+2&=(\frac{\delta^2-2}{\delta}\frac{\gamma-1}{\gamma+1}(q+q^{-1}))^2-2\\
	\delta^2(q-q^{-1})^2+4&=(\frac{\delta^2-2}{\delta}\frac{\gamma-1}{\gamma+1}(q+q^{-1}))^2\\
	(r-r^{-1})^2+4&=(\frac{\delta^2-2}{\delta}\frac{\gamma-1}{\gamma+1}(q+q^{-1}))^2 \label{equ:r}
	\end{align}
	From Equation \eqref{equ:r}, we obtain
		\begin{equation}\label{equ:r'}
		r+r^{-1}=\pm\frac{\delta^2-2}{\delta}\frac{\gamma-1}{\gamma+1}(q+q^{-1})
		\end{equation}
	Then we have the following equations:
	\begin{align}
		r^2-r^{-2}&=\pm(\delta^2-2)\frac{\gamma-1}{\gamma+1}(q^2-q^{-2})\label{equ:r2-}\\
		r^2+r^{-2}&=\delta^2(q^2+q^{-2}-2)+2\label{equ:r2+}
	\end{align}
	Recall that $q\neq\pm1$, we let $A,B$ to be defined by the formulas in Lemma \ref{lemma:idem}
	\begin{align}
	A&=\frac{(r-r^{-1})(rq-r^{-1}q^{-1})(rq^{-2}-r^{-1}q^2)}{(q+q^{-1})(q-q^{-1})^3}\\
	B&=\frac{(r-r^{-1})(rq^2-r^{-1}q^{-2})(rq^{-1}-r^{-1}q)}{(q+q^{-1})(q-q^{-1})^3}
	\end{align}
	Then we have
	\begin{align}
		A+B&=\frac{(r-r^{-1})((r^2+r^{-2})(q+q^{-1})+2(q^3+q^{-3}))}{(q+q^{-1})(q-q^{-1})^3}\label{equ:A+b}\\
		A-B&=\frac{(r-r^{-1})(r^2-r^{-2})(q-q^{-1})}{(q+q^{-1})(q-q^{-1})^3}\label{equ:A-b}
	\end{align}
	Combined with the Equation \eqref{equ:circparameter},\eqref{equ:r2-} and Equation \eqref{equ:r2+}, we see that either
	\begin{align}
		A&=(\delta^3-2\delta)\frac{1}{1+\gamma}  \\
		B&=(\delta^3-2\delta)\frac{\gamma}{1+\gamma}
	\end{align}
	or
	\begin{align}
		A&=(\delta^3-2\delta)\frac{\gamma}{1+\gamma}  \\
		B&=(\delta^3-2\delta)\frac{1}{1+\gamma}
	\end{align}
	In both cases, the singly generated Thurston-relation planar algebra with $(\delta,\gamma)$ and $\mathscr{P}_\bullet^H(q,r)$ have the same skein theory. Therefore, the planar algebras are isomorphic.

\end{proof}
\begin{remark}
	There are 8 solutions $(q,r)$ in the above theorem. The 8 corresponding planar algebras are isomorphic in Proposition \ref{lemsym}.
\end{remark}

\section{Reduced case}\label{Sec:reduced}
In this section, we will classify subfactor planar algebras $\mathscr{P}_{\bullet}$ generated by a 3-box with Thurston-relation for the reduced case, namely $\Pl_{4,\pm}\leq 23$.

\subsection{The case for at most 22 dimensional 4-box space}

\begin{proposition}\label{smallthm} If $\mathscr{P}_\bullet$ is $2$-supertransitive, then $\delta\leq 2$ if and only if $\dim{\mathscr{P}_{4,\pm}}\le 22$. In this case, $\mathscr{P}_{\bullet}$ is either the $E_{6}$ or $E^{(1)}_{6}$ subfactor planar algebra.
\end{proposition}

Thus the main Theorem \ref{Thm:main} holds for the reduced case $\dim{\mathscr{P}_{4,\pm}}\le 22$.

\begin{proof}
If $\delta\le 2$, then $\mathscr{P}_{\bullet}$ is either the $E_{6}$ or $E^{(1)}_{6}$ by the classification of subfactors up to index 4 \cite{Pop94}(Popa). Moreover, $\dim{\mathscr{P}_{4,\pm}}\le 22$.

If $\dim{\mathscr{P}_{4,\pm}}\le 22$, then $\delta\leq 2$ by Ocneanu's triple point obstruction \cite{Ocn88}(Ocn88).
\end{proof}

\subsection{The case for 23 dimensional 4-box space}
In this section, we classify subfactor planar algebras $\mathscr{P}_{\bullet}$ generated by a 3-box $S$ with Thurston-relation,
such that $\dim (\mathscr{P}_{4,\pm})=23$. In this case, we have $\delta>2$ and $\mathscr{P}_{3,+}=6$.
By the result of Jones \cite{JonAnn}, the 14 Temperley-Lieb diagrams and the 8 diagrams in the annular consequence are linearly independent.
Then one of the diagram with two generators $\gra{I}$, $\gra{H}$, $\gra{SI}$ and $\gra{SH}$ is linearly independent with these 22 diagrams. Otherwise $\mathscr{P}_{4,\pm}=22$.

Up to rotation and the duality of the shading, we can assume that $\gra{I}$ and the other 22 diagrams form a basis.
Similarly to Lemma \ref{Lem:Generic-cutdown-4-basis}, the basis for the $f_2$-cutdown of $\mathscr{P}_{4,+}$ is given by
$$B'=\left\{\graa{ISS},\graa{HTLS},\graa{HSTL},\graa{ISTL},\graa{ITLS},\graa{TLID},\graa{TLE},\graa{TLC}\right\}.$$

Note that all results in \S \ref{Sec:generic} work for the case $\mathscr{P}_{4,\pm}=23$, except Theorem \ref{thm1}. There we used the fact that $\graa{BOX4TLSTLS}$ is linearly independent with $B'$, which is no longer true for the reduced case.
Now we give a different proof of $\omega=1$ for the reduced case.
Consequently the main Theorem \ref{Thm:main} holds for the reduced case $\dim{\mathscr{P}_{4,\pm}}\le 23$.

\begin{theorem}\label{thm:omega23}
If $\mathscr{P}_\bullet$ is a singly genrated Thurston-relation subfactor planar algebra with parameters $(\delta,\gamma,\omega,a,a')$, then $\omega=1$.
\end{theorem}

\begin{proof}
One can deduce that the principal graph starts with the graph
$$\grb{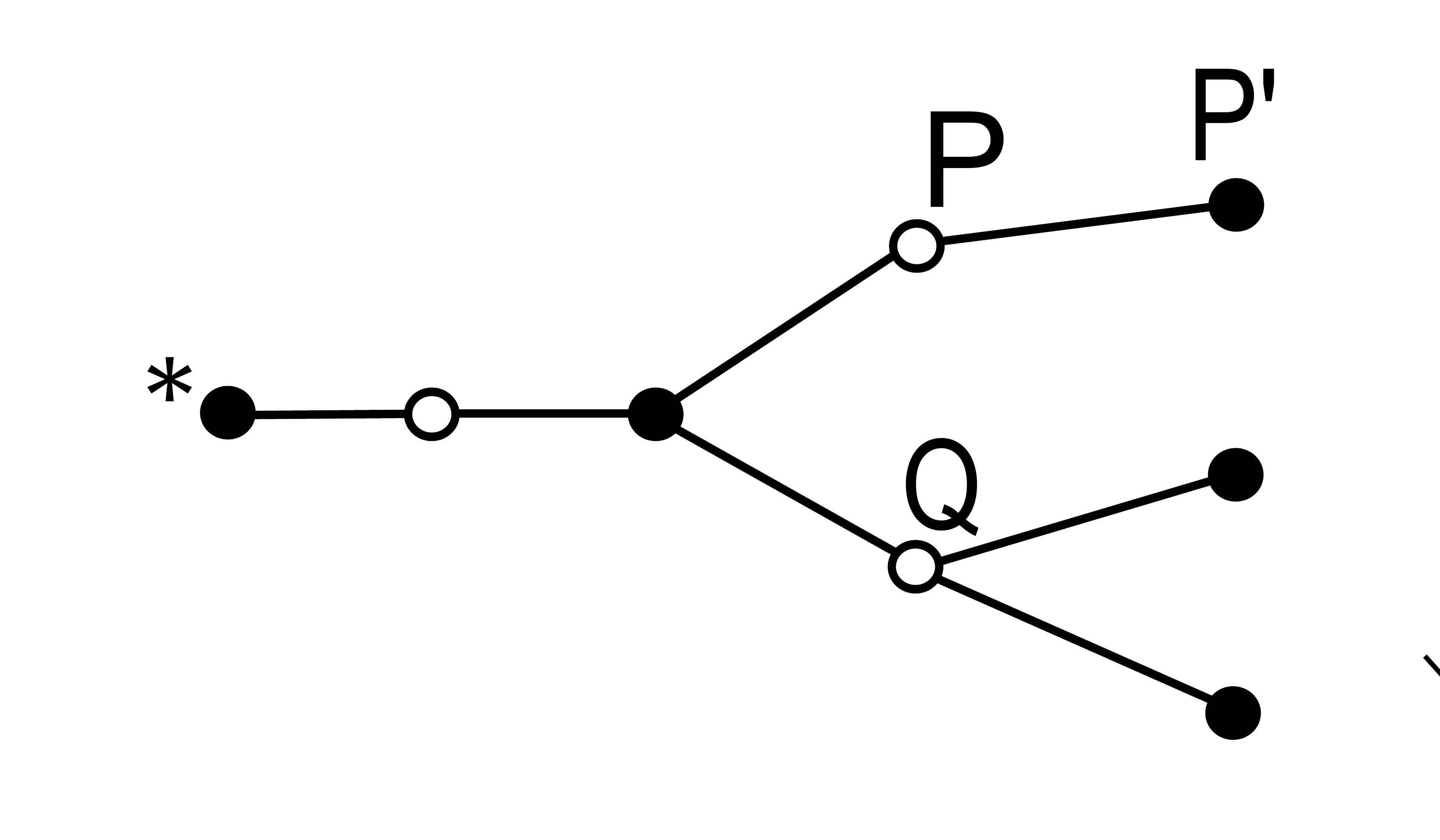}$$
where P and Q correspond to 2 minimal projections in $\mathscr{P}_{3,+}/\mathscr{I}_{3,+}$ described above.
	Now, we can apply Penney's triple point obstruction [Pen13]:
	\begin{equation}
	(\gamma-1)-\frac{\sigma+\sigma^{-1}}{[3]}=-(1+\gamma)\frac{[4]}{[3]}\left(Coef_{\in\cap_{n+1}\bar{P^{'}}}(S)\right)
	\end{equation}
	Where $\sigma^2=\omega$ and $[n]=\frac{q^{n}-q^{-n}}{q-q^{-1}}$ is the $n^{th}$ quantum number (where $q$ is defined so that $\delta=[2]=q+q^{-1}$ and $q>1$), $P^{\prime}$ is the projection labelled in our picture, and $\overline{P}^{\prime}$ is the dual projection.

Under the assumption $\omega=e^{i\frac{1}{3}\pi}$ or $e^{i\frac{2}{3}i\pi}$, $\sigma+\sigma^{-1}=\pm1$. This implies that
    \begin{equation}
	(\gamma-1)\pm\frac{1}{[3]}=-(1+\gamma)\frac{[4]}{[3]}\left(Coef_{\in\cap_{n+1}\bar{P^{'}}}(S)\right)
	\end{equation}
We now consider $2$ subcases:
	Suppose $P^{'}$ is self-dual, i.e $\gra{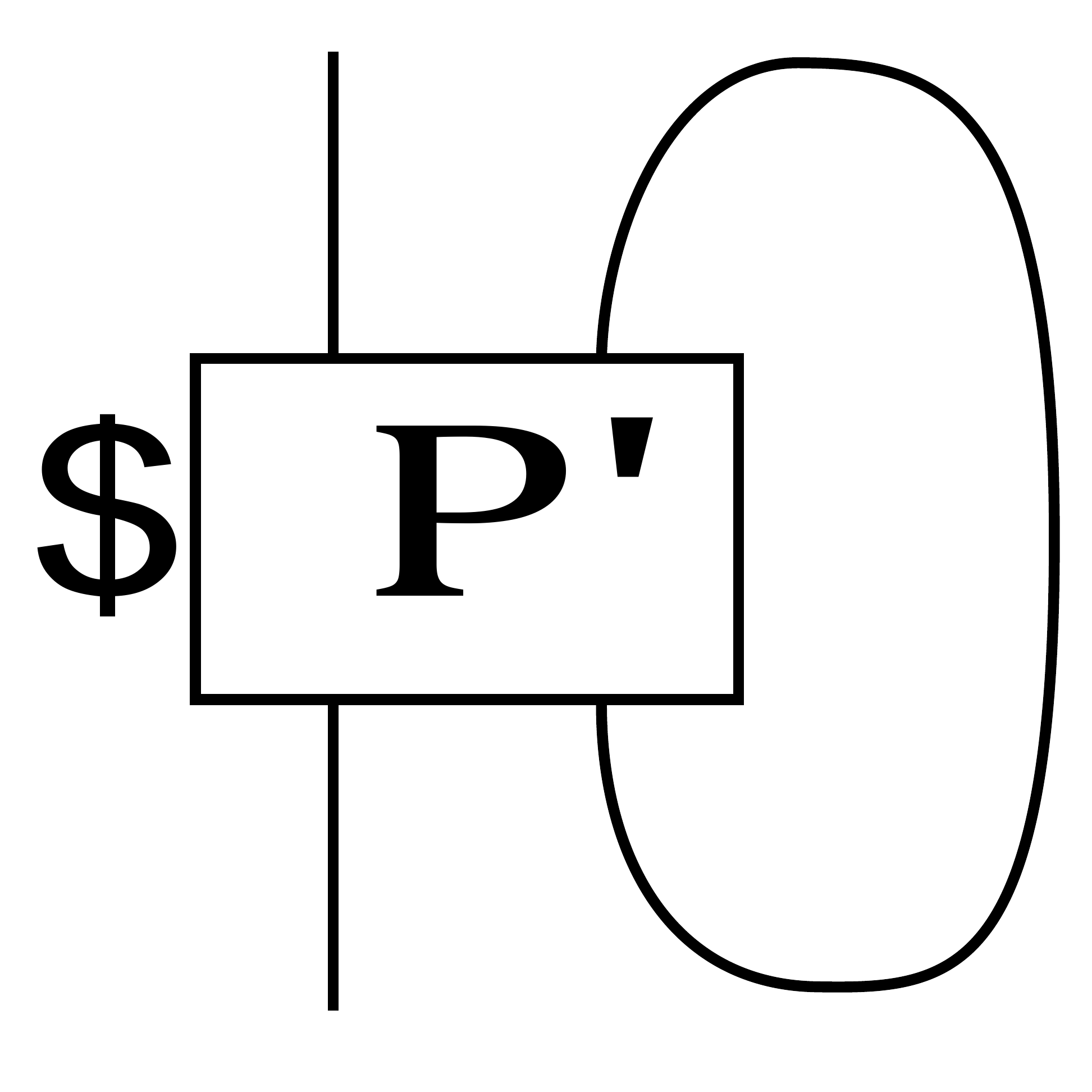}=P$. In this case, we have
	\begin{equation}
	Coef_{\in\cap_{n+1}\bar{P'}}(S)=\frac{tr(P')}{tr(P)}\frac{1}{\gamma+1}=\frac{[2]tr(P)-[3]}{tr(P)}\frac{1}{\gamma+1}=([2]-\frac{[3]}{[4]}(\gamma+1))\frac{1}{\gamma+1}
	\end{equation}

	Therefore,
	\begin{equation}
	\gamma-1\pm\frac{1}{[3]}=-\frac{[4][2]}{[3]}+1+\gamma,
	\end{equation}
	
which implies

    $$\delta^4-4\delta^2+1=0,$$

hence $\delta^2=2+\sqrt{3}<4$, contradicting our assumption that $\delta>2$.
	
	Now suppose that $P'$ is not self-dual, i.e $\gra{PCAP}=Q$. In this case, we have
	\begin{equation}
	Coef_{\in\cap_{n+1}\bar{P'}}(S)=\frac{tr(P')}{tr(Q)}\frac{1}{\gamma+1}=-\frac{[2]tr(P)-[3]}{tr(Q)}\frac{1}{\gamma+1}=-([2]\frac{1}{r}-\frac{[3]}{[4]}(\frac{1}{\gamma}+1))\frac{1}{\gamma+1}
	\end{equation}
	
	Therefore,
	$$\gamma-1\pm\frac{1}{[3]}=-\frac{[4]}{[3]}([2]\frac{1}{\gamma}-\frac{[3]}{[4]}(1+\frac{1}{\gamma})),$$
	
	hence  $[3]\gamma^2\pm\gamma-[5]=0$.
	
    Now suppose $a=1$, then by considering evaluating $\graa{BOX3SSS}$, we have the following equations,
	
	\begin{align}	
		\frac{\gamma\delta}{\delta^2-1}(\delta^2-3)&=2(\gamma-1)^2\frac{\delta}{\delta^2-3}\label{EE1}\;,\\
		[3]\gamma^2\pm\gamma-[5]&=0\;.	\label{EE2}
	\end{align}
	
	One can solove Equation \eqref{EE2} for $\gamma$ in terms of $\delta$ and plug it in back to \eqref{EE1}. We use Mathematica to solve numerical solutions for $\delta$. They are far below 2 hence there is no solution when $\delta>2$.

    If $a=-1$, we obtain two equations following simialr arguments,
	\begin{align}
	-\gamma\frac{\delta^3-2\delta}{\delta^2-1}+b(2(\gamma-1))+c&=0\;,\\
	[3]\gamma^2\pm\gamma-[5]&=0\;,
	\end{align}
	where $b,c$ are given in Equation \eqref{Solution:a=-1}.
	
	One can show there is no solution with $\delta>2$ in a similar way.
	
	 We conclude that $\omega=1$ for the dimension 23 case.
\end{proof}

\begin{theorem}\label{Thm:main2}
	Suppose $\mathscr{P}_\bullet$ is a singly generated Thurston-relation planar algebra and $\dim{\mathscr{P}_{4,\pm}}=23$, then it is isomorphic to the semisimple quotient of $\mathscr{P}_\bullet^H(q,r)$ for some $(q,r)$.
\end{theorem}
\begin{proof}
It follows from Theorem \ref{thm:omega23}, \ref{thm2} and the proof of Theorem \ref{Thm:main1}.
\end{proof}

\section{Positivity}\label{Section:Pos}

In this section, we will determine the positivity of $\mathscr{P}_\bullet^H(q,r)$.

Consider the map $\phi_n:H_n(q,r)\rightarrow\mathscr{P}_n^H(q,r)$, for $x\in H_n(q,r)$ defined $\phi_n(x)$ as follows:
$$\phi_n\left(\grb{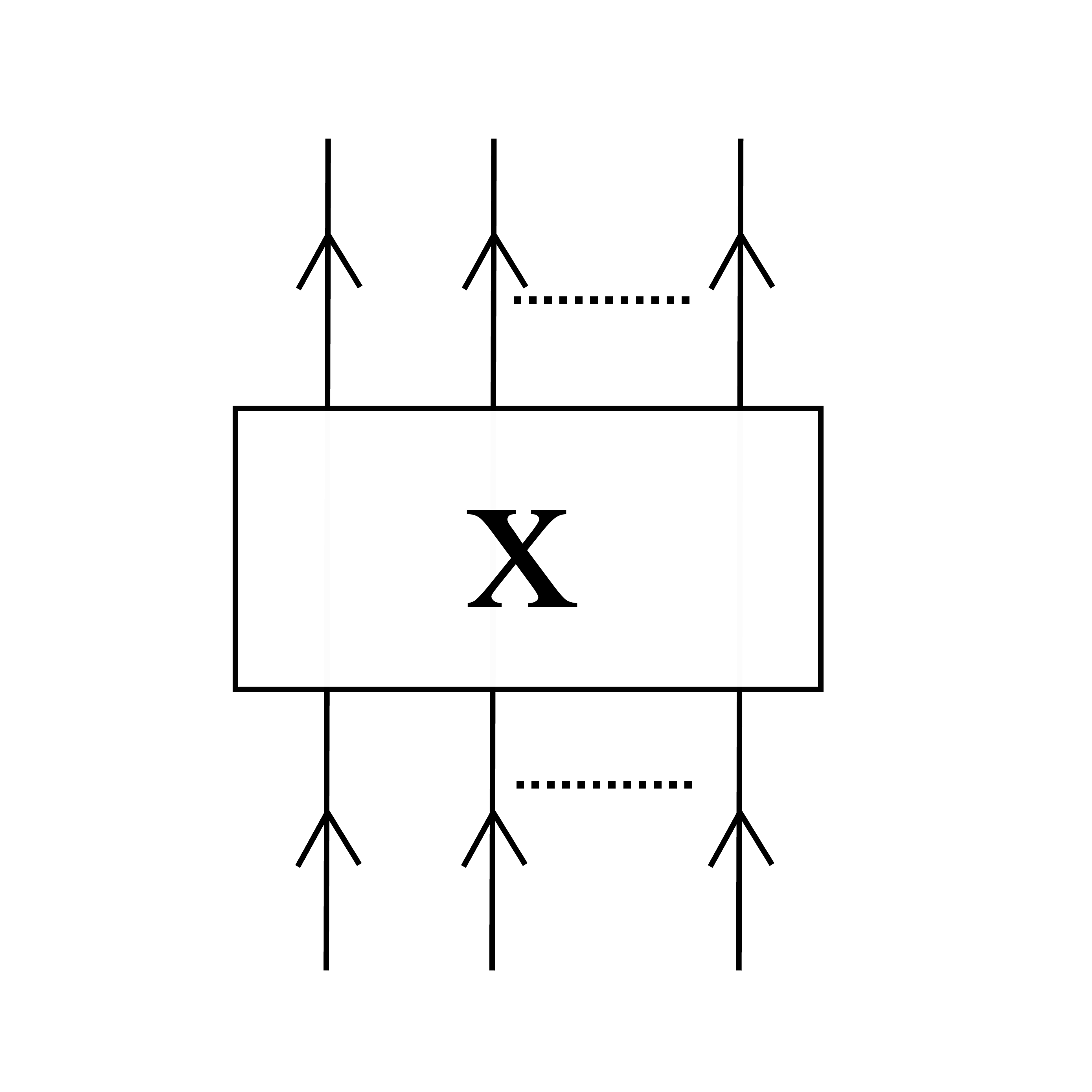}\right)=\grb{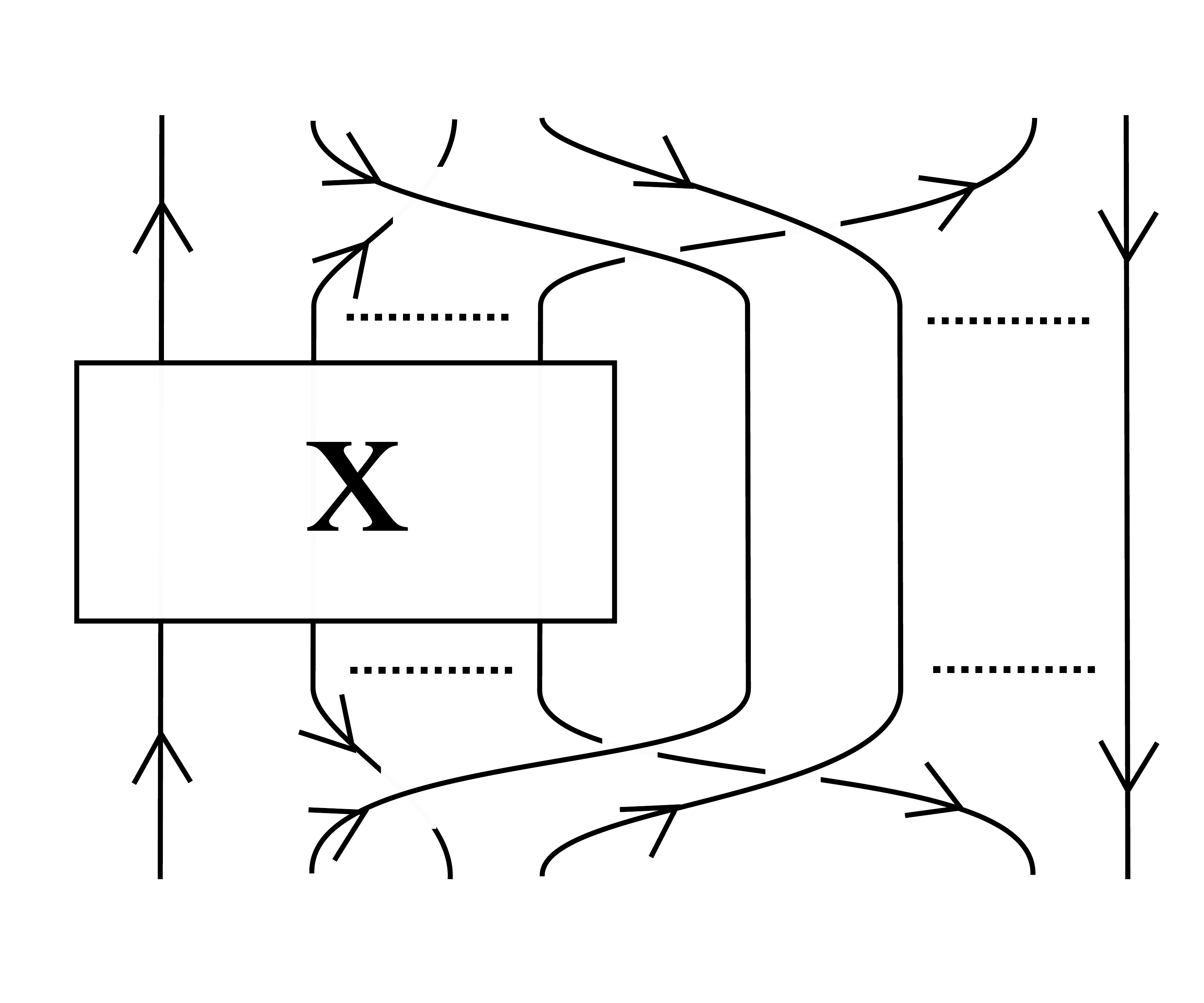}.$$

\begin{proposition}\label{embed}
	The map	$\phi_n$ is an algebra homomorphism from $H_n(q,r)$ to $\mathscr{P}_{2n}^H(q,r)$ preserving the normalized Markov trace. Furthermore, $\phi_{n+1}\vert_{H_{n}(q,r)}=\phi_n$.
\end{proposition}
\begin{proof}
	This follows from the HOMFLY-PT skein relations.
\end{proof}
\begin{notation}\label{phi}
	We define $\phi$ on $H_\bullet(q,r)$ as $\displaystyle \lim_{n\to\infty}\phi_n$.
\end{notation}

Recall that when $r=q^N$ for some $N\in\mathbb{N}$ and $q=e^{\frac{i\pi}{N+l}}$ for some $l\in\mathbb{N}$ or $q\geq1$, $\mathscr{P}_\bullet^H(q,r)$ admits an involution $*$ such that the Markov trace is semipositive-definite. We show that they are the only values of $(q,r)$ and involution such that positivity holds.

\begin{theorem}\label{Thm:positivity}
	The planar algebra $\mathscr{P}_\bullet^H(q,r)$ has positivity if and only if $r=q^N$ for some $N\in\mathbb{N}$, and $q=e^{\frac{i\pi}{N+l}}$ for some $l\in\mathbb{N}$ or $q\geq1$.
\end{theorem}
\begin{proof}	
	If $\mathscr{P}_\bullet^H(q,r)$ has positivity, then the idempotents $P$ and $Q$ in Lemma \ref{lemma:idem} are projections, therefore the generator $S$ is self-adjoint. Since $\mathscr{P}_\bullet^H(q,r)$ is generated by $S$, the involution $*$ is uniquely determined.

	In Theorem \ref{Thm:main1} and \ref{Thm:main2}, we show that a singly generated Thurston-relation planar algebra with parameters $(\delta,\gamma)$ and $\delta>2$ is isomorphic to $\mathscr{P}_\bullet^H(q,r)$ for some $(q,r)$. By Lemma \ref{lemsym}, we can assume that $\Re{q}\geq0$, $\Im{q}\geq0$. Note that $\delta>2$ and $\gamma>0$, then by Equation \eqref{Equ:qsolutions}, we have
	\begin{equation}
	q+q^{-1}=\sqrt{2(\delta^2-2)/(\delta^2-(\frac{\delta^2-2}{\delta}\frac{\gamma-1}{\gamma+1})^2)}.
	\end{equation}
	(One can check the term in the square root is positive.)
	We have $q=e^{i\theta}$ with $0\leq\theta\leq\pi/2$ or $q\geq1$.
	
	Let $[n]$ denote the Young diagram with $1$ row and $n$ columns and $[1^n]$ denote the Young diagram with $n$ rows and $1$ column.
	
	Case 1: $q>1$. By Lemma \ref{lemsym}, we can assume that $\Re{r}\geq0$. By $\displaystyle\frac{r-r^{-1}}{q-q^{-1}}=\delta>2$, we have that $r>1$.
	If $r=q^N$, for some $N\in\mathbb{N}$, then we know that $\mathscr{P}_\bullet^H(q,r)$ has positivity. Otherwise, $q^{N}<r<q^{N+1}$. Then the idempotent $m_{[N+2]}$ is well-defined and $Tr(m_{[N+2]})<0$. By Proposition \ref{embed}, $\mathscr{P}_\bullet^H(q,r)$ does not have positivity.
	
	Case 2: $q=e^{i\theta}$, and $q\neq1$: By Lemma \ref{lemsym}, we can assume that $\Im{r}\geq0$. By $\displaystyle\frac{r-r^{-1}}{q-q^{-1}}=\delta>2$, we have $r=e^{i\alpha}$, for some $\theta<\alpha<\pi-\theta$.
	
	Subcase 1: If $N\theta<\alpha<(N+1)\theta$, for some $N\in \mathbb{N}$, then
	the idempotent $m_{[N+2]}$ is well-defined and $Tr(m_{[N+2]})<0$. By Proposition \ref{embed}, $\mathscr{P}_\bullet^H(q,r)$ does not have positivity.
	
	Subcase 2: If $\alpha=N\theta$ and $\frac{\pi}{N+l+1}<\theta<\frac{\pi}{N+l}$, for some $N,l \in \mathbb{N}$, then the idempotent $m_{[1^{l+1}]}$ is well-defined and $Tr(m_{1^[l+1]})<0$.  By Proposition \ref{embed}, $\mathscr{P}_\bullet^H(q,r)$ does not have positivity.
	
	Subcase 3: If $\alpha=N\theta$ and $\theta=\frac{\pi}{N+l}$, for some $N,l \in \mathbb{N}$, then we know that $\mathscr{P}_\bullet^H(q,r)$ has positivity.
	
	Case 3: $q=1$: By $\displaystyle r-r^{-1}=\delta(q-q^{-1})$, we have $r=1$. By a similar argument in Case 1, one can show that $\delta=N$, for some $N\in \mathbb{N}$. In this case, we know that $\mathscr{P}_\bullet^H(q,r)$ has positivity.
	
\end{proof}

Therefore we obtain our classification result, Theorem \ref{Thm:main}. The ones with positivity all come from representations of the quantum groups $U_q(SU(N))$ or $E_6$. Note that $E_6^{(1)}$ really comes from $SU(3)_3.$


\begin{thebibliography}{}
	
	\bibitem[AM98]{AisMor98}
	AK~Aiston and HR~Morton, \emph{Idempotents of hecke algebras of type a},
	Journal of Knot Theory and Its Ramifications \textbf{7} (1998), no.~04,
	463--487.
	
	\bibitem[Ati88]{Ati88}
	M.~F. Atiyah, \emph{Topological quantum field theory}, Publications
	Math{\'e}matiques de l'IH{\'E}S \textbf{68} (1988), 175--186.
	
	\bibitem[Bis94]{Bis94}
	D.~Bisch, \emph{A note on intermediate subfactors}, Pacific J. Math. \textbf{163} (1994), 201-206.
	
	\bibitem[BJ00]{BisJon00}
	D.~Bisch and V.~F.~R. Jones, \emph{Singly generated planar algebras of small
		dimension}, Duke Math. J. \textbf{101(1)} (2000), 41--76.
	
	\bibitem[BJ03]{BisJon03}
	\bysame, \emph{Singly generated planar algebras of small dimension, part {II}},
	Adv. Math. \textbf{175} (2003), 297--318.
	
	\bibitem[BJL]{BJL}
	D.~Bisch, V.~F.~R. Jones, and Z.~Liu, \emph{Singly generated planar algebras of
		small dimension, part {III}},
	{\color{blue}\url{http://arxiv.org/abs/1410.2876}} To appear Trans. AMS.
	
	\bibitem[BMPS12]{BMPS}
	S.~Bigelow, S.~Morrison, E.~Peters, and N.~Snyder, \emph{Constructing the
		extended {H}aagerup planar algebra}, Acta Math. (2012), 29--82.
	
	\bibitem[FYH{\etalchar{+}}85]{Homfly}
	P.~Freyd, D.~Yetter, J.~Hoste, R.~Lickorish, K.~Millett, and A.~Ocneanu,
	\emph{A new polynomial invariant of knots and links}, Bulletin of the AMS
	\textbf{12} (1985), no.~2, 239--246.
	
	\bibitem[Jon83]{Jon83}
	V.~F.~R. Jones, \emph{Index for subfactors}, Invent. Math. \textbf{72} (1983),
	1--25.
	
	\bibitem[Jon87]{Jon87}
	\bysame, \emph{Hecke algebra representations of braid groups and link
		polynomials}, Ann. of Math \textbf{126} (1987), no.~2, 335--388.
	
	\bibitem[Jon98]{JonPA}
	\bysame, \emph{Planar algebras, {I}}, New Zealand J. Math. (1998),
	{\color{blue}\url{http://arxiv.org/abs/math/9909027}}.
	
	\bibitem[Jon01]{JonAnn}
	\bysame, \emph{The annular structure of subfactors}, Essays on geometry and
	related topics, {V}ol. 1, 2, Monogr. Enseign. Math., vol.~38, Enseignement
	Math., Geneva, 2001, pp.~401--463.
	
	\bibitem[Lan02]{Lan02}
	Z.~Landau, \emph{Exchange relation planar algebras}, Geometriae Dedicata
	\textbf{95} (2002), 183--214.
	
	\bibitem[Liu]{LiuYB}
	Z.~Liu, \emph{Yang-baxter relation planar algebras},
	{\color{blue}\url{http://arxiv.org/abs/1507.06030}}.
	
	\bibitem[Liu16]{Liuex}
	\bysame, \emph{Exchange relation planar algebras of small rank}, Trans. AMS
	(2016), {\color{blue}DOI: \url{http://dx.doi.org/10.1090/tran/6582}}.
	
	\bibitem[MPS]{MPS15}
	Scott Morrison, Emily Peters, and Noah Snyder, \emph{Categories generated by a
		trivalent vertex}, {\color{blue}\url{https://arxiv.org/abs/1501.06869}}.
	
	\bibitem[Ocn88]{Ocn88}
	A.~Ocneanu, \emph{Quantized groups, string algebras and {G}alois theory for
		algebras}, Operator algebras and applications, Vol.\ 2, London Math. Soc.
	Lecture Note Ser., vol. 136, Cambridge Univ. Press, Cambridge, 1988,
	pp.~119--172.
	
	\bibitem[Pop94]{Pop94}
	S.~Popa, \emph{Classification of amenable subfactors of type {II}}, Acta Math.
	\textbf{172} (1994), 352--445.
	
	\bibitem[Pop95]{Pop95}
	\bysame, \emph{An axiomatization of the lattice of higher relative commutants},
	Invent. Math. \textbf{120} (1995), 237--252.
	
	\bibitem[PT88]{PT88}
	JH~Przytycki and P.~Traczyk, \emph{Invariants of links of conway type}, Kobe
	Journal of Mathematics \textbf{4} (1988), no.~2, 115--139.
	
	\bibitem[Res87]{Res87}
	N~Yu Reshetikhin, \emph{Quantized universal enveloping algebras, the
		yang-baxter equation and invariants of links, 1}, Tech. report, 1987.
	
	\bibitem[Thu]{Thu04}
	DP~Thurston, \emph{From dominoes to hexagons},
	{\color{blue}\url{http://arxiv.org/abs/math/0405482}}.
	
	\bibitem[Wen88]{Wen88}
	H.~Wenzl, \emph{On the structure of brauer's centralizer algebras}, Ann. of
	Math. \textbf{128} (1988), 173--193.
	
	\bibitem[Xu98]{Xu98S}
	F.~Xu, \emph{Standard $\lambda$-lattices from quantum groups}, Invent. Math.
	\textbf{134} (1998), no.~3, 455--487.
	


\end{thebibliography}
\end{document}